\documentclass[a4paper,reqno,12pt]{amsart}
\usepackage{amsmath}
  \usepackage{paralist}
  \usepackage{graphics} 
  \usepackage{epsfig} 
\usepackage{graphicx}  \usepackage{epstopdf}
 \usepackage[colorlinks=true]{hyperref}
\hypersetup{urlcolor=blue, citecolor=red}

\newtheorem{theorem}{Theorem}[section]
\newtheorem{corollary}{Corollary}

\newtheorem{lemma}[theorem]{Lemma}

\newtheorem{example}{Example}

\newtheorem{construction}{Construction}

\theoremstyle{definition}
\newtheorem{definition}[theorem]{Definition}
\newtheorem{remark}{Remark}

\newcommand{\gaussm}[3]{\genfrac{[}{]}{0pt}{}{#1}{#2}_{#3}}
\newcommand{\PG}[2]{\operatorname{PG}(#1,#2)}
\newcommand{\F}[2]{\mathbb{F}_{#2}^{#1}}

\newcommand{\aspace}{\PG{v-1}{q}}
\newcommand{\mat}[1]{\mathbf{#1}}
\newcommand{\vek}[1]{\mathbf{#1}}
\newcommand{\frobenius}{\mathrm{F}}

\title[$q^r$-divisible codes] 
      {On projective $\mathbf{q^r}$-divisible codes}

\author[D.~Heinlein, Th.~Honold, M.~Kiermaier, S.~Kurz, and A.~Wassermann]{}


 \email{daniel.heinlein@aalto.fi}
 \email{honold@zju.edu.cn}
 \email{michael.kiermaier@uni-bayreuth.de}
 \email{sascha.kurz@uni-bayreuth.de}
 \email{alfred.wassermann@uni-bayreuth.de}

\thanks{The work of the authors from Bayreuth was supported by the ICT COST Action IC1104
and grant KU 2430/3-1 -- Integer Linear Programming Models for Subspace Codes and Finite Geometry
from the German Research Foundation. Th.~Honold gratefully acknowledges financial
  support for a short visit to the University of Bayreuth in July~2016, where part of this research 
  was done. His work was also supported by the National Natural Science Foundation of China under
  Grant 61571006.}

\begin{document}
\maketitle

\centerline{\scshape Daniel Heinlein}
\medskip
{\footnotesize
 \centerline{Department of Communications and Networking}
   \centerline{Aalto University}
    \centerline{F-76 Aalto, Finland}
} 

\medskip

\centerline{\scshape Michael Kiermaier, Sascha Kurz, and Alfred Wassermann }
\medskip
{\footnotesize
 \centerline{Mathematisches Institut}
   \centerline{Universit\"at Bayreuth}
    \centerline{D-95440 Bayreuth, Germany}
} 

\medskip

\centerline{\scshape Thomas Honold}
\medskip
{\footnotesize
 \centerline{Department of Information and Electronic Engineering}
   \centerline{Zhejiang University}
    \centerline{38 Zheda Road, Hangzhou, Zhejiang 310027, China}
}

\bigskip

\begin{abstract}
A projective linear code over $\mathbb{F}_q$ is called $\Delta$-divisible if all weights of its codewords are divisible by $\Delta$. 
Especially, $q^r$-divisible projective linear codes, where $r$ is some integer, arise in many applications of collections of 
subspaces in $\mathbb{F}_q^v$. One example are upper bounds on the cardinality of partial spreads. Here we survey the known results 
on the possible lengths of projective $q^r$-divisible linear codes. 
\end{abstract}

\section{Introduction}
A \emph{linear $[n,d,k]_q$ code} is a $k$-dimensional subspace of $\F{n}{q}$ whose elements, called 
\emph{codewords}, have minimum Hamming distance $d$. For an introduction into coding theory we refer the 
reader to, e.g., \cite{huffman2010fundamentals}. Linear codes arise in many (combinatorial) contexts 
and were studied intensively since the seminal work of Shannon \cite{shannon1948mathematical}. The number of non-zero entries of a codeword 
is called its \emph{weight}. The \emph{weight distribution}, i.e., the number of occurrences of possible 
(non-zero) weight values, is an important invariant of a (linear) code. One important area of research studies codes 
where the minimum weight $d$ is a large as possible given the parameters $n$, $k$, and $q$. Linear codes 
with just a few different weights were also studied in the literature, where an important case is that of 
\emph{two-weight codes}. In \cite{ward1981divisible} Harold N.~Ward introduced \emph{$\Delta$-divisible} codes, 
where each weight has to be divisible by some given integer $\Delta$. For a survey and further related 
literature, see, e.g., \cite{ward2001divisible_survey} and \cite{liu2006divisible,ward1998quadratic,ward1999introduction}.
The divisibility constant $\Delta$ is relatively constrained, see e.g.~\cite[Theorem 1]{ward1981divisible}: If 
$\gcd(\Delta,q)=1$, then a $\Delta$-divisible code over $\mathbb{F}_q$ (with no 0 coordinates) is a
$\Delta$-fold replicated code. Another example is the Gleason--Pierce--Ward Theorem characterizing $\Delta$-divisible 
codes of dimension $k=n/2$. Here, we study the special case of $q^r$-divisible codes, which was called the \emph{radical case} 
in \cite{ward2001divisible_survey}. Even more, we assume that the codes 
\begin{itemize}
  \item[(1)] have full length, i.e., the generator matrix does not contain a zero column;
  \item[(2)] are \emph{projective}, i.e., the generator matrix does not contain two columns which a scalar
   multiples of each other.
\end{itemize}
In other words, we assume that the minimum distance of the corresponding \emph{dual code}, see e.g.~\cite{huffman2010fundamentals}, 
is at least three. The columns of a generator matrix of a projective linear code can be interpreted as a set of points in $\PG{v-1}{q}$, 
where $v\ge k$, see e.g.~\cite{dodunekov1998codes}.

As an application of projective $q^r$-divisible codes we mention partial $t$-spreads, i.e.~a collections of $t$-dimensional elements 
of $\PG{v-1}{q}$ with trivial intersection. More relations with other combinatorial objects are summarized in Section~\ref{sec_relations}. If $\mathcal{P}$ 
is a partial $t$-spread, then the set of uncovered points, i.e.~the \emph{holes}, form a $q^{t-1}$-divisible set, i.e.~the corresponding code 
is $q^{t-1}$-divisible, see Lemma~\ref{lemma_connection vsp}. The size $|\mathcal{P}|$ of a partial spread is large if the corresponding number $n$ of 
holes is small. The existence of a $q^{t-1}$-divisible set of cardinality $n$ is a necessary condition for the existence of a partial 
$t$-spread with corresponding size. And indeed, all known upper bounds for the size of a partial spread can be deduced from non-existence 
results for projective $q^{t-1}$-divisible codes. 

The main target of this paper is to survey the known existence and non-existence results for possible lengths of projective 
$q^r$-divisible codes. Here we focus on the parameters $q$ and $r$, while refinements to the possible dimensions
$v$ are of course possible and useful in some applications.  

Besides constructions most of the presented results are based on the \emph{linear programming method} for linear codes. Since the constructions often have a nice geometric 
interpretation we will use the language of $q^r$-divisible sets in the following. However, results easily translate to projective 
$q^r$-divisible codes, see Subsection~\ref{subsec_delta_divisible}.

The remaining part of the paper is structured as follows. In Section~\ref{sec_preliminaries} we provide the notation and the 
tools to study our research question. The crucial objects, i.e.~$\Delta$-divisible sets and codes, are defined in 
Subsection~\ref{subsec_delta_divisible} followed by examples of $q^r$-divisible sets in Subsection~\ref{subsec_examples_q_r_divisible}. 
The general linear programming method for linear codes is outlined in Subsection~\ref{subsec_linear_programming_method}. If 
we drop the assumption that the codes have to be projective, then the classification problem of the possible lengths is completely solved, 
see Subsection~\ref{subsec_characterization_q_r_multisets}. Relations to combinatorial objects are summarized in Section~\ref{sec_relations} 
followed by more constructions of $q^r$-divisible sets in Section~\ref{sec_constructions}. In Section~\ref{sec_tools_to_exclude} 
we draw some analytical conclusions from the linear programming method. A main result is Theorem~\ref{thm_exclusion_q_r} in 
Subsection~\ref{subsec_length_classification_small} stating that for $n\le rq^{r+1}$ all possible cardinalities of $q^r$-divisible sets 
can be written as $n=a\cdot \frac{q^{r+1}-1}{q-1}+b\cdot q^{r+1}$ for some integers $a,b\in\mathbb{N}_0$. Concrete numerical results 
are presented in Section~\ref{sec_exclusion}. A conclusion and some lines of future research are given in Section~\ref{sec_conclusion}. 
In an appendix we list examples of $q^r$-divisible sets found by computer searches, see Section~\ref{sec_computer_search_hole_configurations}, 
and excluded intervals of cardinalities of $q^r$-divisible sets based on the linear programming method, Section~\ref{app_exclusion_lists}.

   
\section{Preliminaries}
\label{sec_preliminaries}
In the introduction we have only briefly mentioned the crucial objects. Here we give the precise definitions. As already 
mentioned, there is a correspondence between projective codes and sets in the projective space, see e.g.~\cite{dodunekov1998codes}. 
We will mainly use the latter language and introduce $\Delta$-divisible sets and codes in Subsection~\ref{subsec_delta_divisible}. 
Besides the application point of view, a reason for studying the special case $\Delta=q^r$ is given by the nice inherent 
mathematical properties. As an example, Lemma~\ref{lemma_heritable} states that the $(r-q)$-dimensional subspaces of a $q^r$-divisible 
set are $q^{r-j}$-divisible for $0\le j\le r$. In general the existence of $\Delta$-divisible sets is rather \textit{unlikely} if 
$\Delta$ has a large factor coprime to $q$, see for example Corollary~\ref{cor_nonexistence_arithmetic_progression}.(b). 
In Subsection~\ref{subsec_examples_q_r_divisible} we give some examples and constructions for $q^r$-divisible sets.
The general linear programming method based on the MacWilliams identities can be used to excluded specific cardinalities of 
$\Delta-$ or $q^r$-divisible subsets of $\aspace$ and is described in Subsection~\ref{subsec_linear_programming_method}. In 
Subsection~\ref{subsec_characterization_q_r_multisets} we review the classification of the possible lengths of $q^r$-divisible multisets 
of points in $\aspace$. 


\subsection{$\Delta$-divisible sets and codes}
\label{subsec_delta_divisible}

Let $\PG{v-1}{q}$ denote the \emph{projective space} of dimension $v-1$ over $\mathbb{F}_q$, i.e., the set of $1$-dimensional 
subspaces, called \emph{points}, of $\F{v}{q}$. The $(v-1)$-dimensional subspaces of $\mathbb{F}_q^v$ are called \emph{hyperplanes}. 
The number of $k$-dimensional subspaces of $\mathbb{F}_q^v$ is given by $\gaussm{v}{k}{q}=\prod_{i=0}^{k-1} \frac{q^{v-i}-1}{q^{k-i}-1}$. 
Especially, we have $\gaussm{v}{1}{q}=\gaussm{v}{v-1}{q}=\tfrac{q^v-1}{q-1}$ for the number of points and hyperplanes in $\aspace$.

\begin{definition}
  \label{def_delta_divisible}
  A subset $\mathcal{C}\subseteq \aspace$ is called \emph{$\Delta$-divisible} for an integer $\Delta\ge 1$ if there exists 
  an integer $u$ with $|\mathcal{C}\cap H|\equiv u\pmod{\Delta}$ for each hyperplane $H$. If $\mathcal{C}=\emptyset$ 
  or $v=1$ we call it \emph{trivial}.
\end{definition}

The empty set is $\Delta$-divisible for all $\Delta\ge 1$. If the dimension $v$ of the ambient space $\F{v}{q}$ is $1$, then no hyperplane 
exist at all. All subsets $\mathcal{C}\subseteq \aspace$ are $1$-divisible. Definition~\ref{def_delta_divisible} directly translates 
to multisets of points instead of set of points. 

Taking the elements of $\mathcal{C}\subseteq \aspace$, i.e., one representing vector, as columns of a generator matrix, we obtain 
a \emph{linear $[n,k]$-code} over $\mathbb{F}_q$, where $n=|\mathcal{C}|$ and $\dim(\langle\mathcal{C}\rangle)=k\le v$. The fact that 
$0\notin\mathcal{C}$ translates to the property that the \emph{length} $n$ of the corresponding linear code equals its \emph{effective 
length}, i.e., it contains no zero-column and is of \emph{full length}. An alternative characterization is that the minimum distance 
of the dual code is at least $2$. Starting from a linear $[n,k]$-code over $\mathbb{F}_q$ without zero columns, taking the column vectors of a 
generator matrix gives a multiset of points in $\PG{k-1}{q}$, which is also called \emph{projective system}. In the special case of a set 
$\mathcal{C}$ the projective system is called \emph{simple} and the resulting linear code is called \emph{projective}, i.e., no column of 
the generator matrix is a scalar multiple of another column, or, in other words, the projective system for every generator matrix of the 
code is simple. An alternative characterization is that the minimum distance of the dual code is at least $3$  

The existence of an (arbitrary) integer $u$ with $|\mathcal{C}\cap H|\equiv u\pmod{\Delta}$ for every hyperplane $H$, translates 
to the divisibility of all \emph{weights}, i.e., the number of non-zero entries, of the codewords in the corresponding linear 
code are divisible by $\Delta$. Those so-called $\Delta$-divisible codes were introduced by Harold N.~Ward \cite{ward1981divisible}.
As mentioned in the introduction, the cases where $\Delta$ is coprime to the characteristic of the underlying field are very restricted. 
In the following we will, almost always, restrict ourselves on the special case $\Delta=q^r$ for some integer $r\ge 1$. For $q=2$ the 
case $r=1$ is known under the name \emph{even codes} and the case $r=2$ under the name \emph{doubly even codes.} So far, there are just 
a few papers on the case $r=3$, called \emph{triple even binary codes}, see e.g.\ \cite{betsumiya2012triply,no59}.

First we observe that the variable $u$ in Definition~\ref{def_delta_divisible} can be computed from the cardinality $n$ of 
$\mathcal{C}$, which is the reason why it does not appear in the corresponding definition for divisible codes. To this end, 
we first state the so-called \emph{standard equations}. (These are a special case of the MacWilliams identities~(\ref{mac_williams_identies}), see 
Subsection~\ref{subsec_linear_programming_method}, invoking the minimum distance of the the dual code of at least three.)
\begin{lemma}
  \label{lemma_standard_equations_q}
  Let $\emptyset\neq\mathcal{C}\subseteq\aspace$ with $|\mathcal{C}|=n$ and 
  $a_i$ be the number of hyperplanes containing exactly $i$ points for $0\le i\le n$. Then, we have 
  \begin{eqnarray}
    \sum_{i=0}^{n}a_i &=& \gaussm{v}{1}{q},\label{eq_ste1}\\
    \sum_{i=1}^{n}ia_i &=& n\cdot \gaussm{v-1}{1}{q},\text{ and}\label{eq_ste2}\\
    \sum_{i=2}^{n}{i\choose 2} a_i &=& {n\choose 2}\cdot \gaussm{v-2}{1}{q}\label{eq_ste3}.
  \end{eqnarray}
\end{lemma}
\begin{proof}
  Double-count the incidences of the tuples $(H)$, $(B_1,H)$, and $(\{B_1,B_2\},H)$, where $H$ is a hyperplane 
  and $B_1\neq B_2$ are points contained in $H$. (If $n=1$, then Equation~(\ref{eq_ste3}) reads $0=0$.)
\end{proof}

\begin{lemma}
  \label{lemma_modulo_constraint_all_points}
  If $\mathcal{C}\subseteq \PG{v}{q}$ is $q^r$-divisible, then $|\mathcal{C}\cap H|\equiv |\mathcal{C}|\pmod{q^{r}}$
  for each hyperplane $H$.      
\end{lemma}
\begin{proof}
  Choose $0\le u\le q-1$ such that $|\mathcal{C}\cap H|\equiv u\pmod{q^{r}}$ for each hyperplane $H$ and set $\Delta=q^r$.
  Equation~(\ref{eq_ste1}) and Equation~(\ref{eq_ste2}) from Lemma~\ref{lemma_standard_equations_q} can be rewritten to
  $$
  (q-1)\cdot \left(\sum_{i\ge 0} a_{u+i\Delta}\right) = q^v-1\quad\text{and}\quad
  (q-1)\cdot \left(\sum_{i\ge 0} (u+i\Delta)\cdot a_{u+i\Delta}\right) = n\cdot\left(q^{v-1}-1\right).
  $$
  The second equation minus $u$ times the first equation gives
  $$
    (q-1)\cdot \left(\sum_{i\ge 0} ia_{u+i\Delta}\right) = (n-uq) q^{v-1-r}-\frac{n-u}{q^r},  
  $$
  so that $(n-u)/q^r$ is an integer.
\end{proof}

\begin{lemma}
  \label{lemma_heritable}
  If $\mathcal{C}\subseteq \PG{v}{q}$ is $q^r$-divisible, then each $(r-j)$-dimensional subspace is $q^{r-j}$-divisible for 
  all $0\le j\le r$. 
\end{lemma}
\begin{proof}
  It suffices to consider $j=1$, where $m\ge 1$. Choose $u$ such that $|\mathcal{C}|\equiv|\mathcal{C}\cap H|\equiv u\pmod{q^{r}}$ 
  for each hyperplane $H$. Let $S$ be a subspace of co-dimension $2$ with $|\mathcal{C}\cap S|\equiv\overline{u}\pmod{q^{r-1}}$. 
  Counting the total number of elements of $\mathcal{C}$ via the $q+1$ hyperplanes containing $S$ gives 
  $(q+1)u-q\overline{u}\equiv u\pmod{q^{r}}$, i.e., $u\equiv \overline{u}\pmod{q^{r-1}}$. 
\end{proof}

The term hyperplane of $\mathcal{C}$ translates to the \emph{residual code} of the corresponding linear code. So, we have 
weakened modulo conditions on the number of elements of $\mathcal{C}$ in subspaces or weakened divisibility conditions for 
the weights of the iterative residual codes. Next, we give some examples of $q^r$-divisible 
sets $\mathcal{C}\subseteq\aspace$.


\subsection{Examples of $\mathbf{q^r}$-divisible subsets of $\mathbf{\aspace}$ and a Frobenius type number}
\label{subsec_examples_q_r_divisible}

The family of $q$-ary $\left[\frac{q^k-1}{q-1},k,q^{k-1}\right]$ simplex codes (dual Hamming
codes) is $q^{k-1}$-divisible since the non-zero codewords have constant weight $\Delta=q^{k-1}$.

\begin{example}
  \label{example_r_flat}
  If $\mathcal{C}\subseteq\PG{v}{q}$ is the union of all points of an $r+1$-dimensional subspace, i.e., an $(r+1)$-flat, 
  where $r\ge 1$, then $n=\gaussm{r+1}{1}{q}$, $\dim(\mathcal{C})=r+1$, and $\mathcal{C}$ is $q^r$-divisible.
\end{example}   

The family of $\left[q^{r+1},r+2,q^r\right]$ first-order (generalized) Reed-Muller codes is $q^r$-divisible since the weights of 
the non-zero codewords are either $q^r$ or $q^{r+1}$.

\begin{example}
  \label{example_affine_r_space}
  If $\mathcal{C}\subseteq\PG{v}{q}$ is the complement of an $r+1$-dimensional subspace in an $r+2$-dimensional subspace, 
  where $r\ge 1$, then $n=q^{r+1}$, $\dim(\mathcal{C})=r+2$, and $\mathcal{C}$ is $q^r$-divisible. 
\end{example}   

For the combination of two $q^r$-divisible sets it is advantageous to formulate construction via the corresponding codes. To this end, 
let $C_i$ ($i=1,2$) be linear $\left[n_i,k_i\right]$ codes over $\mathbb{F}_q$ with generating matrices $\mat{G}_i$ (in the broader sense), chosen as follows:
$\mat{G}_1$ and $\mat{G}_2$ have the same number $k$ of rows, and their left kernels intersect only in $\{\vek{0}\}$. Then
$\mat{G}=(\mat{G}_1|\mat{G}_2)$ generates a linear $[n_1+n_2,k]$ code $C$, called a \emph{juxtaposition} of $C_1$ and $C_2$. It is clear that $C$
is $q^r$-divisible if $C_1$ and $C_2$ are. If $C_1$ and $C_2$ are projective, we can force $C$ to be projective as well by choosing
$\mat{G}_i$ appropriately, e.g., $\mat{G}_1=\left(
  \begin{smallmatrix}
    \mat{G}_1'\\\mat{0}
  \end{smallmatrix}\right)$, $\mat{G}_2=\left(
  \begin{smallmatrix}
    \mat{0}\\\mat{G}_2'
  \end{smallmatrix}\right)$, in which case $C$ is just the direct sum $C_1\oplus C_2$ of $C_1$ and $C_2$. For the ease of notation we 
also use the direct sum for sets of points in $\aspace$:  

\begin{lemma}
  \label{lemma_add_configurations}
  Let $\mathcal{C}_1\subseteq\PG{v_1-1}{q}$ and $\mathcal{C}_2\subseteq\PG{v_2-1}{q}$ be $q^r$-divisible, then 
  $\mathcal{C}=\mathcal{C}_1\oplus\mathcal{C}_2\subseteq\PG{v_1+v_2-1}{q}$ is $q^r$-divisible of cardinality 
  $\left|\mathcal{C}_1\right|+\left|\mathcal{C}_2\right|$ and we have $\dim(\langle\mathcal{C}_1\oplus\mathcal{C}_2\rangle)=\dim(\langle\mathcal{C}_1\rangle)+\dim(\langle\mathcal{C}_2\rangle)$. 
\end{lemma}

Using Lemma~\ref{lemma_add_configurations} we can show that there is a well-defined function $\frobenius(q,r)$,
assigning to $q,r$ the largest integer $n$ that is not equal to the cardinality of a $q^r$-divisible set in $\aspace$, or is not 
equal to the length of a projective $q^r$-divisible linear code over $\mathbb{F}_q$. The problem of determining $\frobenius(q,r)$ is in 
some sense analogous to the well-known \emph{Frobenius Coin Problem}, see, e.g., \cite{brauer1942problem}). In its simplest form it 
asks for the largest integer not representable as $a_1n_1+a_2n_2$ with $a_1,a_2\geq 0$, where $n_1$ and $n_2$ are given relatively prime
positive integers. The solution is $(n_1-1)(n_2-1)-1$, as is easily shown. With this, Lemma~\ref{lemma_add_configurations}, 
Example~\ref{example_r_flat}, and Example~\ref{example_affine_r_space} imply:
\begin{lemma}
\label{lemma_frobenius}
For each integer $r\ge 1$ we have
\begin{equation}
  \label{eq:frob}
  \begin{aligned}
    \frobenius(q,r)&\leq\frac{q^{r+1}-1}{q-1}\cdot q^{r+1}
    -\frac{q^{r+1}-1}{q-1}-q^{r+1}\\
    &=q^{2r+1}+q^{2r}+\dots+q^{r+2}-q^r-q^{r-1}-\dots-1.
  \end{aligned}
\end{equation}
\end{lemma}
So, indeed $\frobenius(q,r)$ is well-defined and we determine the first few exact values in Section~\ref{sec_exclusion}.


\subsection{MacWilliams identities and the linear programming method}
\label{subsec_linear_programming_method}

We remark that the standard equations from Lemma~\ref{lemma_standard_equations_q} have a natural generalization in the language 
of linear codes, see e.g.~\cite{huffman2010fundamentals}. To this end let $\mathcal{C}\subset \aspace$ and $\mathcal{L}$ denote 
the corresponding linear projective code 
over $\mathbb{F}_q$. Let $n$ denote the cardinality of $\mathcal{C}$, i.e., the length of $\mathcal{L}$, and $k$ denote 
the dimension of $\langle\mathcal{C}\rangle$ and $\mathcal{L}$. By $a_i$ we denote the number of hyperplanes having an intersection 
of cardinality $i$ with $\mathcal{C}$, by $A_i$ the number of codewords of $\mathcal{L}$ of weight $i$ and by $A_i^\perp$ the number 
of codewords of weight $i$ of the dual code $\mathcal{L}^\perp$, which has dimension $n-k$. The connection between the $a_i$ and $A_i$ 
is given by $A_i=(q-1)a_{n-i}$ for all $0<i\le n$, $A_0=1$, and $a_n=0$. The \emph{weight distribution} $(A_0^\perp,\dots,A_n^\perp)$ of the 
dual code $\mathcal{L}^\perp$ can be computed from the weight distribution $(A_0,\dots,A_n)$ of the (primal) code $\mathcal{L}$. 
One way\footnote{Another way uses the \emph{weight enumerator} $W_{\mathcal{L}}(x,y)=\sum_{i=0}^n A_i(\mathcal{L})y^ix^{n-i}$. With 
this, the weight enumerator of the dual code is given by $W_{\mathcal{C}^\perp}(x,y)=\frac{1}{|\mathcal{C}|}\cdot W_{\mathcal{C}}(x+(q-1)y,x-y)$. }
to write down the underlying relation are the so-called \emph{MacWilliams identities}:
\begin{equation}
  \label{mac_williams_identies}
  \sum_{j=0}^{n-\nu} {{n-j}\choose\nu} A_j=q^{k-\nu}\cdot \sum_{j=0}^\nu {{n-j}\choose{n-\nu}}A_j^\perp\quad\text{for }0\le\nu\le n, 
\end{equation} 
where, additionally, $A_0^\perp=1$. The fact that the $A_i^\perp$ are uniquely determined by the $A_i$ can e.g.\ be seen 
by providing explicit equations for each $A_i^\perp$ in dependence of the $A_j$. Those formulas involve the so-called 
\emph{Krawtchouk polynomials} \cite{krawtchouk1929generalisation}.    

We remark that we have $A_1^\perp=A_2^\perp=0$, since the minimum distance of the dual code is 
at least $3$, in our situation. With this, the first three equations of (\ref{mac_williams_identies}) are equivalent to the equations 
from Lemma~\ref{lemma_standard_equations_q}. 

Of course the $A_i$ and the $A_i^\perp$ in (\ref{mac_williams_identies}) have to be non-negative integers, which moreover are divisible by $q-1$. 
Omitting the integrality condition yields the so-called \emph{linear programming method}, see e.g.\ \cite[Section 2.6]{huffman2010fundamentals}, where the $A_i$ and $A_i^\perp$ 
are variables satisfying the mentioned constraints.\footnote{Typically, the $A_i^\perp$ are removed from the formulation using the 
explicit formulas based on the Krawtchouk polynomials, which may of course also be done automatically in the preprocessing step of 
a customary linear programming solver.} Given some further constraints on the weights of the code and/or the dual code, one may check 
whether the corresponding polyhedron is empty or contains non-negative rational solutions. In general, this is a very powerful 
approach\footnote{The work of Delsarte, see e.g.~\cite{delsarte1973algebraic,delsarte1998association}, tells us, that, essentially, 
there is some kind of linear programming bound if we have an underlying association scheme, see \cite{zieschang2006theory} for an introduction. 
Even if there is no underlying association scheme, but a product scheme, similar methods can be used, see e.g.~\cite{brouwer1998bounds} for 
an application in coding theory. Replacing linear programming by semidefinite programming, one might obtain even stronger results, see 
e.g. \cite{bachoc2013bounds} for an application on bounds for projective codes, c.f.\ also \cite{bachoc2008new,barg2013new}.} 
and was e.g.\ be used to compute bounds for codes with a given minimum distance, see \cite{delsarte1972bounds,lloyd1957binary}. Since 
the involved binomial coefficients grow rather quick, one has to take care of numerical issues.\footnote{This is no principal problem, 
see e.g.~\cite{applegate2007exact} or \cite[Section 1.4.1]{espinoza2006linear}.} In this paper, we will mostly 
consider a subset of the MacWilliam identities and use analytical arguments, see Section~\ref{sec_tools_to_exclude}.\footnote{The 
use of special polynomials, like we will do, is well known in the context of the linear programming method, see 
e.g.~\cite[Section 18.1]{bierbrauer2005introduction}.}    

In the case of $q^r$ divisible sets, we have $A_i=0$ if $q^r$ does not divide $i$. The resulting linear program is denoted by $\operatorname{LP}_m$. 
For $m=n+1$ we use the notation $\operatorname{LP}_m=\operatorname{LP}_\infty$. If $A_{iq^r}>0$ then there exists a $q^{r-1}$-divisible set 
$\mathcal{C}$ (contained in the corresponding hyperplane or alternatively the residual code to the corresponding codeword of that weight). If the 
existence of $\mathcal{C}$ is recursively excluded by the linear programming method, then we may introduce the constraint $A_{iq^r}=0$. The resulting 
enhanced linear program is denoted by $\operatorname{LP}_m^\star$. For $m=n+1$ we use the notation $\operatorname{LP}_m^\star=\operatorname{LP}_\infty^\star$. The tightest 
results are of course given by $\operatorname{LP}_\infty^\star$, which may be tighter than applying the linear programming method 
directly see Remark~\ref{remark_ILP_LP_RECURSIVE}. Since the dimension $k$ is a parameter in these formulations, those LPs principally 
have to be solved for all possible values for $k$. Trivially, we have $1\le k\le n$. Tighter bounds for $\Delta$-divisible 
codes are discussed in the literature, see e.g.~\cite{liu2006divisible,liu2006weights,liu2010binary,liu2011equivalence,ward1992bound} . 
By a relaxation it is possible to compute most of \textit{the information} of the first four MacWilliams identities by a single linear program.

\begin{lemma}
   \label{lemma_reduced_lp_certificate}
   We define the linear program $\widetilde{\operatorname{LP}}$ as
  \begin{eqnarray*}
  \min z\\
  {n\choose\nu}+\sum_{j=1}^{n-\nu} {{n-j}\choose\nu} A_j=y\cdot q^{3-\nu}\cdot {{n}\choose{n-\nu}}&&\text{for }0\le \nu\le 2\\
  {n\choose 3}+\sum_{j=1}^{n-3} {{n-j}\choose 3} A_j=y\cdot  {{n}\choose{n-3}}+x\\  
  x+z\ge 0\\
  y+z\ge 0\\
  A_j+z \ge 0 &&\text{for }1\le j\le n\\
  A_j\in\mathbb{R}&&\text{for }1\le j\le n\\
  x,y\in\mathbb{R}\\
  z\in\mathbb{R}_{\ge 0}
  \end{eqnarray*}
  If $\widetilde{\operatorname{LP}}$ is infeasible or has an optimum target value strictly larger than zero, then the first four MacWilliams 
  identities do not admit a non-negative real solution.
\end{lemma}    
\begin{proof}
  We have used the substitutions $y=q^{k-3}$ and $x=q^{k-3} \cdot A_3^\perp=y\cdot A_3^\perp$.  
\end{proof}
We remark that uncoupling the non-linear dependency between $x$ and $y$ may weaken the strength of the LP formulation. However, 
in almost all cases where $\operatorname{LP}_4$ yields a certificate for non-existence $\widetilde{\operatorname{LP}}$
yields it too. 
By $\widetilde{\operatorname{LP}}^\star$ we denote the resulting linear program if we incorporate the iterative exclusion of specific weights. 

Invoking the integrality of the $A_i$ and the $A_j^\perp$ may result in tighter constraints. Using the lattice point enumeration 
software \texttt{solvediophant},  which is based on the LLL algorithm \cite{lenstra1982factoring}, see \cite{wassermann2002attacking}, we also 
checked whether the MacWilliams identities have non-negative integer solutions for some parameters. We denote the corresponding 
integer linear programming formulation by $\operatorname{ILP}$ and $\operatorname{ILP}^\star$. A case where $\operatorname{ILP}$ 
could exclude a cardinality $n$ that can not be excluded by $\operatorname{LP}_\infty$ is described in Remark~\ref{remark_ILP_LP_RECURSIVE}. However, 
$\operatorname{LP}_\infty^\star$ is sufficient for this example. So far we do not any example where $\operatorname{ILP}^\star$ was strictly superior 
to $\operatorname{LP}_\infty^\star$.

What can be done beyond the (integer) linear programming method, i.e., the feasibility of the MacWilliams identities? 
Some generalizations of the MacWilliams identities can be found in \cite{dougherty2002macwilliams,greferath2000finite,
kim2005classification,shiromoto1996new,tang2016macwilliams}. 
Here we emphasize the so-called split weight enumerator, see e.g.~\cite{simonis1995macwilliams}. In the classification 
of optimal linear codes this enumerator and the corresponding linear programming technique was very successful, see 
e.g.~\cite{jaffe1996binary,jaffe1997brief,jaffe2000optimal}. 
For even or doubly-even codes some more restrictions on the weights can e.g.\ be found in \cite[Section 7.8]{huffman2010fundamentals}, 
see also \cite{guritman2001degree,simonis1994restrictions} for further generalizations. 


\subsection{The possible lengths of $q^r$-divisible multisets}
\label{subsec_characterization_q_r_multisets}

As mentioned, Definition~\ref{def_delta_divisible} directly translates to multisets. In order to describe a multiset $\mathcal{M}$ 
of points in $\aspace$ we can use a characteristic function $\chi$ that maps each point $P$ of $\aspace$ to an integer $\chi(P)$, 
which is the number of occurrences of $P$ in $\mathcal{M}$. With this, the cardinality $\#\mathcal{M}=|\mathcal{M}|$ of $\mathcal{M}$ 
is just the sum over $\chi(P)$ for all points $P$ in $\aspace$. For a hyperplane $H$ of $\aspace$ we denote by $|H\cap \mathcal{M}$ 
the cardinality of $\mathcal{M}$ restricted to $H$, i.e., the sum over $\chi(P)$ for all points $P$ in $H$. With this we can state, 
that a multiset $\mathcal{M}$ is $q^r$-divisible if we have $|H\cap\mathcal{M}|\equiv |\mathcal{M}|\pmod{q^r}$ for each hyperplane 
$H$ of $\aspace$. (Of course we can also define $\Delta$-divisible or use the indirect formulation of the existence of an integer $u$ 
with $|H\cap\mathcal{M}|\equiv u\pmod \Delta$ for all hyperplanes $H$.)

The $(r+1)$-flats, i.e., the sets of all points of an $(r+1)$-dimensional subspace in $\mathbb{F}_q^v$, or simplex codes in coding theoretic 
language are of course also valid examples for $q^r$-divisible multisets. If multiset $\mathcal{M}$ is $q^s$-divisible and described by 
a characteristic function $\chi$, then the characteristic function $\chi'$ defined by $\chi'(P)=q^t\cdot \chi(P)$ for every point 
$P$ of $\aspace$ corresponds to a $q^{s+t}$-divisible multiset $\mathcal{M}'$. Thus, for each $0\le i\le r$ there exists a $q^r$-divisible 
multiset of cardinality $\gaussm{r+1-i}{1}{q}\cdot q^i$. Since Lemma~\ref{lemma_add_configurations} is also valid for multisets of points 
instead of sets of points, for each non-negative integers $a_0,\dots a_r$ there exists a $q^r$-divisible multiset of cardinality
\begin{equation}
  \sum_{i=0}^r a_i\cdot\gaussm{r+1-i}{1}{q}\cdot q^i. 
\end{equation}  
Interestingly enough, every integer $n$ that cannot be represented as such a non-negative linear combination of the mentioned base examples 
can indeed not be the cardinality of a $q^r$-divisible multiset, see \cite[Theorem 1]{divisibleIEEE}. Moreover, the question whether such a 
representation exists can be answered by a fast and simple algorithm, see \cite[Algorithm 1]{divisibleIEEE}. So, the question which integers 
$n$ can be the cardinality of a $q^r$-divisible multiset of points in $\aspace$ is completely resolved for every positive integer $r$. The 
corresponding question remain unsolved for fractional values of $r$ and also for the cases of $q^r$-divisible sets of points, where only 
partial results are known, which is actually the topic of this paper.  


\section{Relations to other combinatorial objects}
\label{sec_relations}

In the previous two subsections we have formulated the classification of the possible length of $q^r$-divisible sets in $\aspace$, for some 
dimension $v$, as an interesting research problem and stated some examples of $q^r$-divisible sets. Additionally, we have shown that those 
objects are in direct correspondence to projective $q^r$-divisible linear codes, i.e., we are looking at a special class of $\Delta$-divisible 
codes, where the restrictions on $\Delta$ are not that strong as they seem to be at first sight. Even more, $q^r$-divisible sets or 
$q^r$-divisible multisets in $\aspace$ are related to several other combinatorial objects:
\begin{itemize}
  \item vector space partitions $\rightarrow$ Subsection~\ref{subsec_vector_space_partitions}
  \item partial $t$-spreads $\rightarrow$ Subsection~\ref{subsec_partial_t_spreads}
  \item subspace codes $\rightarrow$ Subsection~\ref{subsec_subspace_codes}
  \item subspace packings and coverings $\rightarrow$ Subsection~\ref{subsec_subspace_packings_coverings}
  \item orthogonal arrays $\rightarrow$ Subsection~\ref{subsec_orthogonal_arrays}
  \item $(s,r,\mu)$-nets $\rightarrow$ Subsection~\ref{subsec_nets}
  \item minihypers $\rightarrow$ Subsection~\ref{subsec_minihypers}
  \item two-weight codes $\rightarrow$ Subsection~\ref{subsec_two_weight_codes}
  \item optimal linear codes $\rightarrow$ Subsection~\ref{subsec_optimal_codes}
  \item $q$-analogs of group divisible designs $\rightarrow$ Subsection~\ref{subsec_qgdd}
  \item codes of nodal surfaces $\rightarrow$ Subsection~\ref{subsec_nodal_surfaces}
\end{itemize}

\subsection{Vector space partitions}
\label{subsec_vector_space_partitions}

A \emph{vector space partition} $\mathcal{P}$ of $\F{v}{q}$ is a collection of 
subspaces with the property that every non-zero vector is contained in a unique member of $\mathcal{P}$. If $\mathcal{P}$  contains 
$m_d$ subspaces of dimension $d$, then $\mathcal{P}$ is of type $k^{m_k}\dots 1^{m_1}$. We may leave out some of the cases with $m_d=0$. 
Subspaces of dimension~$1$ are called \emph{holes}. If there is at least one non-hole, then $\mathcal{P}$ is called non-trivial. 

\begin{lemma}
  \label{lemma_connection vsp}
  Let $\mathcal{P}$ be a vector space partition of type $t^{m_t}\dots s^{m_s}1^{m_1}$ of $\F{v}{q}$, where $n>t\ge s\ge 3$ Then, 
  the $1$-dimensional elements of $\mathcal{P}$ form a  $q^{s-1}$-divisible set.
\end{lemma}
\begin{proof}
  Let $l,x\in\mathbb{N}_0$ with $\sum_{i=s}^t m_i =lq^s+x$ and $H$ be a hyperplane corresponding vector space partition of 
  type $t^{m_t'}\dots (s-1)^{m_{s-1}'}1^{m_1'}$ of $\F{v-1}{q}$. Counting the number of non-zero vectors in $\F{v}{q}$ and $H$ yields
  \begin{eqnarray*}
    \left(lq^s+x\right)\cdot\gaussm{s}{1}{q}+aq^s+m_1=\gaussm{v}{1}{q}
    \quad\text{and}\quad\\
    \left(lq^s+x\right)\cdot\gaussm{s-1}{1}{q}+a'q^{s-1}+m_1'=\gaussm{v-1}{1}{q}
  \end{eqnarray*}
  for some $a,a'\in\mathbb{N}_0$. $q-1$ times the first equation yields $-x+(q-1)m_1\equiv -1 \pmod{q^s}$ 
  and $q-1$ times the second equation yields $-x+(q-1)m_1'\equiv -1 \pmod{q^{s-1}}$, so that 
  $m_1\equiv m_1' \pmod {q^{s-1}}$. 
\end{proof}

For $s=2$ the two-dimensional elements of $\mathcal{P}$ in Lemma~\ref{lemma_connection vsp} might correspond to 
$1$-dimensional elements in $\mathcal{P}'$. If we distinguish those elements from the original $1$-dimensional elements form $\mathcal{P}$ 
lying in $H$, we see that the $1$-dimensional elements of $\mathcal{P}$ also form a $q^1$-divisible set.

\begin{lemma}
  \label{lemma_vsp_as_mod_regular_configuration_via_holes}
  Let $\mathcal{P}$ be a vector space partition of type $t^{m_t}\dots 2^{m_2}1^{m_1}$ of $\F{v}{q}$ for some integers $v>t\ge 2$. 
  Then, the $1$-dimensional elements of $\mathcal{P}$ form a $q^1$-divisible set.
\end{lemma}
\begin{proof}
  We denote the $1$-dimensional elements of $\mathcal{P}$ as \emph{holes} and choose $l,x\in\mathbb{N}_0$ with $\sum_{i=2}^t m_i =lq^2+x$. 
  Let $H$ by a hyperplane with corresponding vector space partition $\mathcal{P}'$ of type $t^{m_t'}\dots 1^{m_1'}$ of $\F{n-1}{q}$. 
  By $\tilde{m}_1'$ we denote the number of $2$-dimensional elements of $\mathcal{P}$ with $1$-dimensional intersection with $H$ and by 
  $\hat{m}_1'$ we denote the number of holes contained in $H$, i.e., $m_1'=\tilde{m}_1'+\hat{m}_1'$. Counting the number of non-zero vectors 
  in $\F{n}{q}$ and $H$ yields
  \begin{eqnarray*}
    \left(lq^2+x\right)\cdot(q+1)+aq^2+m_1=\gaussm{n}{1}{q}
    \quad\text{and}\quad\\
    \left(lq^2+x-\tilde{m}_1'\right)\cdot(q+1)+a'q^2+\tilde{m}_1'+\hat{m}_1'=\gaussm{n-1}{1}{q}
  \end{eqnarray*}
  for some $a,a'\in\mathbb{N}_0$. $q-1$ times the first equation yields $(q-1)m_1\equiv x-1 \pmod{q^2}$ and 
  $q-1$ times the second equation yields $(q-1)\hat{m}_1'\equiv x-1 \pmod{q}$, so that $m_1\equiv \hat{m}_1' \pmod {q}$.
\end{proof}

Thus, our subsequent results on the non-existence of $q^r$-divisible subsets of $\aspace$ will result in improved 
upper bounds on the maximum size $A_q(v,2t;t)$ of a partial $t$-spread of $\aspace$. More generally, we can conclude 
some non-existence results for vector space partitions, whose classification is an ongoing, very hard, major project, see e.g.\ 
\cite{el2009partitions,heden2009length,heden2012survey,heden2013supertail,lehmann2012some,seelinger2012partitions}. To that end, we 
mention the following theorem.

\begin{theorem} (Theorem 1 in \cite{heden2009length})
  \label{thm_length_of_tail}
  Let $\mathcal{P}$ be a vector space partition of type $t^{m_t}\dots {d_2}^{m_{d_2}}{d_1}^{m_{d_1}}$ of $\F{n}{q}$, 
  where $m_{d_2},m_{d_1}>0$ and $n_1=m_{d_1}$, $n_2=m_{d_2}$.
  \begin{enumerate}
    \item[(i)]   if $q^{d_2-d_1}$ does not divide $n_1$ and if $d_2<2d_1$, then $n_1\ge q^{d_1}+1$;
    \item[(ii)]  if $q^{d_2-d_1}$ does not divide $n_1$ and if $d_2\ge 2d_1$, then $n_1>2q^{d_2-d_1}$ or $d_1$ divides $d_2$ and 
                 $n_1=\left(q^{d_2}-1\right)/\left(q^{d_1}-1\right)$;
    \item[(iii)] if $q^{d_2-d_1}$ divides $n_1$ and $d_2<2d_1$, then $n_1\ge q^{d_2}-q^{d_1}+q^{d_2-d_1}$;
    \item[(iv)]  if $q^{d_2-d_1}$ divides $n_1$ and $d_2\ge 2d_1$, then $n_1\ge q^{d_2}$.
  \end{enumerate}   
\end{theorem}

Theorem~\ref{thm_length_of_tail} is slightly improved and reproven in \cite{kurz2018heden} using the language of $q^r$-divisible sets. 
Generalizations of vector space partitions, where the theory of $q^r$-divisible sets or multisets can in principle be used, are studied 
in \cite{el2011lambda,heinlein2019generalized}.

\subsection{Partial $t$-spreads}
\label{subsec_partial_t_spreads}

A \emph{partial $t$-spread} in $\F{v}{q}$ is a collection of $t$-dimensional subspaces such that the non-zero vectors are covered at most 
once, i.e., a vector space partition of type $t^{m_t}1^{m_1}$. By $A_q(v,2t;t)$ we denote the maximum value of $m_t$\footnote{The more 
general notation $A_q(v,2t-2w;t)$ denotes the maximum cardinality of a collection of $t$-dimensional subspaces, 
whose pairwise intersections have a dimension of at most $w$, see Subsection~\ref{subsec_subspace_codes}.}.

For a long time the best upper bound for partial spreads was given by Drake and Freeman:
\begin{theorem} (Corollary 8 in \cite{nets_and_spreads})
  \label{thm_partial_spread_4}
  If $v=kt+r$ with $0<r<t$, then 
  $$
    A_q(v,2t;t)\le \sum_{i=0}^{k-1} q^{it+r} -\left\lfloor\theta\right\rfloor-1
    =q^r\cdot \frac{q^{kt}-1}{q^t-1}-\left\lfloor\theta\right\rfloor-1
    =\frac{q^{v}-q^r}{q^t-1}-\left\lfloor\theta\right\rfloor-1,
  $$
  where $2\theta=\sqrt{1+4q^t(q^t-q^r)}-(2q^t-2q^r+1)$.
\end{theorem}
Their result is based on the work of Bose and Bush \cite{bose1952orthogonal} and uses nets as crucial objects. We 
give a similar statement bases on $q^{t-1}$-divisible sets in Lemma~\ref{lemma_hyperplane_types_arithmetic_progression}. In 
both cases a quadratic polynomial plays an essential role. For the details on the relation between both methodes we refer the 
interested reader to \cite{honold2018partial,kurz2017packing}.

Using the parameterization $v=kt+r$ with $0\le r<t$, the cases $r=0,1$ are completely settled 
for a long time, see e.g.~\cite{beutelspacher1975partial}. The case $q=2$, $r=2$ was completely resolved 
\cite{spreadsk3,kurzspreads}. Then, a major breakthrough was obtained by N{\u{a}}stase and Sissokho:
\begin{theorem}(Theorem 5 in \cite{nastase2016maximum})
   If $v=kt+r$ with $0<r<t$ and $t>\gaussm{r}{1}{q}$, then $A_q(v,2t;t)=\frac{q^n-q^{t+r}}{q^t-1}+1$.
\end{theorem}
The underlying techniques could be used to obtain some improved upper bounds for partial spreads, see 
\cite[Lemma 10]{nastase2016maximum}, \cite[Theorems 6,7]{nastase2016maximumII}, and \cite[Theorems 2.9, 2.10]{kurz2017packing}. 
All these results can be explained in the language of $q^r$-divisible sets, see \cite{honold2018partial}. Indeed, 
all currently known best upper bounds for partial spreads, see e.g.~\cite{TableSubspacecodes}, can be obtained from 
non-existence results for $q^r$-divisible sets.

\subsection{Subspace codes}
\label{subsec_subspace_codes}
For two subspaces $U$ and $U'$ of $\aspace$ the \emph{subspace distance} is given by $d_S(U,U') = \dim (U+U') - \dim (U\cap U')$. 
A set $\mathcal{C}$ of subspaces in $\aspace$, called codewords, with minimum subspace distance $d$ is called a \emph{subspace code}. 
Its maximal possible cardinality is denoted by $A_q(v,d)$, see e.g.\ \cite{honold2016constructions}. If all codewords have the same dimension, 
say $k$, then we speak of a \emph{constant dimension code} and denote the corresponding maximum possible cardinality by $A_q(v,d;k)$, see 
e.g.~\cite{etzionsurvey}. For known bounds, we refer to \url{http://subspacecodes.uni-bayreuth.de} \cite{TableSubspacecodes} containing 
also the generalization to subspace codes of mixed dimension. As mentioned before, for $2k\le v$ the cardinality $A_q(v,2k;k)$ is the maximum 
size of a partial $k$-spread. For $d<2k$ the recursive Johnson bound 
$$A_q(v,d;k)\le \left\lfloor \gaussm{v}{1}{q} \cdot A_q(v-1,d;k-1)/\gaussm{k}{1}{q}\right\rfloor,$$  
see \cite{xia2009johnson}, recurs on this situation. The involved rounding can be slightly sharpened using the non-existence of $q^r$-divisible 
multisets of a certain cardinality, see \cite[Lemma 13]{divisibleIEEE}. For $d<2k$ this gives the tightest known upper bound for $A_q(v,d;k)$ 
except $A_2(6,4;3)=77<81$ \cite{honold2015optimal} and $A_2(8,6;4)=257< 289$ \cite{heinlein2019classifying}. For general subspace codes the 
underlying idea of the Johnson bound in combination with $q^r$-divisible multisets has been generalized in \cite{ubt_eref52236}. 

\subsection{Subspace packings and coverings}
\label{subsec_subspace_packings_coverings}
A constant-dimension code consisting of $k$-di\-men\-sional codewords in $\aspace$ has minimum subspace distance $d$ iff each 
$(k-\tfrac{d}{2}+1)$-dimensional subspace is contained in at most one codeword. If we relax the condition a bit and require that 
each $(k-\tfrac{d}{2}+1)$-dimensional subspace is contained in at most $\lambda$ codewords, then we have the definition of a 
\emph{subspace packing}. Of course, similar to constant-dimension codes, $q^r$-divisible multisets can be used to obtain 
upper bounds on the cardinality of a subspace packing, see \cite{ubt_eref48694,etzion2019subspace}. Indeed, 
\cite[Lemma 13]{divisibleIEEE} covers that case. 

If we replace {\lq\lq}contained in at most $\lambda$ codewords{\rq\rq} by {\lq\lq}contained in at least $\lambda$ codewords{\rq\rq} 
we obtain so-called \emph{subspace packings}. For the special case of $\lambda=1$ we refer e.g.\ to \cite{etzion2014covering,etzion2011q}. 
\cite[Lemma 13]{divisibleIEEE} also covers this situation and relates it to $q^r$-divisible multisets.

\subsection{Orthogonal arrays}
\label{subsec_orthogonal_arrays}

A $t-(v,k,\lambda)$ orthogonal array, where $t\le k$, is a $\lambda vt \times k$ array whose entries are chosen from a set 
$X$ with $v$ points such that in every subset of $t$ columns of the array, every $t$-tuple of points of $X$ appears in exactly $\lambda$ rows. 
Here, $t$ is called the strength of the orthogonal array. For a survey see e.g.\ \cite{hedayat2012orthogonal}. A library of orthogonal arrays 
can be found at \url{http://neilsloane.com/oadir/}. A variant of the linear programming method for orthogonal arrays with mixed levels was 
presented in \cite{sloane1996linear}. Orthogonal arrays can be regarded as natural generalizations of orthogonal Latin squares\cite{keedwell2015latin}, 
c.f.~\cite{bose1952orthogonal}. 

\subsection{$(s,r,\mu)$-nets}
\label{subsec_nets}

\begin{definition}(\cite[Definition 2]{nets_and_spreads})\\
  Let $J$ be an incidence structure. Define $B\parallel G$ for blocks $B$, $G$ of $J$ to mean that either $B = G$ or [B,G] = 0. Then $J$ is called an $(s,r,\mu)$-net provided: 
  \begin{enumerate}
    \item[(i)] $||$ is a parallelism; 
    \item[(ii)] $G\not\parallel H$ implies $[G,H] = \mu$; 
    \item[(iii)] there is at least one point, some parallel class has $s\ge 2$ blocks, and there are $r\ge 3$ parallel classes. 
  \end{enumerate}
\end{definition} 
We note that the existence of an $(s,r,\mu)$-net is equivalent to the existence of an orthogonal array of strength two, see Subsection~\ref{subsec_orthogonal_arrays}. 
From partial spreads $(s,r,\mu)$-nets can be constructed, see \cite{nets_and_spreads}. Additionally, there is a connection between $3$-nets and Latin squares, see 
e.g.~\cite[Section 8.1]{keedwell2015latin}. 

Nets can be seen as a relaxation of a finite projective plane, see e.g.~\cite{ostrom1968vector}. For the famous existence question of finite projective planes of 
small order we refer to \cite{lam1991search,perrott2016existence}.

\subsection{Minihypers}
\label{subsec_minihypers}

An $(f,m;v,q)$-minihyper is a pair $(F,w)$, where $F$ is a subset of the point set of $\aspace$ and $w$ is a weight function 
$w\colon \aspace\to\mathbb{N}$, $x\mapsto w(x)$, satisfying
\begin{enumerate}
  \item[(1)] $w(x)>0$ $\Longrightarrow$ $x\in F$, 
  \item[(2)] $\sum_{x\in F} w(x)=f$, and
  \item[(3)] $\min\{\sum_{x\in H} w(x) \mid H\in\mathcal{H}\}=m$, where $\mathcal{H}$ is the set of hyperplanes of $\aspace$.
\end{enumerate}
The set of holes $\mathcal{P}$, i.e., uncovered points, of a partial $k$-spread is a $q^{k-1}$-divisible set, i.e., we have 
$\#(H\cap \mathcal{P})\equiv u\pmod {q^{k-1}}$ for some integer $u$ and each hyperplane $H\in\mathcal{H}$. Thus, $\mathcal{P}$ 
corresponds to a minihyper with $m=u$. 

Minihypers have e.g.\ been used to prove extendability results for partial spreads, see e.g.~\cite{ferret2003results,govaerts2002particular,govaerts2003particular}. 
If $\mathcal{P}$ is the set of holes of a partial $k$-spread, then the partial spread is extendible iff $\mathcal{P}$ contains of points of 
a $k$-dimensional subspace. As an example, in \cite{honold2019classification} the possible hole configurations of partial $3$ spreads in $\mathbb{F}_2^7$ 
of cardinality $16$ were classified. In four cases the partial spread is extensible and in one case it is not. In a similar vein, but using more sophisticated 
arguments, extendability results for constant-dimension codes can be found in \cite{nakic2016extendability}. 

Minihypers were used to study codes meeting the Griesmer bound, see e.g.~\cite{hamada1993characterization,hill2007geometric}. A close relation 
between divisible sets and minihypers can be found in \cite{landjev2016extendability}.

\subsection{Two-weight codes}
\label{subsec_two_weight_codes}

A linear $[n,k]_q$ code $C$ is called a \emph{two-weight code} if the non-zero codewords of $C$ attain just two possible weights. 
An online-table for known two-weight codes is at \url{http://www.tec.hkr.se/~chen/research/2-weight-codes/} and an exhaustive 
survey can be found in \cite{calderbank1986geometry}. Due to their prominence a lot of research has been done on two-weight codes 
and many examples are available. So, we have used quite some two-weight codes as examples of $q^r$-divisible sets in Section~\ref{sec_exclusion}, 
see e.g.\ Lemma~\ref{lemma_picture_q_2_r_4} and Lemma~\ref{lemma_picture_q_2_r_5}, which take examples from e.g.\ \cite{kohnert2007constructing_twoweight}.
Moreover, projective two-weight codes automatically satisfy some divisibility:
\begin{lemma}(\cite[Corollary 2]{delsarte1972weights})\\
\label{lemma_delsarte}
Let $C$ be a projective $2$-weight code over $\mathbb{F}_q$, where $q=p^e$ for some prime $p$. Then there exist suitable integers $u$ and $t$ with
  $u\ge 1$, $t\ge 0$ such that the weights are given by $w_1=up^t$ and $w_2=(u+1)p^t$.
\end{lemma}

As we will mention in Section~\ref{sec_constructions}, there is also some literature on three-weight codes, see e.g.\ \cite{calderbank1984three,ding2015class,
heng2015several,kiermaier2019three,zhou2014class}, which yield $q^r$-divisible codes in many cases, while the relation is not that strong as in Lemma~\ref{lemma_delsarte}.

\subsection{Optimal codes}
\label{subsec_optimal_codes}
Among optimal linear codes, i.e., those $[n,k,d]_q$ codes that achieve the maximum possible minimum distance $d$, there are quite some $q^r$-divisible 
codes with \textit{interesting} parameters, see 
e.g.\ the examples in the proof of Lemma~\ref{lemma_picture_q_3_r_2}. This phenomenon can partially be explained by our search technique screening 
the lists of available optimal linear codes and checking them for divisibility. Our sources were 
\url{http://www.codetables.de/} maintained by Markus Grassl, \url{http://mint.sbg.ac.at/} maintained at 
the university of Salzburg, and the database of \textit{best known linear codes} implemented in \texttt{Magma}. We refer to the latter database 
whenever we mention an optimal linear code in Section~\ref{sec_exclusion} without explicitly stating a reference.   
Another reason may be given by \cite[Theorem 1]{ward1998divisibility} stating that $[n,k,d]_p$ code meeting the 
Griesmer bound is $p^r$-divisible if $p^r$ divides the minimal distance $d$. Also in the cases where this result does not apply some optimal linear codes 
are $q^r$-divisible for some $r>1$ nevertheless. An interesting example is given by the $[46,9,20]_2$ code found in \cite{kurz201946}. It is optimal, unique, 
and does not have non-trivial authomorphisms. So, heuristic searches prescribing automorphisms had to be unsuccessful for this example. An alternative 
heuristic might be to assume $q^r$-divisibility for the largest possible $r$ so that the minimum distance $d$ is divisible by $q^r$, just to reduce 
the search space to a more manageable size. 

\subsection{$q$-analogs of group divisible designs}
\label{subsec_qgdd}
A \emph{$q$-analog of a group divisible design} of index $\lambda$  and
order $v$ is a triple $(\mathcal{V}, \mathcal{G}, \mathcal{B})$, where
\begin{itemize}
\item[--] $\mathcal{V}$ is a vector space over $\mathbb{F}_q$ of dimension $v$,
\item[--] $\mathcal{G}$ is a vector space partition whose dimensions lie in $G$, and
\item[--] $\mathcal{B}$ is a family  of subspaces (blocks) of $\mathcal{V}$,
\end{itemize}
that satisfies
\begin{enumerate}
\item $\#  \mathcal{G} > 1$,
\item if $B\in \mathcal{B}$ then $\dim B \in K$,
\item every $2$-dimensional subspace of $\mathcal{V}$ occurs in exactly  $\lambda$ blocks or one group, but not both.
\end{enumerate}

This notion was introduced in \cite{ubt_eref48691} and generalizes the classical definition of a group divisible design in the 
set case, see e.g.~\cite{brouwer1977group}. All necessary existence conditions of the set case can be easily transfered to the 
$q$-analog case. Moreover, there is an additional necessary existence condition based on $q^r$-divisible multisets, see 
\cite[Lemma 5]{ubt_eref48691} for the details.

\subsection{Codes of nodal surfaces}
\label{subsec_nodal_surfaces}

In algebraic geometry, a nodal surface is a surface in (usually complex) projective space whose only singularities are nodes. A
major problem about them is to find the maximum number of nodes of a nodal surface of given degree. In \cite{beauville1979nombre} to each 
such nodal surface is assigned. Using this link it was shown in \cite{jaffe1997sextic} that a sextic surface can have at most $65$ nodes.
Those codes are either doubly-even or triply even, depending on whether the degree of the surface is odd or even, see \cite{catanese1981babbage}.


\section{Constructions for $\mathbf{q^r}$-divisible subsets of $\mathbf{\aspace}$}
\label{sec_constructions}

In this section we want to complement the two examples and Lemma~\ref{lemma_add_configurations} from Subsection~\ref{subsec_examples_q_r_divisible}  
with further parametric families of $q^r$-divisible sets. More constructions can also be found in \cite{ubt_eref40887}. 

\begin{example}
  \label{example_projective_basis}
  A \emph{projective basis} is a set of $v+1\ge 3$ points in $\aspace$ such that the span of each $v$ points has dimension $v$.   
  Up to isomorphism, a projective basis is unique and can be parameterized as 
  $\mathcal{C}=\{\langle e_1\rangle,\dots,\langle e_v\rangle,\langle e_1+\ldots+e_v\rangle\}$,  
  where $e_i$ denote the unit vectors. For $q=2$, we have $a_{m-2j+1}={m\choose {2j-1}}+{m\choose {2j}}$\footnote{We set 
  ${m \choose {m+1}}=0$. The $a_i$ can be computed by considering hyperplanes where the corresponding normal vector consists 
  of exactly $i$ ones for $1\le i\le m$.} for $1\le j\le \frac{m+1}{2}$ and $a_i=0$ for all other indices $i$, i.e., the 
  numbers of points in the hyperplanes either are all even or they are all odd. 
\end{example}

\begin{example}
  \label{example_projective_basis_factor_geometry_line} 
  For $m\ge 2$ let $\mathcal{C}\subseteq\aspace$ be given by the union of $m$ lines through a common point $P$ such that 
  the images of the lines in the factor geometry $\F{v}{q}/P$ form a projective basis. We have $\dim(\langle\mathcal{C}\rangle)=m$ 
  and $n=|\mathcal{C}|=mq+1$. Since each hyperplane $H$ intersects a line in exactly one or $q+1$ points, we have 
  $|H\cap\mathcal{C}|\equiv 1\pmod{q}$ if $P\in H$. If $P\notin H$, then $|H\cap\mathcal{C}|=m$, i.e., $\mathcal{C}$ is $q$-divisible 
  iff $m\equiv 1\pmod{q}$. 
\end{example}

\begin{example}
  \label{example_projective_basis_factor_geometry_line_minus_point} 
  For $m\ge 2$ let $\mathcal{C}'\subseteq\aspace$ be given by the union of $m$ lines through a common point $P$ such that 
  the images of the lines in the factor geometry $\F{v}{q}/P$ form a projective basis. Set $\mathcal{C}=\mathcal{C}'\backslash P$. 
  We have $\dim(\langle\mathcal{C}\rangle)=m$ and $n=|\mathcal{C}|=mq$. Since each hyperplane $H$ intersects a line in exactly one 
  or $q+1$ points, we have $|H\cap\mathcal{C}|\equiv 0\pmod{q}$ if $P\in H$. If $P\notin H$, then $|H\cap\mathcal{C}|=m$, i.e., 
  $\mathcal{C}$ is $q$-divisible iff $m\equiv 0\pmod{q}$. 
\end{example}

Example~\ref{example_projective_basis_factor_geometry_line} and Example~\ref{example_projective_basis_factor_geometry_line_minus_point} 
can be generalized.

\begin{example}
  \label{example_construction_michael_1}
  Let $\overline{C}$ be a $q^r$-divisible set in $\aspace$ of length $\overline{n}$ and dimension $\overline{k}$. We construct a set 
  $\mathcal{C}$ of dimension $k=\overline{k}+1$. To this end let $P$ be a point outside the span of $\overline{\mathcal{C}}$. For each 
  point $Q\in\overline{\mathcal{C}}$ we add the line through $P$ and $Q$ to $\mathcal{C}$. The cardinality of $\mathcal{C}$ is given by 
  $1+q\overline{n}$. For a hyperplane $H$ we have $q^r \,|\, \overline{n}-|H\cap\overline{\mathcal{C}}|$. If $P\in H$, then 
  $|H\cap\mathcal{C}|=1+q\cdot|H\cap \overline{C}|$, so that $q^{r+1}$ divides $n-|H\cap\mathcal{C}|
  =1+q\overline{n}-(1+q\cdot|H\cap\overline{\mathcal{C}}|)=q\cdot(\overline{n}-H\cap\overline{\mathcal{C}})$. 
  If $P\notin H$, then $|H\cap \mathcal{C}|=\overline{n}$, i.e., $\mathcal{C}$ is $q^{r+1}$-divisible 
  iff $(q-1)\overline{n}\equiv -1\pmod{q^{r+1}}$, which equivalent to $\overline{n}\equiv \gaussm{r+1}{1}{q}\pmod{q^{r+1}}$. 
\end{example}

\begin{example}
  \label{example_construction_michael_2}
  Let $\overline{C}$ be a $q^r$-divisible set in $\aspace$ of length $\overline{n}$ and dimension $\overline{k}$. We construct a set 
  $\mathcal{C}$ of dimension $k=\overline{k}+1$. To this end let $P$ be a point outside the span of $\overline{\mathcal{C}}$. For each 
  point $Q\in\overline{\mathcal{C}}$ we add the line through $P$ and $Q$ without point $P$ to $\mathcal{C}$. The cardinality of 
  $\mathcal{C}$ is given by $q\overline{n}$. For a hyperplane $H$ we have $q^r \,|\, \overline{n}-|H\cap\overline{\mathcal{C}}|$. 
  If $P\in H$, then $|H\cap\mathcal{C}|=q\cdot|H\cap \overline{C}|$, so that $q^{r+1}$ divides $n-|H\cap\mathcal{C}|
  =q\overline{n}-(q\cdot|H\cap\overline{\mathcal{C}}|)=q\cdot(\overline{n}-H\cap\overline{\mathcal{C}})$. 
  If $P\notin H$, then $|H\cap \mathcal{C}|=\overline{n}$, i.e., $\mathcal{C}$ is $q^{r+1}$-divisible 
  iff $(q-1)\overline{n}\equiv 0\pmod{q^{r+1}}$, which is equivalent to $\overline{n}\equiv 0\pmod{q^{r+1}}$. 
\end{example}

Example~\ref{example_construction_michael_1} and Example~\ref{example_construction_michael_2} can be iteratively applied, since the assumptions 
are automatically satisfied.

\begin{example}
  \label{example_ovoid}
  An \emph{ovoid} in $\PG{3}{q}$ is a set $\mathcal{C}$ of $q^2+1$ points, no three collinear, such that every hyperplane 
  contains $1$ or $q+1$ points, i.e., $\mathcal{C}$ is $q$-divisible. Ovoids exist for all $q>2$, see e.g.\ \cite{o1996ovoids}.   
\end{example}

From \emph{two-weight codes}, see \cite{calderbank1986geometry} for an overview, with weights $w_1>w_2$, we obtain 
$(w_1-w_2)$-divisible codes. Several of the known two-weight codes give rise to $q^r$-divisible sets, see 
\cite[Section 13]{calderbank1986geometry}. 

\begin{example}
  Let $\mathcal{C}$ be a projective $q^r$ divisible code over $\mathbb{F}_q$. Considering the 
  a generator matrix of $\mathcal{C}$ as a generator matrix over $\mathbb{F}_{q^s}$ yields 
  a $\Delta$-divisible code over $\mathbb{F}_{q^s}$, where $\Delta=q^{r-s+1}$.
\end{example}

There is also some literature on three-weight codes, see e.g.\ \cite{zhou2014class}. For each odd prime $p$ and each 
integer $t\ge 1$ \cite[Theorem 3.7]{zhou2014class} provides a $p^t$-divisible set over $\mathbb{F}_p$ of cardinality 
$n=p^{2t+1}-p^{2t}+\frac{p-1}{2}\cdot p^t$ and dimension $k=4t+2$. For each odd prime $p$ and each integer $t\ge 3$ 
\cite[Theorem 3.11]{zhou2014class} provides a $p^t$-divisible set over $\mathbb{F}_p$ of cardinality 
$n=p^{2t}-p^{2t-1}+(p-1)\cdot p^t$ and dimension $k=4t$. In general, cyclic codes are worthwhile  field of study. 
Due to their special structure they can be enumerated for large parameters and it commonly happens that their weights 
are $q^r$-divisible. Even in the cases where not all weights are $q^r$-divisible, see e.g.~\cite{heng2015several}, 
the codes may be expanded to $q^r$-divisible codes.  

Let us revisit Lemma~\ref{lemma_add_configurations} again. 
Packing $\mathcal{C}_1$ in $\PG{v_1-1}{q}$ and $\mathcal{C}_2$ in an orthogonal $\PG{v_2-1}{q}$ ensures that the points of the two sets 
end up in different points in $\mathcal{C}$, i.e., are disjoint. Now let $\mathcal{B}$ be some $k$-dimensional subspace. If one can arrange 
$\mathcal{C}_1\cap \mathcal{B}$ and $\mathcal{C}_2\cap \mathcal{B}$ in $\mathcal{B}$ in a disjoint way, then one gets a $q^r$ divisible 
set $\mathcal{C}$ of cardinality $\left|\mathcal{C}_1\right|+\left|\mathcal{C}_2\right|$ with $\dim(\mathcal{C})= 
\dim(\langle\mathcal{C}_1\rangle)+\dim(\langle\mathcal{C}_2\rangle)-k$. We denote this construction by $\mathcal{C}_1\oplus_{-k}\mathcal{C}_2$ 
and remark  
$$\max\{\dim(\langle\mathcal{C}_1\rangle),\dim(\langle\mathcal{C}_2\rangle)\}
  \le \dim(\langle\mathcal{C}_1\oplus_{-k}\mathcal{C}_2\rangle)\le \dim(\langle\mathcal{C}_1\rangle)+\dim(\langle\mathcal{C}_2\rangle)$$
for all $k\in\mathbb{N}_0$.  

\begin{remark}
  \label{remark_save_dimension}
  Let $\mathcal{C}_1\in \PG{v_1-1}{q}$ and $\mathcal{C}_2\in \PG{v_2-1}{q}$ be $q^r$-divisible sets and $\mathcal{B}$ be a $k$-dimensional 
  subspace, where $0\le k\le \min\{v_1,v_2\}$. If $\mathcal{C}_1\cap\mathcal{B}=\emptyset$ (embedding $\mathcal{B}$ in $\PG{v_1-1}{q}$), then 
  $\mathcal{C}_1\oplus_{-k}\mathcal{C}_2$ is a $q^r$-divisible set of cardinality $|\mathcal{C}_1|+|\mathcal{C}_2|$ (if the twist of 
  $\mathcal{C}_1$ and $\mathcal{C}_2$ in $\mathcal{B}$ is suitably chosen) and dimension $\dim(\langle\mathcal{C}_1\rangle)+\dim(\langle\mathcal{C}_2\rangle)-k$. 
  If $\mathcal{C}_1$ and $\mathcal{C}_2$ are $q^r$-divisible and not both of them are subspaces entirely consisting of holes, then 
  choosing $k=1$ is always possible.  
\end{remark}

\begin{remark}
  \label{remark_combine_flat_and_affine_space} 
  Combining Example~\ref{example_r_flat} with Example~\ref{example_affine_r_space} gives $q^r$-divisible sets of cardinality $\gaussm{r+2}{1}{q}$ 
  and dimensions between $r+2$ and $2r+3$.
\end{remark}

\begin{remark}
  \label{remark_combine_two_affine_spaces} 
  Combining Example~\ref{example_affine_r_space} with itself gives $q^r$-divisible sets of cardinality $2q^{r+1}$ 
  and dimensions between $r+3$ and $2r+4$.
\end{remark}

\begin{construction}
  \label{construction_qt4p1}
  Let $P$ be a point, $L\ge P$ be a line, $E\ge L$ be a plane, and $E_i$ be planes for $1\le i\le q$ such that $\dim(E_i\cap L)=1$, 
  $\dim(E\cap E_i)=1$, and the $E_i$ are pairwise disjoint. We set $Q_i=E_i\cap L_i$ and $S_i=\langle E_i,L\rangle$. With this, let 
  $\mathcal{C}=P\cup (E\backslash L) \cup_{i=1}^q \big(S_i\backslash (E_i\cup L)\big)$, so that $|\mathcal{C}|=1+q^2 +q\cdot (q^3-q)=q^4+1$ 
  and $\dim(\langle \mathcal{C}\rangle)=3+2q$. 
  We now verify that $\mathcal{C}$ is $q^2$-divisible. To this end let $H$ be an arbitrary hyperplane and we consider the following cases:
  \begin{enumerate}
    \item[(1)] $H\cap L=P$: We have $\dim(H\cap S_i)=$ and $\dim(H\cap E_i)=2$, so that $(H\cap S_i)\backslash P$ contains $q^2-1$ holes. 
               Since $\dim(H\cap E)=2$, the line $H\cap E$ contains $q$ holes. Thus, $|H\cap\mathcal{C}|=q\cdot(q^2-1)+q+1=q^3+1$.
    \item[(2)] $H\cap L=Q_j$: We have $\dim(H\cap E_j)=\dim(H\cap S_j)=3$ and $\dim(H\cap E_i)=2$, $\dim(H\cap S_i)=3$ for all 
               $1\le i\le q$ with $i\neq j$. Since $\dim(H\cap E)=2$, the line $H\cap E$ contains $q$ holes. Thus, $|H\cap\mathcal{C}|=
               1\cdot 0+(q-1)\cdot(q^2-1)+q=q^3-q^2+1$.
    \item[(3)] $L\le H$: The number of holes in $(H\cap E)\backslash P$ is either zero or $q^2$. For each $1\le i\le q$ the 
               number of holes in $(H\cap E_i)\backslash P$ is either $q^2-q$ or $q^3-q$, so that $|H\cap\mathcal{C}|\equiv 1 \pmod {q^2}$.                        
  \end{enumerate} 
\end{construction}

\begin{construction}
  \label{construction_q_2_n_17_k_8_generalization}
  Let $\mathcal{C}_1'$ be $q$-divisible with cardinality $m$, set $F'=\langle\mathcal{C}_1'\rangle$, and let $P$ be a point outside 
  of $F'$. With this, set $\mathcal{C}_1=\{\langle Q,P\rangle\backslash P \,:\, Q\in \mathcal{C}_1'\}$, i.e., $\left|\mathcal{C}_1\right|=mq$ 
  and $\dim(\langle\mathcal{C}_1\rangle)=\dim(\langle\mathcal{C}_1'\rangle)+1$. Let $S$ be a solid trough $P$ such that 
  $S\cap\langle\mathcal{C}_1\rangle=P$ and let $E$ be a plane in $S$ that is disjoint to $P$. With this, set $\mathcal{C}_2=S\backslash\{E\cup P\}$, 
  i.e., $\left|\mathcal{C}_2\right|=q^3-1$ and $\dim(\mathcal{C}_2)=4$. For $\mathcal{C}=\mathcal{C}_1\cup\mathcal{C}_2$ we have 
  $\left|\mathcal{C}\right|=mq+q^3-1$ and $\dim(\langle\mathcal{C}\rangle)=\dim(\langle\mathcal{C}_1'\rangle)+4$.
\end{construction}  

\begin{lemma}
  \label{lemma_q_2_n_17_k_8_generalization}
  If $m\equiv q^2-q-1\equiv -q-1\pmod{q^2}$, then Construction~\ref{construction_q_2_n_17_k_8_generalization} gives a $q^2$-divisible set. 
\end{lemma}
\begin{proof}
  Let $H$ be an arbitrary hyperplane. If $P\le H$, then $|H\cap\mathcal{C}_1|\equiv q^2-q\pmod {q^2}$ since $m\equiv q-1\pmod{q}$, 
  $\mathcal{C}_1'$ is $q$-divisible, and each line in $\mathcal{C}_1$ is intersected by $H$ in dimension $1$ or $2$. If $P\notin H$, 
  then $|H\cap\mathcal{C}_1|=m\equiv q^2-q-1\pmod{q^2}$. If $P\le H$, then $|H\cap\mathcal{C}_2|\in \{q^3-1,q^2-1\}$, i.e., 
  $|H\cap\mathcal{C}_2|\equiv -1\pmod {q^2}$. If $P\notin H$, then $|H\cap\mathcal{C}_2|= q^2$, i.e., $|H\cap\mathcal{C}_2|\equiv 0\pmod {q^2}$. 
\end{proof}

\begin{corollary}
  \label{corollary_q_2_n_17_k_8_generalization}
  For each prime power $q$, there exists a $q^2$-divisible set of cardinality $3q^3-q^2-q-1$ and dimension $\max\{8,q+5\}\le k\le \max\{8,2q+3\}$.
\end{corollary}
\begin{proof}
  For $q=2$ choose $\mathcal{C}_1'$ as a projective basis of cardinality $5$, see Example~\ref{example_projective_basis}. 
  For $q=3$ choose $\mathcal{C}_1'$ as the union of an ovoid, see Example~\ref{example_ovoid}, and $q-2$ lines, see 
  Example~\ref{example_r_flat} for $r=1$.
\end{proof}

\begin{construction}
  \label{construction_surgery}
  Let $\Delta$ and $\Delta'$ be two integers such that $\rho:=\frac{\Delta}{\Delta'}\in\mathbb{N}$ and $q$ be an arbitrary prime power. 
  With this, let $\mathcal{C}_1,\dots,\mathcal{C}_{\rho}$ be $\Delta$-divisible sets over $\mathbb{F}_q$ such that 
  $\mathcal{C}_i\cap \langle\mathcal{D}\rangle=\mathcal{D}$ for a $\Delta'$-divisible set $\mathcal{D}$ over $\mathbb{F}_q$ and 
  $1\le i\le \rho$. Then, $\mathcal{C}:=\oplus_{i=1}^{\rho} \left(\mathcal{C}_i\backslash\mathcal{D}\right)$ is a 
  $\Delta$-divisible set of cardinality $n=\sum_{i=1}^{\rho}\left|\mathcal{C}_i\right|-\rho\cdot\left|\mathcal{D}\right|$ 
  and dimension $\dim(\mathcal{D})+\sum_{i=1}^{\rho} \left(\dim(\mathcal{C}_i)-\dim(\mathcal{D})\right)$.
\end{construction}
\begin{proof}
  We write $\left|\mathcal{C}_i\right|=u_i+l_i\Delta$ and $\left|\mathcal{D}\right|=u'+l'\Delta'$ for some integers 
  $u_i,l_i,u',l'$, so that $\left|\mathcal{C}\right|=n=\sum_{i=1}^\rho\left|\mathcal{C}_i\right|-\rho\cdot\left|\mathcal{D}\right|\equiv 
  \sum_{i=1}^\rho u_i\,-\rho u'\pmod {\Delta}$. For an arbitrary but fixed hyperplane $H$ there are integers $a_i$, $a'$ such that
  $\left|H\cap\mathcal{C}_i\right|=u_i+a_i\Delta$ and $\left|H\cap\mathcal{D}\right|=u'+a'\Delta'$. With this, we compute 
  $$
    \left|H\cap \mathcal{C}\right|=\sum_{i=1}^\rho\left(u_i+a_i\Delta\right)\,-\rho\cdot \left(u'+a'\Delta'\right)
    \equiv \sum_{i=1}^\rho u_i\,-\rho u'\pmod {\Delta}, 
  $$ 
  i.e., $\mathcal{C}$ is $\Delta$-divisible. Using the characteristic function $\chi_\mathcal{S}$ of a point set $\mathcal{C}$, 
  we may also directly read of the divisibility from $\chi_{\mathcal{C}}=\sum_{i=1}^\rho\chi_{\mathcal{C}_i}-\rho\cdot \chi_{\mathcal{D}}$.
\end{proof}

\begin{corollary}
  \label{corollary_surgery_1}
  There exist $q^r$-divisible sets of cardinality $n=q^{r+1}$ and dimension $r+2\le k\le qr$ for all feasible parameters $q$, $r$. 
\end{corollary}
\begin{proof}
  Apply Construction~\ref{construction_surgery} with $q$ flats of dimension $r+1$. 
\end{proof}

\begin{corollary}
  \label{corollary_surgery_2}
  There exist $q^r$-divisible sets of cardinality $$n=\gaussm{2r}{1}{q}+j\cdot\left((q-1)q^r-\gaussm{r}{1}{q}\right),$$ where 
  $0\le j\le q^r+1$, for all feasible parameters $q$, $r$.
\end{corollary}
\begin{proof}
  Consider an $r$-spread of $\mathbb{F}_q^{2r}$ having cardinality $q^r+1$. For $0\le j\le q^r+1$ we replace 
  $j$ of the $r$-dimensional elements of the spread by $q-1$ times an $(r+1)$-dimensional subspace minus the spread element, 
  i.e., we apply Construction~\ref{construction_surgery} iteratively.
\end{proof}

\begin{remark}
  \label{remark_surgery_baer}
  Starting from a three-dimensional subspace in $\mathbb{F}_4^v$, we can apply Construction~\ref{construction_surgery} 
  and replace Baer planes by an affine Baer solid yielding $4^1$ divisible sets for $21\le n\le 24$. 
\end{remark}

The search problem for $q^r$-divisible codes can easily be formulated as an integer linear programming problem using binary 
characteristic variables $x_P$ for all points $P$ of $\aspace$. Prescribing the desired cardinality $n=\sum_{P} x_P$ and the dimension 
$k$, it remains to convert the restrictions of Definition~\ref{def_delta_divisible} into linear constraints:
$$
  \sum_{P\le H} x_P = n- z_H\cdot q^r
$$   
for each hyperplane $H$, where $z_H\in \mathbb{Z}$, $0\le z_H\le \left\lfloor n/q^r\right\rfloor$. The integer variable $z_H$ 
may be replaced by several binary variables $y_{H,n'}$, which are equal to $1$ iff hyperplane $H$ contains exactly $n'$ selected 
points, i.e., holes. This way, it is possible to exclude some specific values for the number of holes in hyperplanes or to 
count (and incorporate given bounds on) the number of hyperplanes with a given number of holes. Restrictions for $n$ and the $n'$ 
are given by our exclusion results for $q^r$-divisible and $q^{r-1}$-divisible sets, respectively, see 
Section~\ref{sec_exclusion}. Prescribing a specific solution of the MacWilliams identities directly translates to equations for 
the number of hyperplanes with a given number of holes. Although, Lemma~\ref{lemma_heritable} allows to include \textit{modulo-constraints} 
on the number of holes for subspaces other than hyperplanes, ILP solvers seem not to benefit from these extra constraints.  
Since larger instances can not be successfully treated directly by customary ILP solvers, we have additionally prescribed some 
symmetry to find examples. This general approach is called the Kramer--Mesner method \cite{KramerMesner:76}. For an exemplary 
application to the construction of constant-dimension codes we refer e.g.\ to \cite{MR2796712}.


\section{Analytical tools to exclude specific cardinalities of $\mathbf{q^r}$-divisible subsets of $\mathbf{\aspace}$}
\label{sec_tools_to_exclude}

In this section we summarize analytical conclusions from the linear programming method applied to the first few 
MacWilliams identities. We took the material mainly from \cite{kurz2017packing,honold2018partial} and give further 
details and remarks.

As a first non-existence result, we observe that cardinality $1$ can be excluded in all non-trivial cases. 

\begin{lemma}
  \label{lemma_exclude_cardinality_one}
  No non-trivial $q^1$-divisible $\mathcal{C}\subseteq\PG{v}{q}$ with $n=|\mathcal{C}|=1$ exists. 
\end{lemma}
\begin{proof}
  In Definition~\ref{def_delta_divisible} we have to choose $u=1$ due to $q\ge 2$, so that all hyperplanes have to 
  contain the unique point, which is not possible. 
\end{proof}

From the standard equations of Lemma~\ref{lemma_standard_equations_q} we can deduce the non-existence of $q^r$ divisible 
sets for many cases where $n$ is relatively small.

\begin{definition}
  \label{def_hyperplane_types}
  For $\mathcal{C}\subseteq \aspace$ we use the abbreviation 
  $$\mathcal{T}(\mathcal{C}):=\left\{0\le i\le c\mid a_i>0\right\},$$
  where $a_i$ denotes the number of hyperplanes with $|\mathcal{C}\cap H|=i$.
\end{definition}

\begin{lemma}
  \label{lemma_hyperplane_types_arithmetic_progression_2}
  For integers $u\in\mathbb{Z}$, $m\ge 0$, and $\Delta\ge 1$ let $\mathcal{C}\subseteq \aspace$ 
  be $\Delta$-divisible of cardinality $n=u+m\Delta\ge 0$.  
 Then, we have $$
   (q-1)\cdot\sum_{h\in\mathbb{Z},h\le m} ha_{u+h\Delta}=\left(u+m\Delta-uq\right)\cdot \frac{q^{v-1}}{\Delta}-m,
 $$
 where we set $a_{u+h\Delta}=0$ if $u+h\Delta<0$.
\end{lemma}
\begin{proof}
  Rewriting the equations from Lemma~\ref{lemma_standard_equations_q}  
  yields 
  \begin{eqnarray*}
    (q-1)\cdot\sum_{h\in\mathbb{Z},h\le m} a_{u+h\Delta} &=& q\cdot q^{v-1}-1 \text{ and}\\
    (q-1)\cdot\sum_{h\in\mathbb{Z},h\le m} (u+h\Delta)a_{u+h\Delta} &=& (u+m\Delta)(q^{v-1}-1).
  \end{eqnarray*}
  $u$ times the first equation minus the second equation gives $\Delta$ times the stated equation. 
\end{proof}
 
\begin{corollary}
  \label{cor_nonexistence_arithmetic_progression_2}
  For integers $u,m\ge 0$ and $\Delta\ge 1$ let $\mathcal{C}\subseteq \aspace$ satisfy 
  $n=|\mathcal{C}|=u+m\Delta$ and $\mathcal{T}(\mathcal{C})\subseteq\{u,u+\Delta,\dots,u+m\Delta\}$. 
  Then, $u<\frac{m\Delta}{q-1}$ or $u=m=0$.
\end{corollary} 

\begin{lemma}
  \label{lemma_average}
  If $\mathcal{C}\subseteq\aspace$ is $q^r$-divisible with $|\mathcal{C}|=a\cdot q^{r+1}+b$ for $a,b\in\mathbb{N}_0$, 
  then there exists a hyperplane $H$ such that $|\mathcal{C}\cap H|\le (a-1)\cdot q^{r}+b$.  
\end{lemma}
\begin{proof}
  Set $n=|\mathcal{C}|$ and choose a hyperplane $H$ such that $n':= |\mathcal{C}\cap H|$ is minimal. Then, we have
  $$
    n'\le \underset{\text{average}}{\underbrace{\frac{1}{\gaussm{v}{1}{q}}\cdot \sum_{\text{hyperplane} H'}  \left|\mathcal{C}\cap H'\right|}}
    =n\cdot\gaussm{v-1}{1}{q}/\gaussm{v}{1}{q}<\frac{n}{q}
    \quad\text{and}\quad
    n'\equiv b\pmod{q^m}.\\[-5mm]
  $$ 
\end{proof}

Using not only the first two standard equations gives further conditions. 

\begin{lemma}
  \label{lemma_hyperplane_types_arithmetic_progression}
  For integers $u\in\mathbb{Z}$, $m\ge 0$, and $\Delta\ge 1$ let $\mathcal{C}\subseteq \aspace$ be $\Delta$-divisible of cardinality $n=u+m\Delta\ge 0$.
  Then, we have $$(q-1)\cdot\sum_{h\in\mathbb{Z},h\le m} h(h-1)a_{u+h\Delta}=\tau_q(u,\Delta,m)\cdot \frac{q^{v-2}}{\Delta^2}-m(m-1),$$ 
  where we set
  $
    \tau_q(u,\Delta,m)=m(m-q)\Delta^2+\left(q^2u-2mqu+mq+2mu-qu-m\right)\Delta+(q-1)^2u^2+(q-1)u
  $  
  and $a_{u+h\Delta}=0$ if $u+h\Delta<0$.
\end{lemma}
\begin{proof}
  Rewriting the equations from Lemma~\ref{lemma_standard_equations_q}  
  yields 
  \begin{eqnarray*}
    (q-1)\cdot\sum_{h\in\mathbb{Z},h\le m} a_{u+h\Delta} &=& q^2\cdot q^{v-2}-1,\\
    (q-1)\cdot\sum_{h\in\mathbb{Z},h\le m} (u+h\Delta)a_{u+h\Delta} &=& (u+m\Delta)(q\cdot q^{v-2}-1),\\
    (q-1)\cdot\!\!\!\!\!\sum_{h\in\mathbb{Z},h\le m}\!\!\!\!\! (u+h\Delta)(u+h\Delta-1)a_{u+h\Delta} &=& (u+m\Delta)(u+m\Delta-1)(q^{v-2}-1).
  \end{eqnarray*}
  $u(u+\Delta)$ times the first equation minus $(2u+\Delta-1)$ times the second equation plus the third equation gives 
  $\Delta^2$ times the stated equation. 
\end{proof}

\begin{corollary}
  \label{cor_nonexistence_arithmetic_progression}
  For integers $u\in\mathbb{Z}$ and $\Delta,m\ge 1$ let $\mathcal{C}\subseteq \aspace$ be $\Delta$-divisible of cardinality $n=u+m\Delta\ge 0$. 
  If one of the following conditions hold, then $(q-1)\cdot \sum_{i=2}^m i(i-1)x_i\notin\mathbb{N}_0$, which is impossible.
  \begin{enumerate}
    \item[(a)] $\tau_q(u,\Delta,m)<0$;
    \item[(b)] $\tau_q(u,\Delta,m)\cdot q^{v-2}$ is not divisible by $\Delta^2$;
    \item[(c)] $m\ge 2$ and $\tau_q(u,\Delta,m)=0$.
  \end{enumerate}
  We have the following special cases:
  \begin{eqnarray*}
    \tau_q(u,q^r,m)&=&\left(m(m-q)q^r-2mqu+q^2u+mq+2mu-qu-m\right)\cdot q^r\\
     && +\left(q^2u^2-2qu^2+qu+u^2-u\right),\\
    \tau_2(u,2^r,m)&=&\left(m(m-2)2^r -2mu+m+2u\right)\cdot 2^r +\left(u^2+u\right).\\
  \end{eqnarray*}
\end{corollary}

Condition (b) gives a partial explanation why the cases with $\gcd(\Delta,q)=1$ are that restricted.

\begin{lemma}
  \label{lemma_negative_tau}    
  Given a positive integer $m$, we have $\tau_q(u,\Delta,m)\le 0$ iff 
  \begin{eqnarray}
    \label{ie_forbidden_interval}
    &&(q-1)u-(m-q/2)\Delta+\frac{1}{2}\notag\\ 
    &\in& \left[-\frac{1}{2}\cdot\sqrt{q^2\Delta^2-4qm\Delta+2q\Delta+1},
    \frac{1}{2}\cdot\sqrt{q^2\Delta^2-4qm\Delta+2q\Delta+1}\right].
  \end{eqnarray}
  The last interval is non-empty, i.e., the radicand is non-negative, iff $1\le m\le \left\lfloor(q\Delta+2)/4\right\rfloor$. 
  We have $\tau_q(u,\Delta,1)=0$ iff $u=(\Delta-1)/(q-1)$.
\end{lemma}
\begin{proof}
  Solving $\tau_q(u,\Delta,m)=0$ for $u$ yields the boundaries for $u$ stated in Inequality~(\ref{ie_forbidden_interval}).  
  Inside this interval we have $\tau_q(u,\Delta,m)\le 0$. Now, $q^2\Delta^2-4qm\Delta+2q\Delta+1\ge 0$ is equivalent to 
  $m\le \frac{q\Delta}{4}+\frac{1}{2}+\frac{1}{4q\Delta}$. Rounding down the right hand side, while observing $\frac{1}{4q\Delta}<
  \frac{1}{4}$ yields $\left\lfloor(q\Delta+2)/4\right\rfloor$.
\end{proof}

We remark that \cite[Theorem 1.B]{bose1952orthogonal} is quite similar to Lemma~\ref{lemma_hyperplane_types_arithmetic_progression} 
and its implications. For the use of a quadratic non-negative polynomial over the integers see Inequality~(3.2). The multipliers used 
in the proof of Lemma~\ref{lemma_hyperplane_types_arithmetic_progression} can be directly read off from the following observation.

\begin{lemma}
  \label{lemma_binomial_matrix_3}
  For pairwise different non-zero 
  numbers $a,b,c$ the inverse matrix of 
  $$
    \begin{pmatrix}
      1     & 1     & 1     \\
      a     & b     & c     \\
      a^2-a & b^2-b & c^2-c \\
    \end{pmatrix}
  $$
  is given by 
  $$
    \begin{pmatrix}
       bc(c-b) & -(c+b-1)(c-b) &  (c-b) \\
      -ac(c-a) &  (c+a-1)(c-a) & -(c-a) \\
       ab(b-a) & -(b+a-1)(b-a) &  (b-a) \\
    \end{pmatrix}\cdot \big((c-a)(c-b)(b-a)\big)^{-1}
  $$
\end{lemma}

As we have remarked before, the standard equations correspond to the first three MacWilliams identities. By additionally considering the 
fourth MacWilliams identity we obtain a further criterion. Before we state the general result, we illustrate it by a concrete example  

\begin{lemma}
  \label{lemma_no_8_div_52}
  No $8$-divisible $\mathcal{C}\subseteq\F{v}{2}$ of cardinality $52$ exists.
\end{lemma}
\begin{proof}
  From Lemma~\ref{lemma_picture_q_2_r_2} we conclude that there is no hyperplane with $4$ or $12$ holes, i.e., $A_{40}=A_{48}=0$.
  Using the abbreviation $y=2^{v-3}$ the first four MacWilliams identities, see~(\ref{mac_williams_identies}), are given by
  \begin{eqnarray*}
    A_0+A_8+A_{16}+A_{24}+A_{32} &=& 8y\cdot A_0^\perp \\
    {52\choose1}A_0+{44\choose1}A_8+{36\choose1}A_{16}+{28\choose1}A_{24}+{20\choose1}A_{32} &=& 4y\cdot 52A_0^\perp \\
    {52\choose2} A_0+{44\choose2} A_8+{36\choose2} A_{16}+{28\choose2} A_{24}
    +{20\choose2} A_{32} &=& 2y\cdot {52\choose2} A_0^\perp \\
    {52\choose3} A_0+{44\choose3} A_8+{36\choose3} A_{16}&&\\
    +{28\choose3} A_{24}+{20\choose3} A_{32} &=& y\cdot\left({52\choose3} A_0^\perp+A_3^\perp\right)
  \end{eqnarray*} 
  Plugging in $A_0=A_0^\perp=1$ and substituting $x=yA_3^\perp$ yields
  \begin{eqnarray*}
    A_8 &=& -4+\frac{1}{512} x+\frac{7}{64}y\\
    A_{16} &=& 6-\frac{3}{512} x-\frac{17}{64}y\\
    A_{24} &=& -4+ \frac{3}{512} x+\frac{397}{64}y\\
    A_{32} &=& 1- \frac{1}{512} x+\frac{125}{64}y.
  \end{eqnarray*}
  Since $A_{16},x\ge 0$, we have $y\le \frac{384}{17}<32$, so that $v\le 7$. Since $3A_8+A_{16}\ge 0$, 
  we have $-6+\frac{y}{16}\ge 0$, i.e., $y\ge 96$, so that $v\ge 10$ -- a contradiction.
\end{proof}

We remark that the non-existence of a $2^3$-divisible set of cardinality $n=52$ implies several upper bounds for partial spreads, e.g., 
$129\le A_2(11,8;4)\le 132$, $2177\le A_2(15,8;4) \le 2180$, and $34945\le A_2(19,8;4) \le 34948$. The underlying idea of the proof of 
Lemma~\ref{lemma_no_8_div_52} can be generalized:

\begin{lemma}
  \label{lemma_implication_fourth_mac_williams}
  For $t\in\mathbb{Z}$ be an integer and $\mathcal{C}\subseteq \aspace$ be $\Delta$-divisible of cardinality $n>0$. Then, we have
  $$
    \sum_{i\ge 1} \Delta^2(i-t)(i-t-1)\cdot (g_1\cdot i+g_0)\cdot A_{i\Delta}\,\,+qhx 
    = n(q-1)(n-t\Delta)(n-(t+1)\Delta)g_2,
  $$
  where $g_1=\Delta qh$, $g_0=-n(q-1)g_2$, $g_2=h-\left(2\Delta qt+\Delta q-2nq+2n+q-2\right)$ and 
  $$
    h= \Delta^2q^2t^2+\Delta^2q^2t-2\Delta nq^2t-\Delta nq^2+2\Delta nqt+n^2q^2+\Delta nq-2n^2q+n^2+nq-n.  
  $$ 
\end{lemma}
\begin{proof}
  We slightly rewrite the first four MacWilliams identities to
  \begin{eqnarray*}
    \sum_{i=1}^s A_{i\Delta} \,\,-q^3y &=& -1 \\
    \sum_{i=1}^s (n-i\Delta)\cdot A_{i\Delta} \,\,-nq^2y &=& -n \\
    \sum_{i=1}^s (n-i\Delta)(n-i\Delta-1)\cdot A_{i\Delta} \,\,-2{n\choose 2}qy &=& -2{n\choose 2} \\
    \sum_{i=1}^s (n-i\Delta)(n-i\Delta-1)(n-i\Delta-2)\cdot A_{i\Delta} \,\,-6{n\choose3}y -x&=& -6{n\choose3},
  \end{eqnarray*}
  where $s$ is suitably large, $k=\dim(\langle \mathcal{C}\rangle)$, $y=q^{k-3}$, and $x=y\cdot A_3^\perp$. We consider 
  a linear combination of those equations with multipliers
  \begin{eqnarray*}
    f_1 &=& (bcq^2-bnq-bq^2-cnq-cq^2+bq+cq+n^2+3nq+q^2-3n-3q+2)bcn\\
    f_2 &=& -b^2c^2q^3+b^2cq^3+bc^2q^3+b^2n^2q+bcn^2q-bcq^3+c^2n^2q-b^2nq-bcnq\\ 
        &&  -bn^3-3bn^2q-c^2nq-cn^3-3cn^2q+3bn^2+3bnq+3cn^2+3cnq+n^3\\ 
        && +2n^2q-2bn-2cn-3n^2-2nq+2n\\
    f_3 &=& b^2cq^3+bc^2q^3-b^2nq^2-bcnq^2-3bcq^3-c^2nq^2+3bnq^2+3cnq^2+n^3\\ 
        &&-2nq^2-3n^2+2n\\
    f_4 &=& -(bcq^2-bnq-cnq+n^2+nq-n)q,
  \end{eqnarray*}
  where $b=n-t\Delta$ and $c=n-(t+1)\Delta$. 
  
  For the coefficient of $A_{i\Delta}$ we have 
  \begin{eqnarray*}
    && f_1+f_2\cdot(n-i\Delta)+f_3\cdot(n-i\Delta)(n-i\Delta-1)\\ 
    &&+f_4\cdot (n-i\Delta)(n-i\Delta-1)(n-i\Delta-2)\\
    &=& \Delta^2(i-t)(i-t-1)\cdot (g_1\cdot i+g_0). 
  \end{eqnarray*}
  
  The coefficient of $y$ vanishes, i.e., $-q^3f_1-nq^2f_2 -n(n-1)qf_3-n(n-1)(n-2)f_4=0$. The coefficient of $x$ is 
  given by $(bcq^2-bnq-cnq+n^2+nq-n)q=qh$. The right hand side is given by 
    $-f_1-nf_2-n(n-1)f_3-n(n-1)(n-2)f_4 = n(q-1)(n-t\Delta)(n-(t+1)\Delta)g_2$.
\end{proof}

\begin{corollary}
  \label{cor_implication_fourth_mac_williams}
  Using the notation of Lemma~\ref{lemma_implication_fourth_mac_williams}, if $n/\Delta\notin [t,t+1]$, $h\ge 0$, and $g_2<0$, 
  then there exists no $\Delta$-divisible set $\mathcal{C}\subseteq \aspace$ of cardinality $n$.
\end{corollary}
\begin{proof}
  First we observe $(i-t)(i-t-1)\ge 0$, $(n-t\Delta)(n-(t+1)\Delta)> 0$, and $g_1\ge 0$. Since $g_2<0$, we have $g_0\ge 0$ 
  so that $g_1i+g_0\ge 0$. Thus, the entire left hand side is non-negative and the right hand side is negative --a contradiction.
\end{proof}

Applying Corollary~\ref{cor_implication_fourth_mac_williams} with $t=3$ gives Lemma~\ref{lemma_implication_fourth_mac_williams}. 
For the somehow related use of a cubic polynomial over the integers in a related context see \cite[Section 4]{bose1952orthogonal}, especially 
Inequality~(4.1). In the proof of Lemma~\ref{lemma_implication_fourth_mac_williams} we are essentially solving the linear equation system, given by 
the first four MacWilliams identities,for $A_{s\Delta}, A_{t\Delta}, A_{(t+1)\Delta}$ and $y$. The corresponding multipliers are 
given by:
\begin{lemma}
  \label{lemma_special_matrix_4}
  For pairwise different numbers $a,b,c,n$ and $q,y\neq 0$ let 
  $$
    M=
    \begin{pmatrix}
      1           & 1           & 1           & -q^3y         \\
      a           & b           & c           & -nq^2y        \\
      a(a-1)      & b(b-1)      & c(c-1)      & -n(n-1)qy     \\
      a(a-1)(a-2) & b(b-1)(b-2) & c(c-1)(c-2) & -n(n-1)(n-2)y \\
    \end{pmatrix}.
  $$
  With this, the entries of the first row of $\frac{1}{(b-c)y}\cdot \det(M)\cdot M^{-1}$ are given by 
  $f_1$, $f_2$, $f_3$, and $f_4$ as stated in the proof of Lemma~\ref{lemma_implication_fourth_mac_williams}. 
\end{lemma}
\begin{proof}
  Just insert the expression into a computer algebra system like e.g.\ \texttt{Maple}.
\end{proof}

As a further example we consider the parameters $q=2$, $\Delta=2^4=16$, and $n=235$. The condition $n/\Delta\notin [t,t+1]$ excludes $t\in\{14,15\}$. 
The condition $h\ge 0$ is satisfied for all integers $t$ since the excluded interval $(6.700,6.987)$ contains no integer. The condition 
$g_2<0$ just allows to choose $t=7$, which also satisfies $qh\ge -g_0$. 

We can perform a closer analysis in order to develop computational cheap checks. We have $g_2<0$ iff
$$
  n\in\left( \frac{\Delta qt+\frac{\Delta q}{2}-\frac{3}{2}-\frac{1}{2}\cdot\sqrt{\omega}}{q-1}, \frac{\Delta qt+\frac{\Delta q}{2}-\frac{3}{2}+\frac{1}{2}\cdot\sqrt{\omega}}{q-1}\right),
$$
where $\omega=\Delta^2q^2-4qt\Delta-2\Delta q+4q+1$. Thus, $\omega>0$, i.e., we have $$t\le \left\lfloor\frac{q\Delta-2}{4}+\frac{1}{\Delta}+\frac{1}{4q\Delta}\right\rfloor.$$ 
We have $h\ge 0$ iff
$$
  n\notin\left( \frac{\Delta qt+\frac{\Delta q}{2}-\frac{1}{2}-\frac{1}{2}\cdot\sqrt{\omega-4q}}{q-1}, \frac{\Delta qt+\frac{\Delta q}{2}-\frac{1}{2}+\frac{1}{2}\cdot\sqrt{\omega-4q}}{q-1}\right).
$$
The most promising possibility, if not the only at all, seems to be 
$$
  n\in \Big( \frac{\Delta qt+\frac{\Delta q}{2}-\frac{3}{2}-\frac{1}{2}\cdot\sqrt{\omega}}{q-1},\frac{\Delta qt+\frac{\Delta q}{2}-\frac{1}{2}-\frac{1}{2}\cdot\sqrt{\omega-4q}}{q-1}\Big],
$$
which allows the choice of at most one integer $n$. In our example $q=2$, $\Delta=2^4=16$ the possible $n$ for $t=1,\dots,7$ correspond to 
$33,66,99,132,166,200,235$, respectively. The two other conditions are automatically satisfied.


\subsection{The complete classification of the possible lengths $n$ of $q^r$-divisible sets when $n$ is \textit{small}}
\label{subsec_length_classification_small} 
 
While the possible length of $q^r$-divisible multisets and codes have been completely characterized, see Subsection~\ref{subsec_characterization_q_r_multisets}, 
an analogous result for $q^r$-divisible sets seems to be out of reach. In this subsection we want to solve that problem for the 
\textit{small} lengths $n\le rq^r+1$. 

Lemma~\ref{lemma_negative_tau} excludes quite some values. We start by analyzing the right side of the corresponding interval. 
First we note that examples of $q^r$-divisible sets of cardinality $m\cdot \gaussm{r+1}{1}{q}$ can be obtained from Example~\ref{example_r_flat} 
for all $m\in\mathbb{N}_{>0}$. If $m$ is not too large, then cardinalities one less are impossible.

\begin{lemma}
  \label{lemma_m_bound_rhs}
  For $1\le m\le \left\lfloor\sqrt{(q-1)q\Delta}-q+\frac{3}{2}\right\rfloor$, we have 
  $$
    (q-1)(n-m\Delta)-(m-q/2)\Delta+\frac{1}{2}\le\frac{1}{2}\cdot\sqrt{q^2\Delta^2-4qm\Delta+2q\Delta+1},
  $$
  where $n=m\cdot \gaussm{r+1}{1}{q}-1$ and $\Delta=q^r$.
\end{lemma}
\begin{proof}
  Plugging in and simplifying yields $$q\Delta+3-2m-2q\le\sqrt{q^2\Delta^2-(4m-2)q\Delta+1},$$ so that squaring and 
  simplifying gives $m\le \sqrt{(q-1)q\Delta}-q+\frac{3}{2}$.
\end{proof}

\begin{theorem}
  \label{thm_exclusion_r_1_to_ovoid}
  Let $\mathcal{C}\subseteq \aspace$ be $q^1$-divisible with $2\le n=|\mathcal{C}|\le q^2$, 
  then either $n=q^2$ or $q+1$ divides $n$. Additionally, the non-excluded cases can be realized. 
\end{theorem}
\begin{proof}
  First we show $n\notin \left[(m-1)(q+1)+2,m(q+1)-1\right]$ for $1\le m\le q-1$. To this end, we apply Lemma~\ref{lemma_negative_tau} 
  to deduce $\tau_{q}(u,q,m)\le 0$ for $m+1-q\le u\le m-1$, so that the statement follows from 
  Corollary~\ref{cor_nonexistence_arithmetic_progression}. 
  For $u\ge m+1-q$ we have
  \begin{eqnarray*}
  (q-1)u-(m-q/2)\Delta+\frac{1}{2} &\ge& -\frac{1}{2}\cdot\left(q^2-4q+1+2m\right)\\ 
  &\ge&
   -\frac{1}{2}\cdot
   \left(q^2-2m-3\right)\\
 &\ge& -\frac{1}{2}\cdot\sqrt{q^4-4mq^2+2q^2+1}\\ 
  &=& -\frac{1}{2}\cdot\sqrt{q^2\Delta^2-4qm\Delta+2q\Delta+1}
  \end{eqnarray*}
  and for $u\le m-1$ we have
  \begin{eqnarray*}
    (q-1)u-(m-q/2)\Delta+\frac{1}{2}&\le& \frac{1}{2}\cdot\left(q^2-2m -2q+3\right)\\
    &\overset{\star}{\le}& \frac{1}{2}\cdot\sqrt{q^4-4mq^2+2q^2+1}\\ 
  &=& \frac{1}{2}\cdot\sqrt{q^2\Delta^2-4qm\Delta+2q\Delta+1}.
  \end{eqnarray*}
  With respect to the estimation $\star$, we remark that 
  $$
    -4q^3+8q^2-12q+8 +4m(m+2q-3) \overset{m\le q-1}{\le} -4(q-1)(q^2-4q+6)\overset{q\ge 2}{\le} 0. 
  $$
  
  Applying Corollary~\ref{cor_nonexistence_arithmetic_progression_2} with $u=m+1$ and $\Delta=q$ 
  yields $n\neq m(q+1)+1$ for all $1\le m\le q-2$.
  Example~\ref{example_r_flat} provides a $q$-divisible set $\mathcal{C}\subseteq \aspace$ of cardinality $q+1$ for $r=1$. 
  Using Lemma~\ref{lemma_add_configurations} we obtain such sets, of cardinality $n$, for all $n>0$ with $n\equiv 0\pmod{q+1}$. 
  An example of cardinality $q^2$ is given in Example~\ref{example_affine_r_space} for $r=1$.
\end{proof}

We remark that there exists a $q^1$-divisible set of cardinality $q^2+1$ for all $q\ge 2$ (given by ovoids for $q\ge 3$ and 
a projective base for $q=2$).

\begin{theorem}
  \label{thm_exclusion_q_r}
  For the cardinality $n$ of a $q^r$-divisible set $\mathcal{C}$ we have
  $$
    n\notin\left[(a(q-1)+b)\gaussm{r+1}{1}{q}+a+1,(a(q-1)+b+1)\gaussm{r+1}{1}{q}-1\right],
  $$
  where $a,b\in\mathbb{N}_0$ with $b\le q-2$ and $a\le r-1$. If $n\le rq^{r+1}$, then all other cases can be realized.
\end{theorem}
\begin{proof}
  Combinations of Example~\ref{example_r_flat} and Example~\ref{example_affine_r_space} give a construction of a $q^r$-divisible 
  set with cardinality $n$ iff there exists an integer $m\in\mathbb{N}_{>0}$ with $(m-1)\cdot \gaussm{r+1}{1}{q}\le n
  \le (m-1)\cdot \gaussm{r+1}{1}{q}+\left\lfloor\frac{m-1}{q-1}\right\rfloor$. It remains to exclude the stated cases. We 
  prove by induction on $r$, set $\Delta=q^r$, and write $n=(m-1)\gaussm{r+1}{1}{q}+x$, where $a+1\le x\le \gaussm{r+1}{1}{q}-1$ and 
  $m-1=a(q-1)+b$ for integers $0\le b\le q-2$, $0\le a\le r-1$.
  
  The induction start $r=1$ is given by Theorem~\ref{thm_exclusion_r_1_to_ovoid}. 
  
  Now, assume $r\ge 2$. From the induction hypothesis we conclude that for $0\le b'\le q-2$, $0\le a'\le r-2$ we have 
  $$n'\notin \left[(a'(q-1)+b')\gaussm{r}{1}{q}+a'+1,(a'(q-1)+b'+1)\gaussm{r}{1}{q}-1\right]$$ for the cardinality 
  $n'$ of a $q^{r-1}$-divisible set. If $a\le r-2$ and $x\le \gaussm{r}{1}{q}-1$, then $b'=b$, $a'=a$ yields 
  $\mathcal{T}(\mathcal{C})\subseteq \{u,u+\Delta,\dots,u+(m-2)\Delta\}$ for $u=\Delta+(m-1)\gaussm{r}{1}{q}+x$. We compute
  $$
    (q-1)u=q^{r+1}-q^r+(m-1)q^r-(m-1)+(q-1)x\overset{x\ge a+1}{\ge} (m-2)q^r+q^{r+1}>(m-2)\Delta,
  $$
  so that we can apply Corollary~\ref{cor_nonexistence_arithmetic_progression_2}. If $a=r-1$ and $a+1\le x\le \gaussm{r}{1}{q}-1$, 
  then $b'=b$, $a'=a-1$ yields $\mathcal{T}(\mathcal{C})\subseteq \{u,u+\Delta,\dots,u+(m-1)\Delta\}$ for 
  $u=(m-1)\gaussm{r}{1}{q}+x$. We compute $(q-1)u=(m-1)q^r-(m-1)+x(q-1)> (m-1)\Delta$ using $x\ge a+1$, so that we can 
  apply Corollary~\ref{cor_nonexistence_arithmetic_progression_2}. Thus, we can assume $\gaussm{r}{1}{q}\le x\le \gaussm{r+1}{1}{q}-1$ 
  in the remaining part. Additionally we have $m\le r(q-1)$.
  
  We aim to apply Lemma~\ref{lemma_negative_tau}. Due to Lemma~\ref{lemma_m_bound_rhs} for the upper bound of the interval it suffices 
  to show  
  $$
    r(q-1)\le  \left\lfloor\sqrt{(q-1)q\Delta}-q+\frac{3}{2}\right\rfloor.
  $$
  For $q=2$ the inequality is equivalent to $r\le \left\lfloor\sqrt{2^{r+1}}-\frac{1}{2}\right\rfloor$, which is valid for $r\ge 2$. 
  Since the right hand side is larger then $(q-1)(\sqrt{\Delta}-1)$, it suffices to show $q^{r/2}-1\ge r$, which is valid for $q\ge 3$ 
  and $r\ge 2$. For the left hand side of the interval if suffices to show
  $$
    (q-1)(n-m\Delta)-(m-q/2)\Delta+\frac{1}{2}\ge -\frac{1}{2}\cdot\sqrt{(\Delta q)^2-(4m-2)\Delta q+1},
  $$
  which can be simplified to 
  $$
    \Delta q+2m-3-2(q-1)x\le \sqrt{(\Delta q)^2-(4m-2)\Delta q+1}
  $$
  using $n=(m-1)\gaussm{r+1}{1}{q}+x$. Since $(q-1)x\ge q^r-1$ and $m\le r(q-1)$ it suffices to show
  \begin{equation}
    \label{ie_numeric_1}
    -\Delta^2+2rq\Delta-2r\Delta-\Delta-r+r^2q-r^2\le 0.
  \end{equation} 
  For $q=2$ this inequality is equivalent to $-2^{2r}+r2^{r+1}+r^2-2-2^r\le 0$, which is valid for $r\ge 2$. For $r=2$ 
  Inequality~(\ref{ie_numeric_1}) is equivalent to $-q^4+4q^3-4q^2-q^2+4q-6$, which is valid for $q\in\{2,3\}$ and $q\ge 4$. 
  For $q\ge 3$ and $r\ge 3$ we have $\Delta\ge 3rq$, so that Inequality~(\ref{ie_numeric_1}) is satisfied. 
\end{proof}

In other words Theorem~\ref{thm_exclusion_q_r} says that the cardinality $n$ of a $q^r$-divisible set 
can be written as $a\gaussm{r+1}{1}{q}+bq^{r+1}$ for some $a,b\in\mathbb{N}_0$ if $n\le rq^{r+1}$.

For $r=1$ the required example of cardinality $n=rq^{r+1}$ is given by an ovoid for $q\ge 3$ and 
by a projective base for $q=2$. For $r=2$ our current knowledge is rather sparse. For $q=2$ we know three 
isomorphism types of $2^2$-divisible sets of cardinality $n=17$. For $q=3$ we know an example of cardinality $n=55$ given by 
a shortening the Hill cap. For $q=4$ we do not know a $4^2$-divisible set of cardinality $n=129$ so far.


\section{Non-existence of $q^r$-divisible subsets of $\aspace$}
\label{sec_exclusion}

In this section we want to apply the analytical tools developed in Section~\ref{sec_tools_to_exclude} to characterize the set of 
possible length of $q^r$-divisible sets for small parameters. Of course, we will also need the constructions from Section~\ref{sec_constructions}. 
Concrete numerical results for $q\in\{2,3,4,5,7,8,9\}$ are presented in Subsection~\ref{subsec_complete_picture_q_2}-\ref{subsec_complete_picture_q_9}, 
respectively. Those results include the determination of the Frobenius numbers $\frobenius(2,1)=2$, $\frobenius(2,2)=13$, $\frobenius(2,3)=59$, 
$\frobenius(3,1)=7$, and $\frobenius(4,1)=19$. Partially going beyond the methods of Section~\ref{sec_tools_to_exclude}, we can also apply the (integer) linear programming method 
described in Subsection~\ref{subsec_linear_programming_method}. In Appendix~\ref{app_exclusion_lists} we tabulate some corresponding numerical data.  


\subsection{Possible $\mathbf{q^r}$-divisible sets for $\mathbf{q=2}$}
\label{subsec_complete_picture_q_2}
\begin{lemma}
  \label{lemma_picture_q_2_r_1}
  Let $\mathcal{C}\subseteq\aspace$ be non-trivial and $2^1$-divisible of cardinality $n$, then $n\ge 3$ and all cases can be realized. 
\end{lemma}
\begin{proof}
  The values $n\in\{1,2\}$ are excluded by Theorem~\ref{thm_exclusion_r_1_to_ovoid}. 
  Examples of cardinalities $3$, $4$, and $5$ are given by examples \ref{example_r_flat}, \ref{example_affine_r_space}, and 
  \ref{example_projective_basis}, respectively. Since every integer larger or equal to $6$ is a positive integer linear combination of 
  $3$ and $4$, Lemma~\ref{lemma_add_configurations} provides the corresponding realizations.
\end{proof}

\begin{lemma}
  \label{lemma_picture_q_2_r_2}
  Let $\mathcal{C}\subseteq\aspace$ be non-trivial and $2^2$-divisible of cardinality $n$, then $n\in\{7,8\}$ or $n\ge 14$ 
  and all mentioned cases can be realized. 
\end{lemma}
\begin{proof}
  The cases $1\le n\le 6$ and $9\le n\le 13$ are excluded by Theorem~\ref{thm_exclusion_q_r}.
  
  The \textit{base examples} are given by
  \begin{itemize}
    \item $n=7$: Example~\ref{example_r_flat};
    \item $n=8$: Example~\ref{example_affine_r_space} or Corollary~\ref{corollary_surgery_1};
    \item $15\le n\le 20$: see Corollary~\ref{corollary_surgery_2};
  \end{itemize}
  so that Lemma~\ref{lemma_add_configurations} provides examples for the mentioned cases.  
\end{proof}

\begin{lemma}
  \label{lemma_picture_q_2_r_3}
  Let $\mathcal{C}\subseteq\aspace$ be non-trivial and $2^3$-divisible of cardinality $n$, then 
  $$n\in\{15,16,30,31,32,45,46,47,48,49,50,51\}$$ 
  or $n\ge 60$ and all cases can be realized. 
\end{lemma}
\begin{proof}
  The cases $1\le n\le 14$, $17\le n\le 29$, and $33\le n\le 44$ are excluded by Theorem~\ref{thm_exclusion_q_r}.
  The case $n=52$ is excluded by Corollary~\ref{cor_implication_fourth_mac_williams} with $t=3$, see also Lemma~\ref{lemma_no_8_div_52} or
  Lemma~\ref{lemma_reduced_lp_certificate} which is able to compute a similar \textit{certificate} or \textit{proof} automatically. 
  The cases $53\le n\le 58$ are excluded by Lemma~\ref{lemma_hyperplane_types_arithmetic_progression} using $m=4$. 
  The special case $n=59$ is treated in \cite{no59}. 
  
  The \textit{base examples} are given by
  \begin{itemize}
    \item $n=15$: Example~\ref{example_r_flat};
    \item $n=16$: Example~\ref{example_affine_r_space};
    \item $n=49$: see Subsection~\ref{subsec_computer_search_hole_configuration_q_2_r_3};
    \item $n=50$: see Subsection~\ref{subsec_computer_search_hole_configuration_q_2_r_3};    
    \item $n=51$: 
                  Example CY1 in \cite{calderbank1986geometry};
    \item $63\le n\le 72$: see Corollary~\ref{corollary_surgery_2};
    \item $n=73$: see the two-weight codes in \cite{kohnert2007constructing_twoweight};    
    \item $n=74$: see Subsection~\ref{subsec_computer_search_hole_configuration_q_2_r_3} for an example with $k=12$;
  \end{itemize}
  so that Lemma~\ref{lemma_add_configurations} provides examples for the mentioned cases that are not printed in bold face.
\end{proof}

\begin{remark}
  \label{remark_ILP_LP_RECURSIVE}
  The exclusion of $n=33$ for $8$-divisible sets is an interesting example. In the proof of 
  Theorem~\ref{thm_exclusion_q_r} we have used Lemma~\ref{lemma_average} to conclude the existence of a 
  hyperplane with either $1$ or $9$ holes. Both cases are impossible due to Lemma~\ref{lemma_picture_q_2_r_2}.
  If hyperplanes with $1$ hole, i.e., weight $32$, are excluded, then the polyhedron of the linear programming method 
  is empty. (Of course, the same is true if hyperplanes containing $1$ or $9$ holes are forbidden.) If hyperplanes 
  containing $1$ or $9$ holes (allowed weights: $0$, $8$, $16$, $24$, $32$) are included in the formulation, then there 
  are rational non-negative solutions of the MacWilliams identities. However, this system does not admit a non-negative 
  integer solution.  
\end{remark}

\begin{lemma}
  \label{lemma_picture_q_2_r_4}
  Let $\mathcal{C}\subseteq\aspace$ be non-trivial and $2^4$-divisible of cardinality $n$, then 
  \begin{eqnarray*}
    n&\in&\{31,32,62,63,64,93,\dots,96,124,\dots,131,155,\dots,165,185,\dots,199,\\&&215,\dots,234,244,\dots,309\}
  \end{eqnarray*} 
  or $n\ge 310$ and all cases, possibly except 
  \begin{eqnarray*}
    n &\in& \{129,130,131,163,164,165,185,215,216,232,233,244,245,246,247,\\
               && 274,275,277,278,306,309\},
  \end{eqnarray*}
  can be realized. 
\end{lemma}
\begin{proof}
  The cases $1 \le n \le 30$, $33 \le n \le 61$, $65 \le n \le 92$, and $97 \le n \le 123$ are excluded by 
  Theorem~\ref{thm_exclusion_q_r}. The cases $133 \le n \le 154$, $167 \le n \le 184$, 
  $201 \le n \le 214$, and $236 \le n \le 243$ are excluded by Lemma~\ref{lemma_hyperplane_types_arithmetic_progression} 
  using $m=5$, $m=6$, $m=7$, and $m=8$, respectively. The cases $n\in\{132,166,200,235\}$ are excluded by 
  Corollary~\ref{cor_implication_fourth_mac_williams} with $t=4,\dots,7$, respectively. 
  
  The \textit{base examples} are given by
  \begin{itemize}
    \item $n=31$: Example~\ref{example_r_flat};
    \item $n=32$: Example~\ref{example_affine_r_space};
    \item $n=161$: a $[161,10,64]_2$-code is given in Subsection~\ref{subsec_computer_search_hole_configuration_q_2_r_4};
    \item $n=162$: a $[162,10,32]_2$-code is given in Subsection~\ref{subsec_computer_search_hole_configuration_q_2_r_4};
    \item $n=195$: a $[195,10,80]_2$-code is given in Subsection~\ref{subsec_computer_search_hole_configuration_q_2_r_4}; 
    \item $n=196$: BY construction described in \cite{bierbrauer1997family};
    \item $n=197$: a $[197,10,80]_2$-code is given in Subsection~\ref{subsec_computer_search_hole_configuration_q_2_r_4};
    \item $n=198$: an example of dimension $k=10$ was mentioned in \cite{kohnert2007constructing_twoweight};
    \item $n=199$: there is an optimal $[199, 11, 96]$-code; 
    \item $n=231$: an example of dimension $k=10$ was mentioned in \cite{kohnert2007constructing_twoweight,dissett2000combinatorial};
    \item $n=234$: construction of \cite[Theorem 6.1]{calderbank1986geometry} on Example FE3 over $\mathbb{F}_4$;
    \item $255\le n\le 272$: see Corollary~\ref{corollary_surgery_2};
    \item $n=273$: construction bases on quasi-cyclic codes, see \cite{chen2006constructions};
    \item $n=276$: belongs to the family of RT5d codes, see \cite{calderbank1986geometry};
  \end{itemize}
  so that Lemma~\ref{lemma_add_configurations} 
  provides examples for the mentioned cases that are not printed in bold face.  
\end{proof}

\begin{lemma}
  \label{lemma_picture_q_2_r_5}
  Let $\mathcal{C}\subseteq\aspace$ be non-trivial and $2^5$-divisible of cardinality $n$, then 
  \begin{eqnarray*}
    n&\in&\{63,64,126,127,128,189,\dots,192,252,\dots,256,315,\dots,320,\textbf{321},\textbf{322},\textbf{323},\\ 
    &&378,\textbf{385},\dots,\textbf{389},441,\dots,448,\textbf{449},\dots,\textbf{454},455,\textbf{503},504,\dots,512,\textbf{513},\\ 
    &&\dots,\textbf{517},518,519,\textbf{520},\textbf{566},567,\dots,576,\textbf{577},\dots,\textbf{580},581,582,583,\textbf{584},\\ 
    &&\textbf{585},\textbf{586},\textbf{628},\textbf{629},630,\dots,640,\textbf{641},\textbf{642},643,\dots,647,\textbf{648},\dots,\textbf{651},\textbf{691},\\ 
    &&\textbf{692},693,\dots,704,\textbf{705},706,\dots,711,\textbf{712},\dots,\textbf{717},\textbf{753},\textbf{754},\textbf{755},756,\dots,\\ 
    &&775,\textbf{776},\dots,\textbf{779},780,\textbf{781},\dots,\textbf{783},\textbf{815},\dots,\textbf{818},819,\dots,839,\textbf{840},\dots,\\ 
    &&\textbf{842},843,845,\textbf{846},\dots,\textbf{849},\textbf{877},\dots,\textbf{881},882,\dots,903,\textbf{904},\textbf{905},906,\dots,\\ 
    && 910,\textbf{911},\dots,\textbf{916},\textbf{938},\dots,\textbf{944},945,\dots,967,\textbf{968},969,\dots,975,\textbf{976},\dots,\\ 
    &&\textbf{984},\textbf{998},\dots,\textbf{1007},1008,\dots,1056,\textbf{1057},\dots,\textbf{1070},1071,\dots,1120,\textbf{1121},\\ 
    &&\dots\textbf{1133},1134,\dots,1184,\textbf{1185}\}
  \end{eqnarray*} 
  or $n\ge 1186$ and all mentioned cases with $n\le 1185$, that are not in bold face, can be realized. 
\end{lemma}
\begin{proof}
  The cases $1 \le n \le 62$, $65 \le n \le 125$, $129 \le n \le 188$, $193\le n\le 251$, and $257 \le n \le 314$ are excluded by 
  Theorem~\ref{thm_exclusion_q_r}. The cases $326 \le n \le 377$, $391 \le n \le 440$, 
  $457 \le n \le 502$, $522\le n\le 565$, $588\le n\le 627$, $653\le n\le 690$, $719\le n\le 752$, $785\le n\le 814$, $851\le n\le 876$, 
  $918\le n\le 937$, and $986 \le n \le 997$ are excluded by Lemma~\ref{lemma_hyperplane_types_arithmetic_progression} 
  using $m=6,\dots,16$, respectively. 
  The cases $n\in\{325,390,456,521,587,652,718,784,850,917,985\}$ are excluded by 
  Corollary~\ref{cor_implication_fourth_mac_williams} with $t=5,\dots,15$, respectively. The case $n=324$ is 
  excluded by Lemma~\ref{lemma_average} and Lemma~\ref{lemma_picture_q_2_r_4}.
  
  The \textit{base examples} are given by
  \begin{itemize}
    \item $n=63$: Example~\ref{example_r_flat};
    \item $n=64$: Example~\ref{example_affine_r_space};
    \item $n=455$: belongs to the family of CY1 codes, see \cite{calderbank1986geometry};
    \item $n=780$: BY construction described in \cite{bierbrauer1997family};
    \item $n=845$: an example of dimension $k=12$ was mentioned in \cite{kohnert2007constructing_twoweight};
    \item $n=975$: an example of dimension $k=12$ was mentioned in \cite{kohnert2007constructing_twoweight};
    \item $1023\le n\le 1056$: see Corollary~\ref{corollary_surgery_2};
    \item $n=1105$: an example of dimension $k=12$ was mentioned in \cite{kohnert2007constructing_twoweight};
    \item $n=1170$: an example of dimension $k=12$ was mentioned in \cite{kohnert2007constructing_twoweight};
  \end{itemize}
  so that Lemma~\ref{lemma_add_configurations} 
  provides examples for the mentioned cases that are not printed in bold face.  
\end{proof}


\subsection{Possible $\mathbf{q^r}$-divisible sets for $\mathbf{q=3}$}
\label{subsec_complete_picture_q_3}
\begin{lemma}
  \label{lemma_picture_q_3_r_1}
  Let $\mathcal{C}\subseteq\aspace$ be non-trivial and $3^1$-divisible of cardinality $n$, then $n=4$ or $n\ge 8$ and all 
  cases can be realized. 
\end{lemma}
\begin{proof}
  The values $1\le n\le 3$ and $5\le n\le 7$ are excluded by Theorem~\ref{thm_exclusion_r_1_to_ovoid}.
  
  The \textit{base examples} are given by
  \begin{itemize}
    \item $n=4$: Example~\ref{example_r_flat};
    \item $n=9$: Example~\ref{example_affine_r_space};
    \item $n=10$: Example~\ref{example_ovoid};
    \item $n=11$: An example for dimension $k=5$ is given as example RT6 in 
                  \cite{calderbank1986geometry}, see also \cite{delsarte1971twoweight_srg}. (Dual of the $[11,6,5]$ ternary Golay code.);  
  \end{itemize}
  so that Lemma~\ref{lemma_add_configurations} provides examples for the mentioned cases.
\end{proof}

\begin{lemma}
  \label{lemma_picture_q_3_r_2}
  Let $\mathcal{C}\subseteq\aspace$ be non-trivial and $3^2$-divisible of cardinality $n$, then 
  \begin{eqnarray*}
    n&\in&\{13,26,27,39,40,52,\dots,56,65,\dots,70,77,\dots,85,90,\dots,128\}
  \end{eqnarray*} 
  or $n\ge 129$ and all cases, possibly except 
  \begin{eqnarray*}
    n &\in& \{70,77,99,100,101,102,113,114,115,128\},
  \end{eqnarray*}
  can be realized. 
\end{lemma}
\begin{proof}
  The cases $1 \le n \le 12$, $15 \le n \le 25$, $29 \le n \le 38$ and  $43 \le n \le 51$ are excluded by Theorem~\ref{thm_exclusion_q_r}.  
  The case $57 \le n \le 64$, $72 \le n \le 76$, and $87 \le n \le 88$ are excluded by Lemma~\ref{lemma_hyperplane_types_arithmetic_progression} 
  using $m=5,\dots,7$, respectively. The cases $n\in\{71,86\}$ are excluded by 
  Corollary~\ref{cor_implication_fourth_mac_williams} with $t\in\{5,6\}$, respectively.  
  The case $n=89$ is excluded by the linear programming method. This case is special since no certificate 
  could be computed based on Lemma~\ref{lemma_reduced_lp_certificate}. A separate argument is provided in 
  Lemma~\ref{lemma_exclusion_n_89_q_3_Delta_9}. 
   
  The \textit{base examples} are given by
  \begin{itemize}
    \item $n=13$: Example~\ref{example_r_flat};
    \item $n=27$: Example~\ref{example_affine_r_space}; 
    \item $n=55$: see \cite{gulliver1996two} (two-weight code);
    \item $n=56$: belongs to the family of FE2 codes, see \cite{calderbank1986geometry}, see also \cite{kohnert2007constructing_twoweight} (two-weight code);
    \item $n=84$: BY construction described in \cite{bierbrauer1997family}, see also \cite{kohnert2007constructing_twoweight,gulliver1996two} (two-weight code); 
                  this is an optimal $[84, 6, 54]_3$ code
    \item $n=85$: optimal $[85,7,54]_3$ code; 
    \item $n=90$: optimal $[90,8,54]_3$-code; 
    \item $n=98$: see \cite{kohnert2007constructing_twoweight} (two-weight code);
    \item $n=127$: optimal $[127,7,81]_3$-code; 
    \item $n=141$: optimal $[141,7,90]_3$-code; 
  \end{itemize}
  so that Lemma~\ref{lemma_add_configurations} 
  provides examples for the mentioned cases that are not printed in bold face.  
\end{proof}

\begin{remark}
  For $n=89$ the first four MacWilliams identities permit non-negative rational solutions for dimensions $9\le k\le 89$. 
  Adding the fifth MacWilliams identity, the corresponding polyhedron gets empty.
\end{remark} 

\begin{lemma}
  \label{lemma_exclusion_n_89_q_3_Delta_9}
  No $3^2$-divisible set of cardinality $89$ exists.
\end{lemma}
\begin{proof}
  We set $x=3^{k-4}$, $y=3^{k-4}\cdot A_3^\perp$, and $z=3^{k-4}\cdot A_4^\perp$. Solving the first five MacWilliams identities 
  for $A_{9}$, $A_{54}$, $A_{63}$, $x$, and $y$ yields the equation
  $$
    99630A_9+121905A_{18}+99873A_{27}+60021A_{36}+22275A_{45}+22518A_{72}+61236A_{81}+z=0,
  $$
  so that $A_9=A_{18}=A_{27}=A_{36}=A_{45}=A_{72}=A_{81}=z=0$. With that, the equation system has the unique solution 
  $x=189$, $y=33642$, $A_{54}=6230$, and $A_{63}=9078$. However, $189$ is not a power of three, but $x=3^{k-4}$.
\end{proof}


\subsection{Possible $\mathbf{q^r}$-divisible sets for $\mathbf{q=4}$}
\label{subsec_complete_picture_q_4}
\begin{lemma}
  \label{lemma_picture_q_4_r_1}
  Let $\mathcal{C}\subseteq\aspace$ be non-trivial and $4^1$-divisible of cardinality $n$, then 
  $$n\in\{5,10,15,16,17\}$$ or $n\ge 20$ and all cases can be realized. 
\end{lemma}
\begin{proof}
  The values $1\le n\le 4$, $6\le n\le 9$, and $11\le n\le 14$ are excluded by Theorem~\ref{thm_exclusion_r_1_to_ovoid}.
  The cases $n\in\{18,19\}$ are excluded by Lemma~\ref{lemma_hyperplane_types_arithmetic_progression} using $m=4$.
  
  The \textit{base examples} are given by
  \begin{itemize}
    \item $n=5$: Example~\ref{example_r_flat};
    \item $n=16$: Example~\ref{example_affine_r_space};
    \item $n=17$: Example~\ref{example_ovoid};
    \item $21\le n\le 24$: apply the construction of Remark~\ref{remark_surgery_baer};
  \end{itemize}
  so that Lemma~\ref{lemma_add_configurations} provides examples for the mentioned cases.
\end{proof}

\begin{lemma}
  \label{lemma_picture_q_4_r_2}
  Let $\mathcal{C}\subseteq\aspace$ be non-trivial and $4^2$-divisible of cardinality $n$, then 
  \begin{eqnarray*}
    n&\in&\{21,42,63,64,84,85,105,106,126,127,128,\textbf{129},147,148,149,\textbf{150},\textbf{151},168,\\ 
    &&169,170,171,\textbf{172},\textbf{173},189,\dots,192,\textbf{193},\textbf{194},\textbf{195},210,\dots,214,\textbf{215},\dots,\\ 
    &&\textbf{217},231,\dots,235,\textbf{236},\dots,\textbf{239},\textbf{251},252,\dots,257,\textbf{258},\textbf{259},260,\textbf{261},\textbf{272},\\ 
    &&273,\dots,278,\textbf{279},\textbf{280},281,\textbf{282},\textbf{283},\textbf{293},294,\dots,300,\textbf{301},302,303,304,\\ 
    &&\textbf{305},\textbf{313},\textbf{314},315,\dots,321,\textbf{322},323,\dots,325,\textbf{326},327,328,\textbf{333},\textbf{334},\textbf{335}\}
  \end{eqnarray*} 
  or $n\ge 336$ and all mentioned cases with $n\le 335$, that are not in bold face, can be realized.
\end{lemma}
\begin{proof}
  The cases $1 \le n \le 20$, $22 \le n \le 41$, $44 \le n \le 62$, $66 \le n \le 83$, $87 \le n \le 104$, and $109 \le n \le 125$ 
  are excluded by Theorem~\ref{thm_exclusion_q_r}. The case $131\le n\le 146$, $153\le n\le 167$, $174\le n\le 188$, $196\le n\le 209$, 
  $218\le n\le 230$, $240\le n\le 250$, $262\le n\le 271$, $284\le n\le 292$, $306\le n\le 312$, and $329 \le n \le 332$ are excluded by 
  Lemma~\ref{lemma_hyperplane_types_arithmetic_progression} using $m=7,\dots,16$, respectively. 
  The cases $n\in\{130,152\}$ are excluded by Corollary~\ref{cor_implication_fourth_mac_williams} with $t\in\{6,7\}$, respectively.  
  
  The \textit{base examples} are given by
  \begin{itemize}
    \item $n=21$: Example~\ref{example_r_flat};
    \item $n=64$: Example~\ref{example_affine_r_space};
    \item $n=85+43\cdot j$ for $0\le j\le 17$: see Corollary~\ref{corollary_surgery_2};
    \item $n=171$: Corollary \ref{corollary_q_2_n_17_k_8_generalization} based on Construction~\ref{construction_q_2_n_17_k_8_generalization}; see also Corollary~\ref{corollary_surgery_2}
    \item $n=257$: Construction~\ref{construction_qt4p1}; see also Corollary~\ref{corollary_surgery_2}; 
    \item $n=260$: BY construction described in \cite{bierbrauer1997family};
    \item $n=303$: belongs to the family of CY2 codes, see \cite{calderbank1986geometry}
    \item $n=304$: complement of a $[1061,6]_4$ two-weight code, which belongs to the family of CY2 codes, see \cite{calderbank1986geometry}
  \end{itemize}
  so that Lemma~\ref{lemma_add_configurations} 
  provides examples for the mentioned cases that are not printed in bold face.  
\end{proof}


\subsection{Possible $\mathbf{q^r}$-divisible sets for $\mathbf{q=5}$}
\label{subsec_complete_picture_q_5}

\begin{lemma}
  \label{lemma_picture_q_5_r_1}
  Let $\mathcal{C}\subseteq\aspace$ be non-trivial and $5^1$-divisible of cardinality $n$, then 
  $$
    n\in\{6,12,18,24,25,26,30,31,32,36,\dots,40\}
  $$
  or $n\ge 41$ and all cases, possibly except $n=40$, can be realized. 
\end{lemma}
\begin{proof}
  The values $1\le n\le 5$, $7\le n\le 11$, $13\le n\le 17$, and $19\le n\le 23$ are excluded by Theorem~\ref{thm_exclusion_r_1_to_ovoid}.
  The cases $27\le n\le 29$ and $34\le n\le 35$ are excluded by Lemma~\ref{lemma_hyperplane_types_arithmetic_progression} 
  using $m=5$ and $m=6$, respectively. 
  The case $n=33$ is excluded by Corollary~\ref{cor_implication_fourth_mac_williams} with $t=5$ respectively.
  
  The \textit{base examples} are given by
  \begin{itemize}
    \item $n=6$: Example~\ref{example_r_flat};
    \item $n=25$: Example~\ref{example_affine_r_space};
    \item $n=26$: Example~\ref{example_ovoid};
    \item $n=39$: there exists a $[39, 4; 30]_5$ two-weight code, see \cite{dissett2000combinatorial}; 
    \item $n=41$: there exists a $[41,5;25]_5$ code, see Subsection~\ref{subsec_computer_search_hole_configuration_q_5_r_1};
    \item $n=46$: there exists a $[46,5;25]_5$ code, see Subsection~\ref{subsec_computer_search_hole_configuration_q_5_r_1};      
  \end{itemize}
  so that Lemma~\ref{lemma_add_configurations} 
  provides examples for the mentioned cases that are not printed in bold face.   
\end{proof}


\subsection{Possible $\mathbf{q^r}$-divisible sets for $\mathbf{q=7}$}
\label{subsec_complete_picture_q_7}

\begin{lemma}
  \label{lemma_picture_q_7_r_1}
  Let $\mathcal{C}\subseteq\aspace$ be non-trivial and $7^1$-divisible of cardinality $n$, then
  \begin{eqnarray*}
    n&\in&\{8,16,24,32,40,48,49,50,56,57,58,64,65,66,72,\dots,75,80,\dots,83,\\ &&88,\dots,92,95,\dots,167\}
  \end{eqnarray*} 
  or $n\ge 168$ and all cases, possibly except 
  \begin{eqnarray*}
    n&\in&\{75,83,91,92,95,101,102,103,109,110,111,117,118,119,125,126,127,\\ &&133,134,135,142,143,151,159,167\},  
  \end{eqnarray*}
  can be realized. 
\end{lemma}
\begin{proof}
  The values $1\le n\le 7$, $9\le n\le 15$, $17\le n\le 23$, $25\le n\le 31$, $33\le n\le 39$, and $41\le n\le 47$ are excluded by 
  Theorem~\ref{thm_exclusion_r_1_to_ovoid}. The cases $51\le n\le 55$, $59\le n\le 63$, $68\le n\le 71$, $76\le n\le 79$, 
  $85\le n\le 87$, and $93\le n\le 94$ are excluded by Lemma~\ref{lemma_hyperplane_types_arithmetic_progression} 
  using $m=7,\dots,12$, respectively. The cases $n=67$ and $n=84$ are excluded by Corollary~\ref{cor_implication_fourth_mac_williams} 
  with $t=8$ and $t=10$, respectively.
  
  The \textit{base examples} are given by
  \begin{itemize}
    \item $n=8$: Example~\ref{example_r_flat};
    \item $n=49$: Example~\ref{example_affine_r_space};
    \item $n=50$: Example~\ref{example_ovoid};
    \item $n=141$: there exists a $[141,4;112]_7$ code, see Subsection~\ref{subsec_computer_search_hole_configuration_q_7_r_1};
    \item $n=175$: belongs to the family of FE1 codes, see \cite{calderbank1986geometry};    
  \end{itemize}
  so that Lemma~\ref{lemma_add_configurations} provides examples for the mentioned cases that are not printed in bold face.  
\end{proof}


\subsection{Possible $\mathbf{q^r}$-divisible sets for $\mathbf{q=8}$}
\label{subsec_complete_picture_q_8}

\begin{lemma}
  \label{lemma_picture_q_8_r_1}
  Let $\mathcal{C}\subseteq\aspace$ be non-trivial and $8^1$-divisible of cardinality $n$, then
  \begin{eqnarray*}
    n&\in&\{9,18,27,36,45,54,63,64,65,72,73,74,81,82,83,90,\dots,93,99,\dots,102,\\ 
    &&108,\dots,111,117,\dots,121,126,\dots,130,134,\dots,140,143,\dots,251\}
  \end{eqnarray*} 
  or $n\ge 252$ and all cases, possibly except 
  \begin{eqnarray*}
    n &\in& \{93,102,111,120,121,134,140,143,149,\dots,152,158,\dots,161,167,\dots,\\ 
    &&170,176,\dots,179,185,\dots,188,196,197,205,206,214,215,223,224,232,\\ &&233,241,242,250,251\},  
  \end{eqnarray*}
  can be realized. 
\end{lemma}
\begin{proof}
  The values $1\le n\le 8$, $10\le n\le 17$, $19\le n\le 26$, $28\le n\le 35$, $37\le n\le 44$, $46\le n\le 53$, and 
  $55\le n\le 62$ are excluded by Theorem~\ref{thm_exclusion_r_1_to_ovoid}. The cases $66\le n\le 71$, $75\le n\le 80$, 
  $84\le n\le 89$, $94\le n\le 98$, $103\le n\le 107$, $113\le n\le 116$, $122\le n\le 125$, $132\le n\le 133$,    
  and $141\le n\le 142$ are excluded by Lemma~\ref{lemma_hyperplane_types_arithmetic_progression} 
  using $m=8,\dots,16$, respectively. The cases $n=112$ and $n=131$ are excluded by Corollary~\ref{cor_implication_fourth_mac_williams} 
  with $t=12$ and $t=14$, respectively.
    
  Examples of cardinalities $9$, $64$, and $65$ are given by examples \ref{example_r_flat}, \ref{example_affine_r_space}, and 
  \ref{example_ovoid}, respectively. All other cases can be obtained using Lemma~\ref{lemma_add_configurations}.
\end{proof}


\subsection{Possible $\mathbf{q^r}$-divisible sets for $\mathbf{q=9}$}
\label{subsec_complete_picture_q_9}

\begin{lemma}
  \label{lemma_picture_q_9_r_1}
  Let $\mathcal{C}\subseteq\aspace$ be non-trivial and $9^1$-divisible of cardinality $n$, then
  \begin{eqnarray*}
    n&\in&\{10,20,30,40,50,60,70,80,81,82,90,91,92,100,101,102,110,111,112,120,\\ 
    &&\dots,123,130,\dots,133,140,\dots,143,150,\dots,154,160,\dots,164,170,\dots,175,\\ 
    && 179,\dots,185,189,\dots,196,199,\dots,205,\dots,359\}
  \end{eqnarray*} 
  or $n\ge 360$ and all cases, possibly except 
  \begin{eqnarray*}
    n &in& \{123,133,143,153,154,175,179,185,189,195,196,199,206,\dots,209,216,\\
    &&\dots,219,226,\dots,229,236,\dots,239,247,248,249,257,258,259,267,268,\\ 
    &&269,277,278,279,288,289,298,299,308,309,318,319,329,339,349,359\},  
  \end{eqnarray*}
  can be realized. 
\end{lemma}
\begin{proof}
  The values $1\le n\le 9$, $11\le n\le 19$, $21\le n\le 29$, $31\le n\le 39$, $41\le n\le 49$, $51\le n\le 59$, 
  $61\le n\le 69$, and $71\le n\le 79$ are excluded by Theorem~\ref{thm_exclusion_r_1_to_ovoid}. The cases $83\le n\le 89$, 
  $93\le n\le 99$, $103\le n\le 109$, $114\le n\le 119$, $124\le n\le 129$, $134\le n\le 139$, $145\le n\le 149$, 
  $155\le n\le 159$, $165\le n\le 169$, $176\le n\le 178$, $186\le n\le 188$, and $197\le n\le 198$ are excluded by 
  Lemma~\ref{lemma_hyperplane_types_arithmetic_progression} using $m=9,\dots,20$, respectively. 
  
  The cases $n=113$ and $n=144$ are excluded by Corollary~\ref{cor_implication_fourth_mac_williams} 
  with $t=11$ and $t=14$, respectively.
    
  Examples of cardinalities $10$, $81$, and $82$ are given by examples \ref{example_r_flat}, \ref{example_affine_r_space}, and 
  \ref{example_ovoid}, respectively. An example of a $[205, 4]_9$ two-weight code, with weights $180$ and $189$, can be obtained from 
  the complement of a  $[615, 4]_9$ two-weight code, with weights $540$ and $549$, which is stated as Example CY2 
  in \cite{calderbank1986geometry}. An example with $n=369$ and $k=4$ is stated as Example FE1 in \cite{calderbank1986geometry}. 
  Length $369$ can also be obtained as $205+82+82$. All other cases can be obtained using Lemma~\ref{lemma_add_configurations}.
\end{proof}


\section{Conclusion and future research}
\label{sec_conclusion}

We have presented the current knowledge on the set of integers that can be the length of a projective $q^r$-divisible code or, equivalently, 
a $q^r$-divisible set. Standard techniques as the linear programming method are applied onto this problem. Nevertheless there are 
some numerical and analytical results the problem seems to be quite tough. While the corresponding problem for $q^r$-divisible multisets 
is completely solved, it remains open if we either allow the exponent $r$ to be fractional or also want to specify the desired dimension 
$k$ as a further parameter. In between projective and general linear codes there is space for further refinements. I.e., in some applications, 
like e.g.\ packing or covering problems, the maximum number of occurrences of a point in a $q^r$-divisible multiset is bounded by some 
number $\lambda$. 

Our list of construction of $q^r$-divisible sets is rather ad hoc and deserves to be extended. Clearly, more tailored computer search 
should unveil more interesting examples. A very interesting question, not touched here, is whether a $q^r$-divisible set can be realized 
by some vector space partition. So far we are not aware of negative results, besides cases when the dimension of the ambient space is \textit{too} 
small. As an example, there are exactly three non-isomorphic $2^2$-divisible sets of cardinality $17$ and indeed each of them can occur 
as the set of holes of a partial $(3-1)$-spread of the maximum possible cardinality $34$ in $\mathbb{F}_2^8$, see \cite{honold2018partial}. 
Another such example for plane spreads in $\mathbb{F}_2^7$ of cardinality $16$ can be found in \cite{honold2019classification}. Besides 
\cite{ubt_eref40887}, not much is known on classification results for all $q^r$-divisible sets for small parameters. Enumeration algorithms 
for linear codes like \texttt{Q-Extension} \cite{bouyukliev2007q}, \texttt{QextNewEdition} \cite{bouyukliev2019classification}, or \texttt{LinCode} 
\cite{LC} seem to be well suited for such a task. Maybe also theoretical classification results are possible in special cases.





\begin{thebibliography}{100}

\bibitem{applegate2007exact}
D.~Applegate, W.~Cook, S.~Dash, and D.~Espinoza.
\newblock Exact solutions to linear programming problems.
\newblock {\em Operations Research Letters}, 35(6):693--699, 2007.

\bibitem{bachoc2013bounds}
C.~Bachoc, A.~Passuello, and F.~Vallentin.
\newblock Bounds for projective codes from semidefinite programming.
\newblock {\em Advances in Mathematics of Communications}, 7(2):127--145, 2013.

\bibitem{bachoc2008new}
C.~Bachoc and F.~Vallentin.
\newblock New upper bounds for kissing numbers from semidefinite programming.
\newblock {\em Journal of the American Mathematical Society}, 21(3):909--924,
  2008.

\bibitem{barg2013new}
A.~Barg and W.-H. Yu.
\newblock New bounds for spherical two-distance sets.
\newblock {\em Experimental Mathematics}, 22(2):187--194, 2013.

\bibitem{beauville1979nombre}
A.~Beauville.
\newblock Sur le nombre maximum de points doubles d’une surface dans ${P}^3$
  ($\mu (5)= 31$).
\newblock {\em Journ\'ees de G\'eom\'etrie alg\'ebraique d'Angers}, pages
  207--215, 1979.

\bibitem{betsumiya2012triply}
K.~Betsumiya and A.~Munemasa.
\newblock On triply even binary codes.
\newblock {\em Journal of the London Mathematical Society}, page jdr054, 2012.

\bibitem{beutelspacher1975partial}
A.~Beutelspacher.
\newblock Partial spreads in finite projective spaces and partial designs.
\newblock {\em Mathematische Zeitschrift}, 145(3):211--229, 1975.

\bibitem{bierbrauer2005introduction}
J.~Bierbrauer.
\newblock {\em Introduction to coding theory}.
\newblock 2005.

\bibitem{bierbrauer1997family}
J.~Bierbrauer and Y.~Edel.
\newblock A family of $2$-weight codes related to $bch$-codes.
\newblock {\em Journal of Combinatorial Designs}, 5(5):391, 1997.

\bibitem{bose1952orthogonal}
R.~Bose and K.~Bush.
\newblock Orthogonal arrays of strength two and three.
\newblock {\em The Annals of Mathematical Statistics}, pages 508--524, 1952.

\bibitem{bouyukliev2007q}
I.~Bouyukliev.
\newblock What is q-extension?
\newblock {\em Serdica Journal of Computing}, 1(2):115--130, 2007.

\bibitem{bouyukliev2019classification}
I.~Bouyukliev and S.~Bouyuklieva.
\newblock Classification of linear codes using canonical augmentation.
\newblock {\em arXiv preprint 1907.10363}, 2019.

\bibitem{brauer1942problem}
A.~Brauer.
\newblock On a problem of partitions.
\newblock {\em American Journal of Mathematics}, 64(1):299--312, 1942.

\bibitem{brouwer1998bounds}
A.~Brouwer, H.~H{\"a}m{\"a}l{\"a}inen, P.~{\"O}sterg{\aa}rd, and N.~Sloane.
\newblock Bounds on mixed binary/ternary codes.
\newblock {\em IEEE Transactions on Information Theory}, 44(1):140--161, 1998.

\bibitem{brouwer1977group}
A.~Brouwer, A.~Schrijver, and H.~Hanani.
\newblock Group divisible designs with block-size four.
\newblock {\em Discrete Mathematics}, 20:1--10, 1977.

\bibitem{ubt_eref48691}
M.~Buratti, M.~Kiermaier, S.~Kurz, A.~Naki{\'c}, and A.~Wassermann.
\newblock $q$-analogs of group divisible designs.
\newblock In {\em Combinatorics and Finite Fields : Difference Sets,
  Polynomials, Pseudorandomness and Applications}, volume~23 of {\em Radon
  Series on Computational and Applied Mathematics}. De Gruyter, Berlin, 2019.

\bibitem{calderbank1984three}
A.~Calderbank and J.~Goethals.
\newblock Three-weight codes and association schemes.
\newblock {\em Philips Journal of Research}, 39(4-5):143--152, 1984.

\bibitem{calderbank1986geometry}
R.~Calderbank and W.~Kantor.
\newblock The geometry of two-weight codes.
\newblock {\em Bulletin of the London Mathematical Society}, 18(2):97--122,
  1986.

\bibitem{catanese1981babbage}
F.~Catanese.
\newblock Babbage's conjecture, contact of surfaces, symmetric determinantal
  varieties and applications.
\newblock {\em Inventiones Mathematicae}, 63(3):433--465, 1981.

\bibitem{chen2006constructions}
E.~Chen.
\newblock Constructions of quasi-cyclic two-weigh codes.
\newblock In {\em Tenth International Workshop on Algebraic and Combinatorial
  Coding Theory (ACCT-10), Zvenigorod, Russia, September 2006}, pages 56--59,
  2006.

\bibitem{delsarte1971twoweight_srg}
P.~Delsarte.
\newblock Two-weight linear codes and strongly regular graphs, mble research
  laboratory rep.
\newblock {\em R160}, 1971.

\bibitem{delsarte1972bounds}
P.~Delsarte.
\newblock Bounds for unrestricted codes, by linear programming.
\newblock {\em Philips Research Reports}, 27:272--289, 1972.

\bibitem{delsarte1972weights}
P.~Delsarte.
\newblock Weights of linear codes and strongly regular normed spaces.
\newblock {\em Discrete Mathematics}, 3(1-3):47--64, 1972.

\bibitem{delsarte1973algebraic}
P.~Delsarte.
\newblock An algebraic approach to the association schemes of coding theory.
\newblock {\em Philips research reports supplements}, (10):103, 1973.

\bibitem{delsarte1998association}
P.~Delsarte and V.~Levenshtein.
\newblock Association schemes and coding theory.
\newblock {\em IEEE Transactions on Information Theory}, 44(6):2477--2504,
  1998.

\bibitem{ding2015class}
K.~Ding and C.~Ding.
\newblock A class of two-weight and three-weight codes and their applications
  in secret sharing.
\newblock {\em IEEE Transactions on Information Theory}, 61(11):5835--5842,
  2015.

\bibitem{dissett2000combinatorial}
L.~Dissett.
\newblock {\em Combinatorial and computational aspects of finite geometries}.
\newblock PhD thesis, National Library of Canada= Biblioth{\`e}que nationale du
  Canada, 2000.

\bibitem{dodunekov1998codes}
S.~Dodunekov and J.~Simonis.
\newblock Codes and projective multisets.
\newblock {\em The Electronic Journal of Combinatorics}, 5(R37):1--23, 1998.

\bibitem{dougherty2002macwilliams}
S.~Dougherty and M.~Skriganov.
\newblock Mac{W}illiams duality and the {R}osenbloom-{T}sfasman metric.
\newblock {\em Moscow Mathematical Journal}, 2(1):83--99, 2002.

\bibitem{nets_and_spreads}
D.~Drake and J.~Freeman.
\newblock Partial $t$-spreads and group constructible $(s,r,\mu)$-nets.
\newblock {\em Journal of Geometry}, 13(2):210--216, 1979.

\bibitem{spreadsk3}
S.~El-Zanati, H.~Jordon, G.~Seelinger, P.~Sissokho, and L.~Spence.
\newblock The maximum size of a partial $3$-spread in a finite vector space
  over ${G}{F}(2)$.
\newblock {\em Designs, Codes and Cryptography}, 54(2):101--107, 2010.

\bibitem{el2009partitions}
S.~El-Zanati, G.~Seelinger, P.~Sissokho, L.~Spence, and C.~V. Eynden.
\newblock On partitions of finite vector spaces of low dimension over $gf(2)$.
\newblock {\em Discrete Mathematics}, 309(14):4727--4735, 2009.

\bibitem{el2011lambda}
S.~El-Zanati, G.~Seelinger, P.~Sissokho, L.~Spence, and C.~V. Eynden.
\newblock On $\lambda$-fold partitions of finite vector spaces and duality.
\newblock {\em Discrete Mathematics}, 311(4):307--318, 2011.

\bibitem{espinoza2006linear}
D.~Espinoza.
\newblock {\em On linear programming, integer programming and cutting planes}.
\newblock PhD thesis, Georgia Institute of Technology, 2006.

\bibitem{etzion2014covering}
T.~Etzion.
\newblock Covering of subspaces by subspaces.
\newblock {\em Designs, Codes and Cryptography}, 72(2):405--421, 2014.

\bibitem{ubt_eref48694}
T.~Etzion, S.~Kurz, K.~Otal, and F.~{\"O}zbudak.
\newblock Subspace packings.
\newblock In {\em The Eleventh International Workshop on Coding and
  Cryptography 2019 : WCC Proceedings}. Saint-Jacut-de-la-Mer, 2019.

\bibitem{etzion2019subspace}
T.~Etzion, S.~Kurz, K.~Otal, and F.~{\"O}zbudak.
\newblock Subspace packings--constructions and bounds.
\newblock {\em arXiv preprint 1909.06081}, 2019.

\bibitem{etzionsurvey}
T.~Etzion and L.~Storme.
\newblock Galois geometries and coding theory.
\newblock {\em Designs, Codes and Cryptography}, 78(1):311--350, 2016.

\bibitem{etzion2011q}
T.~Etzion and A.~Vardy.
\newblock On $q$-analogs of {S}teiner systems and covering designs.
\newblock {\em Advances in Mathematics of Communications}, 5(2):161--176, 2011.

\bibitem{ferret2003results}
S.~Ferret and L.~Storme.
\newblock Results on maximal partial spreads in $pg (3, p^3)$ and on related
  minihypers.
\newblock {\em Designs, Codes and Cryptography}, 29(1-3):105--122, 2003.

\bibitem{govaerts2002particular}
P.~Govaerts and L.~Storme.
\newblock On a particular class of minihypers and its applications: {I}{I}.
  improvements for $q$ square.
\newblock {\em Journal of Combinatorial Theory, Series A}, 97(2):369--393,
  2002.

\bibitem{govaerts2003particular}
P.~Govaerts and L.~Storme.
\newblock On a particular class of minihypers and its applications. {I}. the
  result for general $q$.
\newblock {\em Designs, Codes and Cryptography}, 28(1):51--63, 2003.

\bibitem{greferath2000finite}
M.~Greferath and S.~Schmidt.
\newblock Finite-ring combinatorics and mac{W}illiams' equivalence theorem.
\newblock {\em Journal of Combinatorial Theory, Series A}, 92(1):17--28, 2000.

\bibitem{gulliver1996two}
T.~Gulliver.
\newblock Two new optimal ternary two-weight codes and strongly regular graphs.
\newblock {\em Discrete Mathematics}, 149(1):83--92, 1996.

\bibitem{guritman2001degree}
S.~Guritman, F.~Hoogweg, and J.~Simonis.
\newblock The degree of functions and weights in linear codes.
\newblock {\em Discrete Applied Mathematics}, 111(1-2):87--102, 2001.

\bibitem{hamada1993characterization}
N.~Hamada.
\newblock A characterization of some $[n, k, d; q]$-codes meeting the
  {G}riesmer bound using a minihyper in a finite projective geometry.
\newblock {\em Discrete Mathematics}, 116(1):229--268, 1993.

\bibitem{hedayat2012orthogonal}
A.~Hedayat, N.~Sloane, and J.~Stufken.
\newblock {\em Orthogonal arrays: theory and applications}.
\newblock Springer Science \& Business Media, 2012.

\bibitem{heden2009length}
O.~Heden.
\newblock On the length of the tail of a vector space partition.
\newblock {\em Discrete Mathematics}, 309(21):6169--6180, 2009.

\bibitem{heden2012survey}
O.~Heden.
\newblock A survey of the different types of vector space partitions.
\newblock {\em Discrete Mathematics, Algorithms and Applications}, 4(1):14p.,
  2012.
\newblock nr. 1250001.

\bibitem{heden2013supertail}
O.~Heden, J.~Lehmann, E.~N{\u{a}}stase, and P.~Sissokho.
\newblock The supertail of a subspace partition.
\newblock {\em Designs, Codes and Cryptography}, 69(3):305--316, 2013.

\bibitem{heinlein2019generalized}
D.~Heinlein, T.~Honold, M.~Kiermaier, and S.~Kurz.
\newblock Generalized vector space partitions.
\newblock {\em Australasian Journal of Combinatorics}, 73(1):162--178, 2019.

\bibitem{ubt_eref40887}
D.~Heinlein, T.~Honold, M.~Kiermaier, S.~Kurz, and A.~Wassermann.
\newblock Projective divisible binary codes.
\newblock In {\em The Tenth International Workshop on Coding and Cryptography
  2017 : WCC Proceedings}. Saint-Petersburg, September 2017.

\bibitem{heinlein2019classifying}
D.~Heinlein, T.~Honold, M.~Kiermaier, S.~Kurz, and A.~Wassermann.
\newblock Classifying optimal binary subspace codes of length $8$, constant
  dimension $4$ and minimum distance $6$.
\newblock {\em Designs, Codes and Cryptography}, 87(2-3):375--391, 2019.

\bibitem{TableSubspacecodes}
D.~Heinlein, M.~Kiermaier, S.~Kurz, and A.~Wassermann.
\newblock Tables of subspace codes.
\newblock {\em arXiv preprint 1601.02864}, 2016.

\bibitem{heng2015several}
Z.~Heng and Q.~Yue.
\newblock Several classes of cyclic codes with either optimal three weights or
  a few weights.
\newblock {\em IEEE Transactions on Information Theory}, 62(8):4501--4513,
  2016.

\bibitem{hill2007geometric}
R.~Hill and H.~Ward.
\newblock A geometric approach to classifying {G}riesmer codes.
\newblock {\em Designs, Codes and Cryptography}, 44(1-3):169--196, 2007.

\bibitem{honold2015optimal}
T.~Honold, M.~Kiermaier, and S.~Kurz.
\newblock Optimal binary subspace codes of length $6$, constant dimension $3$
  and minimum subspace distance $4$.
\newblock {\em Topics in finite fields}, 632:157--176, 2015.

\bibitem{honold2016constructions}
T.~Honold, M.~Kiermaier, and S.~Kurz.
\newblock Constructions and bounds for mixed-dimension subspace codes.
\newblock {\em Advances in Mathematics of Communications}, 10(3):649--682,
  2016.

\bibitem{honold2018partial}
T.~Honold, M.~Kiermaier, and S.~Kurz.
\newblock Partial spreads and vector space partitions.
\newblock In {\em Network Coding and Subspace Designs}, pages 131--170.
  Springer, 2018.

\bibitem{honold2019classification}
T.~Honold, M.~Kiermaier, and S.~Kurz.
\newblock Classification of large partial plane spreads in
  $\operatorname{PG}(6, 2)$ and related combinatorial objects.
\newblock {\em Journal of Geometry}, 110(5):1--31, 2019.

\bibitem{ubt_eref52236}
T.~Honold, M.~Kiermaier, and S.~Kurz.
\newblock Johnson type bounds for mixed dimension subspace codes.
\newblock {\em The Electronic Journal of Combinatorics}, 26(3), August 2019.

\bibitem{no59}
T.~Honold, M.~Kiermaier, S.~Kurz, and A.~Wassermann.
\newblock The lengths of projective triply-even binary codes.
\newblock {\em IEEE Transactions on Information Theory}, page pp.~4, to appear.
\newblock doi: 10.1109/TIT.2019.2940967.

\bibitem{huffman2010fundamentals}
W.~Huffman and V.~Pless.
\newblock {\em Fundamentals of error-correcting codes}.
\newblock Cambridge university press, 2010.

\bibitem{jaffe1996binary}
D.~Jaffe.
\newblock Binary linear codes: new results on nonexistence.
\newblock 1996.

\bibitem{jaffe1997brief}
D.~Jaffe.
\newblock A brief tour of split linear programming.
\newblock In {\em International Symposium on Applied Algebra, Algebraic
  Algorithms, and Error-Correcting Codes}, pages 164--173. Springer, 1997.

\bibitem{jaffe2000optimal}
D.~Jaffe.
\newblock Optimal binary linear codes of length $\le 30$.
\newblock {\em Discrete Mathematics}, 223(1):135--155, 2000.

\bibitem{jaffe1997sextic}
D.~Jaffe and D.~Ruberman.
\newblock A sextic surface cannot have $66$ nodes.
\newblock {\em Journal of Algebraic Geometry}, 6(1):151--168, 1997.

\bibitem{keedwell2015latin}
A.~Keedwell and J.~D{\'e}nes.
\newblock {\em Latin squares and their applications}.
\newblock Elsevier, 2015.

\bibitem{divisibleIEEE}
M.~Kiermaier and S.~Kurz.
\newblock On the lengths of divisible codes.
\newblock {\em IEEE Transactions on Information Theory}, page pp.~10, to
  appear.
\newblock arXiv preprint 1912.03892.

\bibitem{kiermaier2019three}
M.~Kiermaier, S.~Kurz, M.~Shi, and P.~Sol{\'e}.
\newblock Three-weight codes over rings and strongly walk regular graphs.
\newblock {\em arXiv preprint 1912.03892}, 2019.

\bibitem{kim2005classification}
H.~Kim and D.~Y. Oh.
\newblock A classification of posets admitting the mac{W}illiams identity.
\newblock {\em IEEE Transactions on information theory}, 51(4):1424--1431,
  2005.

\bibitem{kohnert2007constructing_twoweight}
A.~Kohnert.
\newblock Constructing two-weight codes with prescribed groups of
  automorphisms.
\newblock {\em Discrete Applied Mathematics}, 155(11):1451--1457, 2007.

\bibitem{MR2796712}
A.~Kohnert and S.~Kurz.
\newblock Construction of large constant dimension codes with a prescribed
  minimum distance.
\newblock In {\em Mathematical methods in computer science}, volume 5393 of
  {\em Lecture Notes in Computer Science}, pages 31--42. Springer, Berlin,
  2008.

\bibitem{KramerMesner:76}
E.~Kramer and D.~Mesner.
\newblock $t$-designs on hypergraphs.
\newblock {\em Discrete Mathematics}, 15:263--296, 1976.

\bibitem{krawtchouk1929generalisation}
M.~Krawtchouk.
\newblock Sur une g{\'e}n{\'e}ralisation des polynomes d’hermite.
\newblock {\em Comptes Rendus}, 189(620-622):5--3, 1929.

\bibitem{kurzspreads}
S.~Kurz.
\newblock Improved upper bounds for partial spreads.
\newblock {\em Designs, Codes and Cryptography}, 85(1):97--106, 2017.

\bibitem{kurz2017packing}
S.~Kurz.
\newblock Packing vector spaces into vector spaces.
\newblock {\em Australasian Journal of Combinatorics}, 68:122--130, 2017.

\bibitem{kurz2018heden}
S.~Kurz.
\newblock Heden’s bound on the tail of a vector space partition.
\newblock {\em Discrete Mathematics}, 341(12):3447--3452, 2018.

\bibitem{kurz201946}
S.~Kurz.
\newblock The $[46,9,20]_2$ code is unique.
\newblock {\em arXiv preprint 1906.02621}, 2019.

\bibitem{LC}
S.~Kurz.
\newblock Lincode - computer classification of linear codes.
\newblock {\em arXiv preprint 1912.09357}, 2019.

\bibitem{lam1991search}
C.~Lam.
\newblock The search for a finite projective plane of order 10.
\newblock {\em The American Mathematical Monthly}, 98(4):305--318, 1991.

\bibitem{landjev2016extendability}
I.~Landjev, A.~Rousseva, and L.~Storme.
\newblock On the extendability of quasidivisible {G}riesmer arcs.
\newblock {\em Designs, Codes and Cryptography}, 79(3):535--547, 2016.

\bibitem{lehmann2012some}
J.~Lehmann and O.~Heden.
\newblock Some necessary conditions for vector space partitions.
\newblock {\em Discrete Mathematics}, 312(2):351--361, 2012.

\bibitem{lenstra1982factoring}
H.~Lenstra, A.~Lenstra, and L.~Lov{\'a}s.
\newblock Factoring polynomials with rational coeficients.
\newblock {\em Mathematische Annalen}, 261(4):515--534, 1982.

\bibitem{liu2006divisible}
X.~Liu.
\newblock {\em On divisible codes over finite fields}.
\newblock PhD thesis, California Institute of Technology, 2006.

\bibitem{liu2006weights}
X.~Liu.
\newblock Weights modulo a prime power in divisible codes and a related bound.
\newblock {\em Information Theory, IEEE Transactions on}, 52(10):4455--4463,
  2006.

\bibitem{liu2010binary}
X.~Liu.
\newblock Binary divisible codes of maximum dimension.
\newblock {\em International Journal of Information and Coding Theory},
  1(4):355--370, 2010.

\bibitem{liu2011equivalence}
X.~Liu.
\newblock An equivalence of ward’s bound and its application.
\newblock {\em Designs, Codes and Cryptography}, 58(1):1--9, 2011.

\bibitem{lloyd1957binary}
S.~Lloyd.
\newblock Binary block coding.
\newblock {\em Bell System Technical Journal}, 36(2):517--535, 1957.

\bibitem{nakic2016extendability}
A.~Naki{\'c} and L.~Storme.
\newblock On the extendability of particular classes of constant dimension
  codes.
\newblock {\em Designs, Codes and Cryptography}, 79(3):407--422, 2016.

\bibitem{nastase2016maximumII}
E.~N{\u{a}}stase and P.~Sissokho.
\newblock The maximum size of a partial spread ii: Upper bounds.
\newblock {\em Discrete Mathematics}, 340(7):1481--1487, 2017.

\bibitem{nastase2016maximum}
E.~N{\u{a}}stase and P.~Sissokho.
\newblock The maximum size of a partial spread in a finite projective space.
\newblock {\em Journal of Combinatorial Theory, Series A}, 152:353--362, 2017.

\bibitem{o1996ovoids}
C.~O'Keefe.
\newblock Ovoids in $pg (3, q)$: a survey.
\newblock {\em Discrete Mathematics}, 151(1):175--188, 1996.

\bibitem{ostrom1968vector}
T.~Ostrom.
\newblock Vector spaces and construction of finite projective planes.
\newblock {\em Archiv der Mathematik}, 19(1):1--25, 1968.

\bibitem{perrott2016existence}
X.~Perrott.
\newblock Existence of projective planes.
\newblock {\em arXiv preprint arXiv:1603.05333}, 2016.

\bibitem{seelinger2012partitions}
G.~Seelinger, P.~Sissokho, L.~Spence, and C.~Vanden~Eynden.
\newblock Partitions of $v(n,q)$ into $2$- and $s$-dimensional subspaces.
\newblock {\em Journal of Combinatorial Designs}, 20(11):467--482, 2012.

\bibitem{shannon1948mathematical}
C.~Shannon.
\newblock A mathematical theory of communication.
\newblock {\em Bell System Technical Journal}, 27(3):379--423, 1948.

\bibitem{shiromoto1996new}
K.~Shiromoto.
\newblock A new mac{W}illiams type identity for linear codes.
\newblock {\em Hokkaido Mathematical Journal}, 25:651--656, 1996.

\bibitem{simonis1994restrictions}
J.~Simonis.
\newblock Restrictions on the weight distribution of binary linear codes
  imposed by the structure of {R}eed-{M}uller codes.
\newblock {\em IEEE Transactions on Information Theory}, 40(1):194--196, 1994.

\bibitem{simonis1995macwilliams}
J.~Simonis.
\newblock Mac{W}illiams identities and coordinate partitions.
\newblock {\em Linear Algebra and its Applications}, 216:81--91, 1995.

\bibitem{sloane1996linear}
N.~Sloane and J.~Stufken.
\newblock A linear programming bound for orthogonal arrays with mixed levels.
\newblock {\em Journal of Statistical Planning and Inference}, 56(2):295--305,
  1996.

\bibitem{tang2016macwilliams}
Y.~Tang, S.~Zhu, and X.~Kai.
\newblock Mac{W}illiams type identities on the lee and euclidean weights for
  linear codes over $\mathbb{Z}_l$.
\newblock {\em arXiv preprint 1605.08994}, 2016.

\bibitem{ward1981divisible}
H.~Ward.
\newblock Divisible codes.
\newblock {\em Archiv der Mathematik}, 36(1):485--494, 1981.

\bibitem{ward1992bound}
H.~Ward.
\newblock A bound for divisible codes.
\newblock {\em Information Theory, IEEE Transactions on}, 38(1):191--194, 1992.

\bibitem{ward1998divisibility}
H.~Ward.
\newblock Divisibility of codes meeting the {G}riesmer bound.
\newblock {\em Journal of Combinatorial Theory, Series A}, 83(1):79--93, 1998.

\bibitem{ward1998quadratic}
H.~Ward.
\newblock Quadratic residue codes and divisibility.
\newblock {\em Handbook of coding theory}, 1:827--870, 1998.

\bibitem{ward1999introduction}
H.~Ward.
\newblock An introduction to divisible codes.
\newblock {\em Designs, Codes and Cryptography}, 17(1):73--79, 1999.

\bibitem{ward2001divisible_survey}
H.~Ward.
\newblock Divisible codes -- a survey.
\newblock {\em Serdica Mathematical Journal}, 27(4):263p--278p, 2001.

\bibitem{wassermann2002attacking}
A.~Wassermann.
\newblock Attacking the market split problem with lattice point enumeration.
\newblock {\em Journal of Combinatorial Optimization}, 6(1):5--16, 2002.

\bibitem{xia2009johnson}
S.-T. Xia and F.-W. Fu.
\newblock Johnson type bounds on constant dimension codes.
\newblock {\em Designs, Codes and Cryptography}, 50(2):163--172, 2009.

\bibitem{zhou2014class}
Z.~Zhou and C.~Ding.
\newblock A class of three-weight cyclic codes.
\newblock {\em Finite Fields and Their Applications}, 25:79--93, 2014.

\bibitem{zieschang2006theory}
P.-H. Zieschang.
\newblock {\em Theory of association schemes}.
\newblock Springer Science \& Business Media, 2006.

\end{thebibliography}

\appendix
\section{Computer search for $\mathbf{q^r}$-divisible sets in $\mathbf{\aspace}$}
\label{sec_computer_search_hole_configurations}

The aim of this appendix is to state sporadic $q^r$-divisible sets that are needed for the statements in Section~\ref{sec_exclusion}. 
All of the subsequently listed examples are found using the integer linear programming approach presented at the end of 
Section~\ref{sec_constructions}. As prescribed groups we have, if any at all, used cyclic groups. So, the results of this section are 
by no means exhaustive and just complete Section~\ref{sec_exclusion}. 

\subsection{$q=2$, $r=3$}
\label{subsec_computer_search_hole_configuration_q_2_r_3}

\begin{itemize}
    \item $n=49$:  
$\begin{smallmatrix}
0000000000000000000000000111111111111111111111111\\
0000000000000111111111111000000000000111111111111\\
0000000001111000000001111000011111111000011111111\\
0000000110000000011110011001100001111111100111111\\
0000011000011001100111111000000110011001111001111\\
0001101010101010101010101010101010101010101010011\\
0110101010100010000000001010000000001000101011101\\
1010001010101010101101001010101101001100101010110\\
\end{smallmatrix}$\\[2mm]
    \item $n=50$:  
$\begin{smallmatrix}
00000000000000000000000000111111111111111111111111\\
00000000000000111111111111000000000000111111111111\\
00000000111111000000111111000000111111000000111111\\
00000011000011000011000011001111001111001111001111\\
00001100000101001111011101000011010001110011010111\\
00110000001111000101101110000101110110010100100111\\
01010101011101011011100111011101011011110101111001\\
10010101011011011101010111011011110011011101111010\\
\end{smallmatrix}$\\[2mm]
\item $n=74$: 
$\begin{smallmatrix}
00001101111111111001000010001001110111111100101001100011011000011011000111\\
11100001101111111000110110001110110001001111100101100011100100110011110000\\
11001100101011011010100010001010100111000000000010000110000011011111110001\\
00001110110001000010010100100111100011101010100110010000000011110001011101\\
01000010111000110011111010110101010111110000101100000001011011111111110001\\
00100111100011111010011100011110010000000011111101001110111110001000110111\\
11001011110011111011010111100110101001001001001111001110000010101001101010\\
01111000111001010010110000011011111111000101011110010000011101101101011100\\
11110110101010111011010111011011110000000110001110010111000101001010011010\\
10101001110000001100111011011110011111000101010010110110011001011101011100\\
00100001111110101101001110111101000111010011011110101000001001011000110111\\
11111110000110110001010001011111000100110110110100011110001110011010011010\\
\end{smallmatrix}$\\[1mm]
The weight distribution is given by $0^1 8^3 24^{60} 32^{1423} 40^{2585} 48^{24}$. 
\end{itemize}

\subsection{$q=2$, $r=4$}
\label{subsec_computer_search_hole_configuration_q_2_r_4}

\begin{itemize}
\item $n=161$:\\[1mm] 
$\begin{smallmatrix}
00000000000000000000000000000000000000000000000000000000000000000000000000000000011111111111111111111111111111111111111111111111111111111111111111111111111111111\\
00000000000000000000000000000000000000000111111111111111111111111111111111111111100000000000000000000000000000000000000001111111111111111111111111111111111111111\\
00000000000000000000000001111111111111111000000000000000000000000111111111111111100000000000000001111111111111111111111110000000000000000111111111111111111111111\\
00000000000000011111111110000000000111111000000000011111111111111000000111111111100000000001111110000000000000011111111110000001111111111000000000011111111111111\\
00000000011111100000011110000111111001111000000111100000000111111001111000011111100000011110000110000001111111100001111110000110000001111000011111100000011111111\\
00000111100011101111100110001000011110001000011000100011111000011010011111100011100001100010001110001110000111100110111110011010001111111000100001100001100011111\\
00011011101101100001111010011001101011111000100001101100011001101000111000100100100110011100010000010000111000101011000110101010010111111011100111100010100101111\\
00101001100010100110101110101110101010010001011010010001101010111100111001101000101111101110110110000010001011000110001011010000001000001001101011101111100110111\\
01000010100101101010101000010001011010010011101010100100100001100000000010111100101110010111011010110101011001110111010011111110000110011110000100100111101111011\\
10001001101100000101101000111010000000110011010111010010101101110110101001011101100010100110100011100100010011000000111011011010111010101111100010111000100111101\\
\end{smallmatrix}$\\[1mm]
The weight distribution is given by
$0^1 64^{50} 80^{886} 96^{87}$.
\item $n=162$:\\[1mm] 
$\begin{smallmatrix}
000000000000000000000000000000000000000000000000000000000000000000000000000000000011111111111111111111111111111111111111111111111111111111111111111111111111111111\\
000000000000000000000000000000000000000000111111111111111111111111111111111111111100000000000000000000000000000000000000001111111111111111111111111111111111111111\\
000000000000000000000011111111111111111111000000000000000000001111111111111111111100000000000000000000111111111111111111110000000000000000000011111111111111111111\\
000000000011111111111100000000111111111111000000001111111111110000000011111111111100000000000011111111000000000000111111110000000000001111111100000000000011111111\\
000000111100000011111100001111000000111111000000110000000011110000001100000000111100001111111100111111000011111111001111110000001111110000111100000011111100001111\\
000011001100011100011100000000000001000001000111010000111100110001110100001111001100110000111101000111001100001111010001110111110111111111111100011100011100110011\\
000100010100000000100000110001000111000110001001100011000101010111111100110011000101110011001100000001010101110011100110111001110001110111001111101111111101011101\\
001001000001100101100101010110011111011010011011000101011110111010111111011111011100010000010000001010001000010101110010011010010000011001010101100101100111110110\\
010000011000101011101101111010101011101111001000000010001011011100010001010101111001001100010111011100100001111011111110110000100010100011010000100010101110101101\\
100001001101111010110101000100110010001101010100111000011111110000101001001001101100100101110110101111000000101101001011100100111011001101110101001000101000111010\\
\end{smallmatrix}$\\[1mm]
The weight distribution is given by $0^1 32^1 64^{30} 80^{890} 96^{102}$.
\item $n=195$:\\[1mm] 
$\begin{smallmatrix}
000000000000000000000000000000000000000000000000000000000000000000000000000000000000000000000000000111111111111111111111111111111111111111111111111111111111111111111111111111111111111111111111111\\
000000000000000000000000000000000000000000000000000111111111111111111111111111111111111111111111111000000000000000000000000000000000000000000000000111111111111111111111111111111111111111111111111\\
000000000000000000000000000111111111111111111111111000000000000000000000000111111111111111111111111000000000000000000000000111111111111111111111111000000000000000000000000111111111111111111111111\\
000000000000000111111111111000000000000111111111111000000000000111111111111000000000000111111111111000000000000111111111111000000000000111111111111000000000000111111111111000000000000111111111111\\
000000000111111000000111111000000111111000000111111000000111111000000111111000000111111000000111111000000111111000000111111000000111111000000111111000000111111000000111111000000111111000000111111\\
000000011000111000111000011000011000111000111000011000011000111000111000011000011000111000111000011001111000111000111001111001111000111000111001111001111000111000111001111001111000111000111001111\\
000001101011000001011001100000101011001000011001101011101001011011011000111011100011011011001001111000011001011001001010011010011001001011001010001010011011011011001110011010111001011011011010011\\
000110100001001000001000101011011101001011101110111001110011101001101011000000101001101111001010101001100011100011110010101110001011011100011000011000011000001101000010101111011110101001101110101\\
011010000001000011101011011001110011010001010110101010011010110101101100100000100100011101110001111000101000101101010111011100111010100001011100110010100111001100011010011010001010100111011101100\\
101000000011011011010000111010101011100001001011011000101110101101010001101001011101001100010110110010100110011010110101001010110110010000101110011010001001111011010011010101001001101001010111010\\
\end{smallmatrix}$\\[1mm]
The weight distribution is given by $0^1 80^{33} 96^{855} 112^{135}$.
\item $n=197$:\\[1mm] 
$\begin{smallmatrix}
00000000000000000000000000000000000000000000000000000000000000000000000000000000000000000000000000000111111111111111111111111111111111111111111111111111111111111111111111111111111111111111111111111\\
00000000000000000000000000000000000000000000000000000111111111111111111111111111111111111111111111111000000000000000000000000000000000000000000000000111111111111111111111111111111111111111111111111\\
00000000000000000000000000000111111111111111111111111000000000000000000000000111111111111111111111111000000000000000000000000111111111111111111111111000000000000000000000000111111111111111111111111\\
00000000000000000111111111111000000000000111111111111000000000000111111111111000000000000111111111111000000000000111111111111000000000000111111111111000000000000111111111111000000000000111111111111\\
00000000001111111000000011111000000011111000001111111000000011111000001111111000001111111000000011111000000011111000001111111000001111111000000011111000001111111000000011111000000011111000001111111\\
00000001110000111000001100011000011100111000110000011000001100011001110000111000110000011000011100111000111100011001110011111000110001111001111100111001110011111000111100011001111100111000110001111\\
00000110110001001000110100100000100101001001000011111001110100101010110111011011110011101001100101011001000100101010110100011001010110011010001100001010011101111001001101101010011101001111110010011\\
00011010010011000000011000011011000111111011010100101010010101000101010011001001010101110010101111101001001001010100011000111011100010101110110101011000101110011001010110000101111100011011001110001\\
01101000000010001011011101100001001001011000110101111010101010010110101001011000110110100101100110101001011000101001111001101101001100100110011010110000010100111010010111110011100100100111010101100\\
10100000010111010001101000101000110010011001101100111010100100110101101010110001011010111011000011010010010010110000110101111110011011001010110001100010101101101100100010101000101100110101101100110\\
\end{smallmatrix}$\\[1mm]
The weight distribution is given by
$0^1 80^{10} 96^{837} 112^{176}$.
\end{itemize}

\subsection{$q=5$, $r=1$}
\label{subsec_computer_search_hole_configuration_q_5_r_1}

\begin{itemize}
\item $n=41$: 
$\begin{smallmatrix}
00000000000111111111111111111111111111111\\
00011111111000000001111111122233344444444\\
00100112344012223330011122403303300111224\\
01003013314040131332201203012211324033021\\
10003222331041124030121012410144041303440\\
\end{smallmatrix}$\\[1mm]
The weight distribution of this $[41,5,25]_5$-code is given by
$0^1 25^{4} 30^{1360} 35^{1760}$.
\item $n=46$: 
$\begin{smallmatrix}
0000000000000000111111111111111111111111111111\\
0000001111111111000000000011111222223333344444\\
0011110000233344000012224401233000441233300114\\
0100330024413301012430340101023024030201413042\\
1004130430300134031034030421224213300100140402\\
\end{smallmatrix}$\\[1mm]
The weight distribution of this $[46,5,25]_5$-code is given by
$0^1 25^4 30^{60} 35^{1860} 40^{1200}$.
\end{itemize}

\subsection{$q=7$, $r=1$}
\label{subsec_computer_search_hole_configuration_q_7_r_1}

\begin{itemize}
\item $n=141$:\\[1mm] 
$\begin{smallmatrix}
000000000000000000000011111111111111111111111111111111111111111111111111111111111111111111111111111111111111111111111111111111111111111111111\\
000011111111111111111100000000000000000011111111111111111122222222222222222233333333333333333344444444444444444455555555555666666666666666666\\
011100000112222344466600011223334444566600111223334444566600022233344556666600122333444555566600111222344555566601112234566000111122335556666\\
102601356240256412502604604260561234001546456032450126115601202614556020345624314234013124524512016012525035613561562424324045145606261350136\\
\end{smallmatrix}$\\[1mm]
The weight distribution of this $[141,4,112]_7$-code is given by
$0^1 112^{30} 119^{1692} 126^{672} 133^{6}$.
\end{itemize}

\section{Excluded cardinalities for $q^r$-divisible sets}
\label{app_exclusion_lists}

We have implemented the (integer) linear programming method described in Subsection~\ref{subsec_linear_programming_method}. In this 
appendix we list the corresponding results, i.e., for moderate size parameters we state ranges of excluded cardinalities for $q^r$-divisible sets. 
As a shorthand we use the notation $[a,b]$ for the set of integers $\{a,a+1,\dots,b\}$.

\subsection{$q=2$}

\medskip

Exclusion list for $q=2$ and $r=1$: $[1,2]$.

\medskip

Exclusion list for $q=2$ and $r=2$: $[1,6]$, $[9,13]$.

\medskip

Exclusion list for $q=2$ and $r=3$:
$[1,14]$,
$[17,29]$,
$[33,44]$, 
$[52,58]$.

\medskip

Exclusion list for $q=2$ and $r=4$:
$[1,30]$, 
$[33,61]$, 
$[65,92]$, 
$[97,123]$, 
$[132,154]$, 
$[166,184]$, 
$[200,214]$, 
$[235,243]$.

\medskip

Exclusion list for $q=2$ and $r=5$:
$[1,62]$, 
$[65,125]$, 
$[129,188]$, 
$[193,251]$, 
$[257,314]$, 
$[324,377]$, 
$[390,440]$, 
$[456,502]$, 
$[521,565]$, 
$[587,627]$, 
$[652,690]$, 
$[718,752]$, 
$[784,814]$, 
$[850,876]$, 
$[917,937]$, 
$[985,997]$.

\medskip

Exclusion list for $q=2$ and $r=6$:
$[1,126]$, 
$[129,253]$, 
$[257,380]$, 
$[385,507]$, 
$[513,634]$, 
$[641,761]$, 
$[772,888]$, 
$[902,1015]$, 
$[1032,1142]$, 
$[1161,1269]$, 
$[1291,1395]$, 
$[1420,1522]$, 
$[1549,1649]$, 
$[1678,1776]$, 
$[1808,1902]$, 
$[1937,2029]$, 
$[2066,2156]$, 
$[2196,2282]$, 
$[2325,2409]$, 
$[2455,2535]$, 
$[2585,2661]$, 
$[2714,2788]$, 
$[2844,2914]$, 
$[2974,3040]$, 
$[3104,3166]$, 
$[3234,3292]$, 
$[3364,3418]$, 
$[3495,3543]$, 
$[3626,3668]$, 
$[3757,3793]$, 
$[3889,3917]$, 
$[4023,4039]$.

\medskip

Exclusion list for $q=2$ and $r=7$:
$[1,254]$, 
$[257,509]$, 
$[513,764]$, 
$[769,1019]$, 
$[1025,1274]$, 
$[1281,1529]$, 
$[1537,1784]$, 
$[1796,2039]$, 
$[2054,2294]$, 
$[2312,2549]$, 
$[2569,2804]$, 
$[2827,3059]$, 
$[3084,3314]$, 
$[3341,3569]$, 
$[3598,3824]$, 
$[3856,4078]$, 
$[4113,4333]$, 
$[4370,4588]$, 
$[4627,4843]$, 
$[4884,5098]$, 
$[5141,5353]$, 
$[5399,5607]$, 
$[5656,5862]$, 
$[5913,6117]$, 
$[6170,6372]$, 
$[6428,6626]$, 
$[6685,6881]$, 
$[6942,7136]$, 
$[7200,7390]$, 
$[7457,7645]$, 
$[7714,7900]$, 
$[7972,8154]$, 
$[8229,8409]$, 
$[8487,8663]$, 
$[8744,8918]$, 
$[9002,9172]$, 
$[9259,9427]$, 
$[9517,9681]$, 
$[9774,9936]$, 
$[10032,10190]$, 
$[10289,10445]$, 
$[10547,10699]$, 
$[10805,10953]$, 
$[11063,11207]$, 
$[11320,11462]$, 
$[11578,11716]$, 
$[11836,11970]$, 
$[12094,12224]$, 
$[12352,12478]$, 
$[12610,12732]$, 
$[12868,12986]$, 
$[13126,13240]$, 
$[13385,13493]$, 
$[13643,13747]$, 
$[13902,14000]$, 
$[14160,14254]$, 
$[14419,14507]$, 
$[14678,14760]$, 
$[14937,15013]$, 
$[15197,15265]$, 
$[15457,15517]$, 
$[15718,15768]$, 
$[15979,16019]$, 
$[16244,16266]$.

\subsection{$q=3$}

\medskip

Exclusion list for $q=3$ and $r=1$: 
$[1,3]$, 
$[5,7]$.

\medskip

Exclusion list for $q=3$ and $r=2$: 
$[1,12]$, 
$[14,25]$, 
$[28,38]$, 
$[41,51]$, 
$[57,64]$, 
$[71,76]$, 
$[86,89]$.

\medskip

Exclusion list for $q=3$ and $r=3$: 
$[1,39]$, 
$[41,79]$, 
$[82,119]$, 
$[122,159]$, 
$[163,199]$, 
$[203,239]$, 
$[246,279]$, 
$[287,319]$, 
$[329,359]$, 
$[370,399]$, 
$[411,439]$, 
$[452,478]$, 
$[493,518]$, 
$[535,558]$, 
$[576,597]$, 
$[618,637]$, 
$[659,676]$, 
$[701,715]$, 
$[743,754]$, 
$[786,793]$.

\medskip

Exclusion list for $q=3$ and $r=4$: 
$[1,120]$, 
$[122,241]$, 
$[244,362]$, 
$[365,483]$, 
$[487,604]$, 
$[608,725]$, 
$[730,846]$, 
$[851,967]$, 
$[975,1088]$, 
$[1097,1209]$, 
$[1220,1330]$, 
$[1342,1451]$, 
$[1464,1572]$, 
$[1586,1693]$, 
$[1708,1814]$, 
$[1831,1935]$, 
$[1953,2056]$, 
$[2075,2177]$, 
$[2197,2298]$, 
$[2319,2419]$, 
$[2441,2539]$, 
$[2563,2660]$, 
$[2685,2781]$, 
$[2807,2902]$, 
$[2930,3023]$, 
$[3052,3144]$, 
$[3174,3265]$, 
$[3296,3385]$, 
$[3418,3506]$, 
$[3540,3627]$, 
$[3663,3748]$, 
$[3785,3869]$, 
$[3907,3989]$, 
$[4029,4110]$, 
$[4152,4231]$, 
$[4274,4352]$, 
$[4396,4472]$, 
$[4518,4593]$, 
$[4641,4714]$, 
$[4763,4834]$, 
$[4885,4955]$, 
$[5008,5076]$, 
$[5130,5196]$, 
$[5253,5317]$, 
$[5375,5437]$, 
$[5498,5558]$, 
$[5620,5678]$, 
$[5743,5799]$, 
$[5865,5919]$, 
$[5988,6040]$, 
$[6111,6160]$, 
$[6233,6280]$, 
$[6356,6400]$, 
$[6479,6520]$, 
$[6602,6640]$, 
$[6725,6760]$, 
$[6848,6880]$, 
$[6972,7000]$, 
$[7096,7119]$, 
$[7220,7237]$, 
$[7347,7354]$.

\subsection{$q=4$}

\medskip

Exclusion list for $q=4$ and $r=1$: 
$[1,4]$, 
$[6,9]$, 
$[11,14]$, 
$[18,19]$.

\medskip

Exclusion list for $q=4$ and $r=2$: 
$[1,20]$, 
$[22,41]$, 
$[43,62]$, 
$[65,83]$, 
$[86,104]$, 
$[107,125]$, 
$[130,146]$, 
$[152,167]$, 
$[174,188]$, 
$[196,209]$, 
$[218,230]$, 
$[240,250]$, 
$[262,271]$, 
$[284,292]$, 
$[306,312]$, 
$[329,332]$.

\medskip

Exclusion list for $q=4$ and $r=3$: 
$[1,84]$, 
$[86,169]$, 
$[171,254]$, 
$[257,339]$, 
$[342,424]$, 
$[427,509]$, 
$[513,594]$, 
$[598,679]$, 
$[683,764]$, 
$[770,849]$, 
$[856,934]$, 
$[942,1019]$, 
$[1028,1104]$, 
$[1114,1189]$, 
$[1200,1274]$, 
$[1286,1359]$, 
$[1371,1444]$, 
$[1457,1529]$, 
$[1543,1614]$, 
$[1628,1699]$, 
$[1714,1784]$, 
$[1800,1869]$, 
$[1886,1954]$, 
$[1971,2039]$, 
$[2057,2124]$, 
$[2143,2208]$, 
$[2229,2293]$, 
$[2314,2378]$, 
$[2400,2463]$, 
$[2486,2548]$, 
$[2572,2633]$, 
$[2658,2718]$, 
$[2743,2803]$, 
$[2829,2887]$, 
$[2915,2972]$, 
$[3001,3057]$, 
$[3087,3142]$, 
$[3173,3227]$, 
$[3258,3312]$, 
$[3344,3396]$, 
$[3430,3481]$, 
$[3516,3566]$, 
$[3602,3651]$, 
$[3688,3735]$, 
$[3774,3820]$, 
$[3860,3905]$, 
$[3946,3990]$, 
$[4032,4074]$, 
$[4118,4159]$, 
$[4204,4244]$, 
$[4290,4328]$, 
$[4376,4413]$, 
$[4462,4497]$, 
$[4548,4582]$, 
$[4634,4666]$, 
$[4720,4751]$, 
$[4807,4835]$, 
$[4893,4920]$, 
$[4979,5004]$, 
$[5066,5088]$, 
$[5153,5172]$, 
$[5240,5256]$, 
$[5327,5339]$, 
$[5415,5422]$.

\subsection{$q=5$}

\medskip

Exclusion list for $q=5$ and $r=1$: 
$[1,5]$, 
$[7,11]$, 
$[13,17]$, 
$[19,23]$, 
$[27,29]$, 
$[33,35]$.

\medskip

Exclusion list for $q=5$ and $r=2$: 
$[1,30]$, 
$[32,61]$, 
$[63,92]$, 
$[94,123]$, 
$[126,154]$, 
$[157,185]$, 
$[188,216]$, 
$[219,247]$, 
$[252,278]$, 
$[283,309]$, 
$[316,340]$, 
$[347,371]$, 
$[379,402]$, 
$[410,433]$, 
$[442,464]$, 
$[473,495]$, 
$[505,526]$, 
$[537,557]$, 
$[568,587]$, 
$[600,618]$, 
$[632,649]$, 
$[663,680]$, 
$[695,711]$, 
$[727,742]$, 
$[758,772]$, 
$[790,803]$, 
$[822,834]$, 
$[854,864]$, 
$[886,895]$, 
$[918,925]$, 
$[951,955]$.

\medskip

Exclusion list for $q=5$ and $r=3$: 
$[1,155]$, 
$[157,311]$, 
$[313,467]$, 
$[469,623]$, 
$[626,779]$, 
$[782,935]$, 
$[938,1091]$, 
$[1094,1247]$, 
$[1251,1403]$, 
$[1407,1559]$, 
$[1563,1715]$, 
$[1719,1871]$, 
$[1877,2027]$, 
$[2033,2183]$, 
$[2191,2339]$, 
$[2347,2495]$, 
$[2504,2651]$, 
$[2660,2807]$, 
$[2817,2963]$, 
$[2973,3119]$, 
$[3130,3275]$, 
$[3287,3431]$, 
$[3443,3587]$, 
$[3600,3743]$, 
$[3757,3899]$, 
$[3913,4055]$, 
$[4070,4211]$, 
$[4226,4367]$, 
$[4383,4523]$, 
$[4539,4679]$, 
$[4696,4835]$, 
$[4852,4991]$, 
$[5009,5147]$, 
$[5165,5303]$, 
$[5322,5459]$, 
$[5478,5615]$, 
$[5635,5771]$, 
$[5791,5927]$, 
$[5948,6083]$, 
$[6104,6239]$, 
$[6261,6395]$, 
$[6418,6551]$, 
$[6574,6707]$, 
$[6731,6863]$, 
$[6887,7019]$, 
$[7044,7175]$, 
$[7200,7330]$, 
$[7357,7486]$, 
$[7513,7642]$, 
$[7670,7798]$, 
$[7826,7954]$, 
$[7983,8110]$, 
$[8140,8266]$, 
$[8296,8422]$, 
$[8453,8578]$, 
$[8609,8734]$, 
$[8766,8890]$, 
$[8922,9046]$, 
$[9079,9202]$, 
$[9236,9358]$, 
$[9392,9514]$, 
$[9549,9670]$, 
$[9705,9826]$, 
$[9862,9981]$, 
\dots

\subsection{$q=7$}

\medskip

Exclusion list for $q=7$ and $r=1$: 
$[1,7]$, 
$[9,15]$, 
$[17,23]$, 
$[25,31]$, 
$[33,39]$, 
$[41,47]$, 
$[51,55]$, 
$[59,63]$, 
$[67,71]$, 
$[76,79]$, 
$[84,87]$, 
$[93,94]$.

\medskip

Exclusion list for $q=7$ and $r=2$: 
$[1,56]$, 
$[58,113]$, 
$[115,170]$, 
$[172,227]$, 
$[229,284]$, 
$[286,341]$, 
$[344,398]$, 
$[401,455]$, 
$[458,512]$, 
$[515,569]$, 
$[572,626]$, 
$[629,683]$, 
$[688,740]$, 
$[745,797]$, 
$[802,854]$, 
$[860,911]$, 
$[917,968]$, 
$[975,1025]$, 
$[1033,1082]$, 
$[1090,1139]$, 
$[1147,1196]$, 
$[1205,1253]$, 
$[1262,1310]$, 
$[1319,1367]$, 
$[1377,1424]$, 
$[1434,1481]$, 
$[1491,1538]$, 
$[1549,1595]$, 
$[1606,1652]$, 
$[1664,1709]$, 
$[1721,1766]$, 
$[1778,1823]$, 
$[1836,1880]$, 
$[1893,1937]$, 
$[1950,1994]$, 
$[2008,2051]$, 
$[2065,2108]$, 
$[2123,2165]$, 
$[2180,2222]$, 
$[2237,2278]$, 
$[2295,2335]$, 
$[2352,2392]$, 
$[2410,2449]$, 
$[2467,2506]$, 
$[2524,2563]$, 
$[2582,2620]$, 
$[2639,2677]$, 
$[2697,2734]$, 
$[2754,2791]$, 
$[2811,2848]$, 
$[2869,2905]$, 
$[2926,2962]$, 
$[2984,3018]$, 
$[3041,3075]$, 
$[3099,3132]$, 
$[3156,3189]$, 
$[3213,3246]$, 
$[3271,3303]$, 
$[3328,3360]$, 
$[3386,3417]$, 
$[3443,3474]$, 
$[3501,3530]$, 
$[3558,3587]$, 
$[3616,3644]$, 
$[3673,3701]$, 
$[3731,3758]$, 
$[3788,3815]$, 
$[3846,3871]$, 
$[3903,3928]$, 
$[3961,3985]$, 
$[4018,4042]$, 
$[4076,4098]$, 
$[4134,4155]$, 
$[4191,4212]$, 
$[4249,4269]$, 
$[4306,4325]$, 
$[4364,4382]$, 
$[4422,4439]$, 
$[4480,4495]$, 
$[4537,4552]$, 
$[4595,4608]$, 
$[4653,4665]$, 
$[4711,4721]$, 
$[4769,4777]$, 
$[4827,4833]$, 
$[4886,4889]$.

\subsection{$q=8$}

\medskip

Exclusion list for $q=8$ and $r=1$: 
$[1,8]$, 
$[10,17]$, 
$[19,26]$, 
$[28,35]$, 
$[37,44]$, 
$[46,53]$, 
$[55,62]$, 
$[66,71]$, 
$[75,80]$, 
$[84,89]$, 
$[94,98]$, 
$[103,107]$, 
$[112,116]$, 
$[122,125]$, 
$[131,133]$, 
$[141,142]$.

\medskip

Exclusion list for $q=8$ and $r=2$: 
$[1,72]$, 
$[74,145]$, 
$[147,218]$, 
$[220,291]$, 
$[293,364]$, 
$[366,437]$, 
$[439,510]$, 
$[513,583]$, 
$[586,656]$, 
$[659,729]$, 
$[732,802]$, 
$[805,875]$, 
$[878,948]$, 
$[951,1021]$, 
$[1026,1094]$, 
$[1099,1167]$, 
$[1172,1240]$, 
$[1246,1313]$, 
$[1319,1386]$, 
$[1392,1459]$, 
$[1466,1532]$, 
$[1539,1605]$, 
$[1613,1678]$, 
$[1686,1751]$, 
$[1759,1824]$, 
$[1833,1897]$, 
$[1906,1970]$, 
$[1979,2043]$, 
$[2053,2116]$, 
$[2126,2189]$, 
$[2199,2262]$, 
$[2273,2335]$, 
$[2346,2408]$, 
$[2419,2481]$, 
$[2492,2554]$, 
$[2566,2627]$, 
$[2639,2700]$, 
$[2712,2773]$, 
$[2786,2846]$, 
$[2859,2919]$, 
$[2932,2992]$, 
$[3006,3065]$, 
$[3079,3138]$, 
$[3152,3211]$, 
$[3226,3284]$, 
$[3299,3357]$, 
$[3372,3430]$, 
$[3446,3503]$, 
$[3519,3576]$, 
$[3592,3649]$, 
$[3666,3722]$, 
$[3739,3795]$, 
$[3812,3868]$, 
$[3886,3940]$, 
$[3959,4013]$, 
$[4032,4086]$, 
$[4106,4159]$, 
$[4179,4232]$, 
$[4252,4305]$, 
$[4326,4378]$, 
$[4399,4451]$, 
$[4472,4524]$, 
$[4546,4597]$, 
$[4619,4670]$, 
$[4692,4743]$, 
$[4766,4816]$, 
$[4839,4889]$, 
$[4912,4962]$, 
$[4986,5035]$, 
$[5059,5108]$, 
$[5132,5181]$, 
$[5206,5253]$, 
$[5279,5326]$, 
$[5352,5399]$, 
$[5426,5472]$, 
$[5499,5545]$, 
$[5573,5618]$, 
$[5646,5691]$, 
$[5719,5764]$, 
$[5793,5837]$, 
$[5866,5910]$, 
$[5939,5983]$, 
$[6013,6056]$, 
$[6086,6128]$, 
$[6160,6201]$, 
$[6233,6274]$, 
$[6306,6347]$, 
$[6380,6420]$, 
$[6453,6493]$, 
$[6527,6566]$, 
$[6600,6639]$, 
$[6673,6712]$, 
$[6747,6784]$, 
$[6820,6857]$, 
$[6894,6930]$, 
$[6967,7003]$, 
$[7040,7076]$, 
$[7114,7149]$, 
$[7187,7222]$, 
$[7261,7294]$, 
$[7334,7367]$, 
$[7408,7440]$, 
$[7481,7513]$, 
$[7555,7586]$, 
$[7628,7659]$, 
$[7702,7731]$, 
$[7775,7804]$, 
$[7849,7877]$, 
$[7922,7950]$, 
$[7996,8022]$, 
$[8069,8095]$, 
$[8143,8168]$, 
$[8216,8241]$, 
$[8290,8313]$, 
$[8363,8386]$, 
$[8437,8459]$, 
$[8510,8532]$, 
$[8584,8604]$, 
$[8658,8677]$, 
$[8731,8749]$, 
$[8805,8822]$, 
$[8879,8895]$, 
$[8953,8967]$, 
$[9027,9039]$, 
$[9101,9112]$, 
$[9175,9184]$, 
$[9249,9256]$, 
$[9324,9327]$.

\subsection{$r=1$ and $9\le q\le 16$}

\medskip

Exclusion list for $q=9$ and $r=1$: 
$[1,9]$, 
$[11,19]$, 
$[21,29]$, 
$[31,39]$, 
$[41,49]$, 
$[51,59]$, 
$[61,69]$, 
$[71,79]$, 
$[83,89]$, 
$[93,99]$, 
$[103,109]$, 
$[113,119]$, 
$[124,129]$, 
$[134,139]$, 
$[144,149]$, 
$[155,159]$, 
$[165,169]$, 
$[176,178]$, 
$[186,188]$, 
$[197,198]$.

\medskip

Exclusion list for $q=11$ and $r=1$: 
$[1,11]$, 
$[13,23]$, 
$[25,35]$, 
$[37,47]$, 
$[49,59]$, 
$[61,71]$, 
$[73,83]$, 
$[85,95]$, 
$[97,107]$, 
$[109,119]$, 
$[123,131]$, 
$[135,143]$, 
$[147,155]$, 
$[159,167]$, 
$[171,179]$, 
$[184,191]$, 
$[196,203]$, 
$[208,215]$, 
$[220,227]$, 
$[233,239]$, 
$[245,251]$, 
$[257,263]$, 
$[270,275]$, 
$[282,287]$, 
$[294,299]$, 
$[306,310]$, 
$[319,322]$, 
$[331,334]$, 
$[344,346]$, 
$[356,357]$.

\medskip

Exclusion list for $q=13$ and $r=1$: 
$[1,13]$, 
$[15,27]$, 
$[29,41]$, 
$[43,55]$, 
$[57,69]$, 
$[71,83]$, 
$[85,97]$, 
$[99,111]$, 
$[113,125]$, 
$[127,139]$, 
$[141,153]$, 
$[155,167]$, 
$[171,181]$, 
$[185,195]$, 
$[199,209]$, 
$[213,223]$, 
$[227,237]$, 
$[241,251]$, 
$[256,265]$, 
$[270,279]$, 
$[284,293]$, 
$[298,307]$, 
$[312,321]$, 
$[327,335]$, 
$[341,349]$, 
$[355,363]$, 
$[369,377]$, 
$[383,391]$, 
$[398,405]$, 
$[412,419]$, 
$[426,433]$, 
$[440,447]$, 
$[455,461]$, 
$[469,474]$, 
$[483,488]$, 
$[498,502]$, 
$[512,516]$, 
$[526,530]$, 
$[540,544]$, 
$[555,558]$, 
$[569,571]$, 
$[584,585]$.

\medskip

Exclusion list for $q=16$ and $r=1$: 
$[1,16]$, 
$[18,33]$, 
$[35,50]$, 
$[52,67]$, 
$[69,84]$, 
$[86,101]$, 
$[103,118]$, 
$[120,135]$, 
$[137,152]$, 
$[154,169]$, 
$[171,186]$, 
$[188,203]$, 
$[205,220]$, 
$[222,237]$, 
$[239,254]$, 
$[258,271]$, 
$[275,288]$, 
$[292,305]$, 
$[309,322]$, 
$[326,339]$, 
$[343,356]$, 
$[360,373]$, 
$[378,390]$, 
$[395,407]$, 
$[412,424]$, 
$[429,441]$, 
$[446,458]$, 
$[463,475]$, 
$[480,492]$, 
$[498,509]$, 
$[515,526]$, 
$[532,543]$, 
$[549,560]$, 
$[566,577]$, 
$[583,594]$, 
$[601,611]$, 
$[618,628]$, 
$[635,645]$, 
$[652,662]$, 
$[669,679]$, 
$[686,696]$, 
$[704,713]$, 
$[721,730]$, 
$[738,747]$, 
$[755,764]$, 
$[772,781]$, 
$[790,798]$, 
$[807,814]$, 
$[824,831]$, 
$[841,848]$, 
$[858,865]$, 
$[876,882]$, 
$[893,899]$, 
$[910,916]$, 
$[927,933]$, 
$[944,950]$, 
$[962,967]$, 
$[979,984]$, 
$[996,1000]$, 
$[1014,1017]$, 
$[1031,1034]$, 
$[1048,1051]$, 
$[1066,1067]$, 
$[1083,1084]$.

\bigskip

\end{document}